\documentclass[12pt]{amsart}
\usepackage{amsfonts,amsthm,amsmath,amscd,amssymb}
\usepackage{t1enc}
\usepackage[mathscr]{eucal}
\usepackage{parskip}
\usepackage{xcolor}
\usepackage{graphics}
\usepackage{tikz}
\usetikzlibrary{calc}
\usepackage{tikz-cd}
\usepackage{spectralsequences}
\sseqset{classes={draw=none}}
\usepackage{hyperref}
\usepackage{subcaption}

\DeclareFontFamily{U}{tipa}{}
\DeclareFontShape{U}{tipa}{m}{sl}{
  <-8.5> tipasl8
  <8.5-9.5> tipasl9
  <9.5-11> tipasl10
  <11-> tipasl12
}{}
\DeclareSymbolFont{tipa}{U}{tipa}{m}{sl}
\DeclareMathSymbol{\mathglotstop}{\mathord}{tipa}{80}
\newcommand{\udq}{\mathbin{\rotatebox[origin=c]{180}{${\mathglotstop}$}}}
\DeclareMathSymbol{\hw}{\mathord}{tipa}{255}

\numberwithin{equation}{section}
\usepackage[margin=2.9cm]{geometry}
\usepackage{booktabs}
\usepackage{diagbox}
\usepackage{enumerate}
\usepackage{bm}
\usepackage{array}
\newcolumntype{C}[1]{>{\centering\arraybackslash}p{#1}}

\theoremstyle{plain}
\newtheorem{thm}{Theorem}[section]
\newtheorem{lem}[thm]{Lemma}
\newtheorem{cor}[thm]{Corollary}
\newtheorem{prop}[thm]{Proposition}

 \theoremstyle{definition}
\newtheorem{defn}[thm]{Definition}
\newtheorem{rem}[thm]{Remark}
\newtheorem{ex}[thm]{Example}
\newtheorem{notn}[thm]{Notation}
\newtheorem{setup}[thm]{Setup}
\newtheorem{prob}[thm]{Problem}
\newtheorem{ques}[thm]{Question}

\newcommand{\N}{\mathbb{N}}

\newcommand{\C}{\mathbb{C}}
\newcommand{\A}{\mathcal{A}}
\newcommand{\Part}{\mathcal{P}}

\newcommand{\Sph}{\mathbb{S}}

\newcommand{\mb}[1]{\mathbb{#1}}

\newcommand{\mr}[1]{\mathrm{#1}}
\newcommand{\eps}{\varepsilon}
\newcommand{\vphi}{\varphi}
\newcommand{\rank}{\operatorname{rank}}

\newcommand{\Spin}{\mathrm{Spin}}
\newcommand{\Spinc}{\mathrm{Spin}^c}
\newcommand{\Spinh}{\mathrm{Spin}^h}
\newcommand{\MSpin}{\mathrm{MSpin}}
\newcommand{\MSpinc}{\mathrm{MSpin}^c}
\newcommand{\MSpinh}{\mathrm{MSpin}^h}
\newcommand{\BSpin}{\mathrm{BSpin}}
\newcommand{\BSpinc}{\mathrm{BSpin}^c}
\newcommand{\BSpinh}{\mathrm{BSpin}^h}
\newcommand{\KU}{\mathrm{KU}}
\newcommand{\KO}{\mathrm{KO}}
\newcommand{\KSp}{\mathrm{KSp}}
\newcommand{\BU}{\mathrm{BU}}
\newcommand{\BO}{\mathrm{BO}}
\newcommand{\BSO}{\mathrm{BSO}}
\newcommand{\BSp}{\mathrm{BSp}}
\newcommand{\ku}{\mathrm{ku}}
\newcommand{\ko}{\mathrm{ko}}
\newcommand{\ksp}{\mathrm{ksp}}
\newcommand{\U}{\mathrm{U}}
\newcommand{\Sp}{\mathrm{Sp}}
\renewcommand{\O}{\mathrm{O}}
\newcommand{\SO}{\mathrm{SO}}
\newcommand{\SU}{\mathrm{SU}}

\newcommand{\Z}{\mathbb{Z}}
\newcommand{\MSO}{\mathrm{MSO}}

\newcommand{\Sq}{\mathrm{Sq}}
\newcommand{\CP}{\mathbb{CP}}
\newcommand{\Pin}{\mathrm{Pin}}
\newcommand{\Th}{\operatorname{Th}}
\newcommand{\id}{\mathrm{id}}

\newenvironment{talign*}
 {\let\displaystyle\textstyle\csname align*\endcsname}
 {\endalign}

\makeatletter
\@namedef{subjclassname@2020}{\textup{2020} Mathematics Subject Classification}
\makeatother

\begin{document}
%%%%%%%%%%%%%%%%%%%%%%%%%%%%%%%%%%%%%%%%
\title{KSp-characteristic classes determine Spin$^h$ cobordism}

\author{Jonathan Buchanan}
\address{Department of Mathematics \\ Massachusetts Institute of Technology}
\email{jbuch333@mit.edu}

\author{Stephen McKean}
\address{Department of Mathematics \\ Brigham Young University} 
\email{mckean@math.byu.edu}
\urladdr{shmckean.github.io}

%%%%%%%%%%%%%%%%%%%%%%%%%%%%%%%%%%%%%%%%

\begin{abstract}
A classic result of Anderson, Brown, and Peterson states that the cobordism spectrum MSpin (respectively, MSpin$^c$) splits as a sum of Eilenberg--Mac Lane spectra and connective covers of real K-theory (respectively, complex K-theory) at 2. We develop a theory of symplectic K-theory classes and use these to build an explicit splitting for MSpin$^h$ in terms of Eilenberg--Mac Lane spectra and spectra related to symplectic K-theory. This allows us to determine the Spin$^h$ cobordism groups systematically. We also prove that two Spin$^h$-manifolds are cobordant if and only if their underlying unoriented manifolds are cobordant and their KSp-characteristic numbers agree.
\end{abstract}

\maketitle

\setcounter{tocdepth}{1}
\tableofcontents

\section{Introduction}
There is an intimate connection, brought to the fore by Atiyah--Bott--Shapiro~\cite{ABS64}, between topological $K$-theory and spin geometry. This connection was further strengthened in the work of Hopkins--Hovey~\cite{HH92}. A crucial bridge between these two results is built in the work of Anderson--Brown--Peterson, who gave a 2-local splitting of the cobordism spectra $\MSpin$ and $\MSpinc$ \cite{ABP67}. The Anderson--Brown--Peterson splitting of $\MSpin$ and $\MSpinc$ also yields combinatorial formulas for the $\Spin$ and $\Spinc$ cobordism groups.

The goal of this work is to give an explicit splitting for the cobordism spectrum $\MSpinh$ (analogous to the Anderson--Brown--Peterson splittings of $\MSpin$ and $\MSpinc$) in terms of ordinary cohomology classes and $\KSp$-characteristic classes. Here, $\Spinh$ is the \textit{quaternionic spin group}, defined as the colimit of the double covers $\Spinh(n)$ of $\SO(n)\times\Sp(1)$. Quaternionic spin theory was first studied systematically by Nagase~\cite{Nag95} and subsequently by Okonek--Teleman~\cite{OT96} and B\"ar~\cite{Bar99}, although $\Spinh(4)$ appeared even earlier~\cite{BFF78,HP78}. There has been a recent resurgence of interest in quaternionic spin theory, in part due to its role in physics \cite{Ch17,SSGR17,LS19,FH21,AM21,Law23,Hu23}.

Let $\Part_\mr{even}$ and $\Part_\mr{odd}$ denote the sets of even and odd partitions, respectively (see Notation~\ref{notn:partitions}). Given a spectrum $E$ and an integer $n$, let $E\langle n\rangle$ denote the $n$-connected cover of $E$. Our main result is an explicit analog of the Anderson--Brown--Peterson splitting.

\begin{thm}\label{thm:main}
Let $F$ be the fiber of the map $\ko \to H\Z / 2\Z$ classifying the non-trivial element of $H^0(\ko;\Z/2\Z)$. Then there are cohomology classes $Z \subset H^*(\MSpinh;\Z/2\Z)$ and a map of spectra
\[
    \MSpinh \to \bigvee_{I \in \Part_{\mathrm{even}}} \ksp \langle 4 | I | \rangle \vee \bigvee_{I \in \Part_{\mathrm{odd}}} \Sigma^{4 | I |} F \vee \bigvee_{z \in Z} \Sigma^{\deg z} H\Z / 2\Z
\]
that is a $2$-local equivalence.
\end{thm}

We prove Theorem~\ref{thm:main} by studying the mod 2 cohomology and homotopy groups of each summand, as well as describing the behavior of the map from $\MSpinh$ to each summand in cohomology. We then show that the map
\[\MSpinh\to\bigvee_{I\in\Part_\mr{even}}\ksp\langle4|I|\rangle\vee\bigvee_{I\in\Part_\mr{odd}}\Sigma^{4|I|}F\]
induces an isomorphism on certain associated Margolis homology groups. We conclude by taking the cokernel of the induced map on cohomology to construct the necessary Eilenberg--Mac Lane summands.

A key input to our approach is the construction of characteristic classes
\begin{align*}
    \kappa^I&\in\ksp\langle 4|I|\rangle^0(\MSpinh),\\ \eps^I&\in\Sigma^{4|I|}F^0(\MSpinh), 
\end{align*}
which we call \textit{$\KSp$-Pontryagin classes} and \textit{elephant classes}, respectively. These have associated \textit{$\KSp$-characteristic numbers}, which can be used to detect cobordisms between $\Spinh$-manifolds.

\begin{thm}\label{thm:cnumber condition}
Two $\Spinh$-manifolds are cobordant if and only if their $\KSp$-characteristic numbers and $\Z / 2\Z$-characteristic numbers are equal.
\end{thm}

We also discuss the asymptotic growth of $\Spinh$ cobordism groups, explicitly calculate the cobordism groups through degree 19999 (and provide the code used in this calculation), compute a $\KSp$-characteristic number of the Wu manifold, and list a few problems of interest in $\Spinh$ geometry.

\begin{rem}
During the preparation of this article, Mills released independent work that obtains some of the same results as us~\cite{Mil23}. In \textit{loc.~cit.} and this paper, we both derive a splitting at 2 of $\MSpinh$ and use it to calculate $\Spinh$ cobordism groups. However, in \textit{loc.~cit.}, the splitting is derived formally from the cohomology of $\MSpinh$, while our splitting is constructed explicitly from $\KSp$-Pontryagin classes and the quaternionic Atiyah--Bott--Shapiro map $\varphi^h : \MSpinh \to \ksp$. As a result of this explicit approach, Theorem~\ref{thm:main} is a strengthening of \cite[Theorem 1.1]{Mil23}.
\end{rem}

\subsection{Outline}
The layout of our article is as follows.
\begin{itemize}
\item In Section \ref{sec:quick facts} we summarize basic facts and constructions involving $\KSp$ and $\MSpinh$.
\item We give an overview of Anderson, Brown, and Peterson's approach to splitting $\MSpin$ in Section \ref{sec:abp}. We then discuss how this inspires our approach to splitting $\MSpinh$.
\item In Sections \ref{sec:cohomology} and \ref{sec:ksp classes}, we explore the cohomology of relevant spaces and spectra and discuss the maps of the splitting in cohomology.
\item In Section \ref{sec:margolis homology} we study the Margolis homology of the relevant Steenrod modules and show that the map from $\MSpinh$ to the sum of the $\ksp \langle 4|I| \rangle$ and $\Sigma^{4|I|} F$ is an isomorphism on Margolis homology.
\item In Section \ref{sec:abp splitting}, we define the ordinary cohomology classes involved in the splitting. We then prove Theorem \ref{thm:main} using the isomorphism on Margolis homology and a filtering procedure. This filtering procedure is inspired by one used in \cite{ABP67}, although some modifications are necessary due to $\MSpinh$ not being a ring spectrum.
\item We discuss the computation of $\Spinh$ cobordism groups in Section \ref{sec:computing groups}, as well as their asymptotic growth. Tables~\ref{table:spin bordism groups}, \ref{table:spinc bordism groups}, and \ref{table:spinh bordism groups} allow the reader to compare the $\Spin$, $\Spinc$, and $\Spinh$ cobordism groups through degree 99.
\item In Section \ref{sec:characteristic classes} we define the $\KSp$-characteristic numbers of a $\Spinh$ manifold and prove Theorem \ref{thm:cnumber condition}.
\item We outline some potential applications and related questions in Section \ref{sec:applications}.
\end{itemize}

\subsection*{Acknowledgements}
We thank Mike Hopkins for explaining the shearing map to us, and Leon Liu for helpful conversations about the degree four shifts relating the real and quaternionic settings. We also thank Jiahao Hu and Keith Mills for helpful remarks. We are grateful to the anonymous referee for their suggestions. JB received support from the Harvard College Research Program. SM received support from an NSF MSPRF grant (DMS-2202825).

\section{Quick facts about $\KSp$ and $\MSpinh$}\label{sec:quick facts}
In this section, we will recall some relevant background material. To begin, we will discuss symplectic $K$-theory. We will then give a brief introduction to Spin$^h$ geometry and gather some useful results from throughout the literature. See~\cite{Law23} for a nice survey of recent developments on Spin$^h$ manifolds.

\subsection{$\KSp$}
Topologically, Bott periodicity manifests as a repeating pattern in the loop spaces $\Omega^n\BO$, $\Omega^n\BU$, and $\Omega^n\BSp$. One can then define the $K$-theory spectra $\KO$, $\KU$, and $\KSp$ as the $\Omega$-spectra associated to $\BO$, $\BU$, and $\BSp$, respectively. It follows that these topological $K$-theory groups will repeat periodically (see Table~\ref{table:bott periodicity}).

\begin{table}
\caption{Bott periodicity in topological $K$-theory}\label{table:bott periodicity}
\begin{tabular}{c|C{2em}C{2em}C{2em}C{2em}C{2em}C{2em}C{2em}C{2em}}
    $n\pmod{8}$ & 0 & 1 & 2 & 3 & 4 & 5 & 6 & 7\\
    \hline
    $\pi_n\KU$ & $\mb{Z}$ & 0 & $\mb{Z}$ & 0 & $\mb{Z}$ & 0 & $\mb{Z}$ & 0\\
    $\pi_n\KO$ & $\mb{Z}$ & $\mb{Z}/2$ & $\mb{Z}/2$ & 0 & $\mb{Z}$ & 0 & 0 & 0\\
    $\pi_n\KSp$ & $\mb{Z}$ & 0 & 0 & 0 & $\mb{Z}$ & $\mb{Z}/2$ & $\mb{Z}/2$ & 0
\end{tabular}
\end{table}

In the process of proving Bott periodicity for $\BO$, one encounters the homotopy equivalences $\Omega^4\BO\simeq\BSp\times\mb{Z}$ and $\Omega^4\BSp\simeq\BO\times\mb{Z}$ (which are visible in Table~\ref{table:bott periodicity}). This means that we get a homotopy equivlence of $\Omega$-spectra $\Sigma^4\KO\to\KSp$, which is simply the identity map in each degree. In fact, the equivalence $\Sigma^4\KO\simeq\KSp$ is more than just an equivalence of spectra: it is an equivalence of $\KO$-modules.

\begin{prop}\label{prop:ko module}
The homotopy equivalence $\Sigma^4\KO\simeq\KSp$ is an equivalence of $\KO$-modules. 
\end{prop}
\begin{proof}
This is a standard fact, but we will point to a reference for the reader's convenience. The $\KO$-module structure on $\KSp$ is induced by taking the tensor product of a quaternionic bundle with a real bundle, which yields a quaternionic bundle. One has to show that this module map is a degree 4 shift of the tensor product of two real bundles, since the $\KO$-module structure on $\Sigma^4\KO$ is given by
\[\KO\wedge\Sigma^4\KO\simeq\Sigma^4(\KO\wedge\KO)\xrightarrow{\Sigma^4\mu}\KO.\]
(Here, $\mu:\KO\wedge\KO\to\KO$ is the ring structure induced by the tensor product of real bundles.) That the $\KO$-module map on $\KSp$ is indeed a degree 4 shift of the ring map on $\KO$ is worked out in \cite[\S 7]{Str92}. The relevant quaternionic bundle is denoted by $\theta$ in \textit{loc.~cit.}
\end{proof}

\subsection{$\Spinh(n)$}
We begin by introducing the Spin$^h$ groups. Write $\{\pm 1\}$ to denote the matrix group consisting of the identity matrix and its negative. Recall that $\Spin(n)$ is the universal cover of $\SO(n)$ for $n\geq 3$. Since $\Spin(n)\to\SO(n)$ is a double cover, we get a short exact sequence
\[1\to\{\pm 1\}\to\Spin(n)\to\SO(n)\to 1.\]
Analogously, $\Spinc(n)$ is defined as the double (not universal) cover of $\SO(n)\times\U(1)$, giving us the exact sequence
\[1\to\{\pm 1\}\to\Spinc(n)\to\SO(n)\times\U(1)\to 1.\]
We may thus realize $\Spinc(n)$ as the quotient $(\Spin(n)\times\U(1))/\{\pm 1\}\cong\Spin(n)\times_{\{\pm 1\}}\U(1)$. Regarding the unitary factor in $\Spinc(n)$ as carrying complex structure, we are inclined to rewrite $\Spin(n)$ as $\Spin(n)\cong\Spin(n)\times_{\{\pm 1\}}\O(1)$. This indicates how quaternionic (i.e.~symplectic) structure should be introduced.

\begin{defn}
Let $n\geq 3$. The \textit{quaternionic spin group} $\Spinh(n)$ is defined to be the double cover of $\SO(n)\times\SO(3)$. Equivalently, define
\[\Spinh(n):=\Spin(n)\times_{\{\pm 1\}}\Sp(1).\]
\end{defn}

\begin{rem}
The universal cover of $\SO(n,\mb{C})$ is often called complex spin, but this is different from $\Spinc$. We will never work with $\SO(n,\mb{C})$ in this article, so by \textit{complex spin} we always mean $\Spinc$.
\end{rem}

There is a commutative diagram
\[
    \begin{tikzcd}
        \arrow[from=1-1, to=1-2]
        \arrow[from=1-2, to=1-3]
        \arrow[from=1-1, to=2-2]
        \arrow[from=1-2, to=2-2]
        \arrow[from=1-3, to=2-2]
        \Spin (n) & \Spinc (n) & \Spinh (n) \\
        & \SO (n)
    \end{tikzcd}
\]
The map $\Spin (n) \to \Spinc (n)$ is the composition of the inclusion of $\Spin (n)$ into $\Spin (n) \times U (1)$ followed by the quotient map $\Spin (n) \times U (1) \to \Spinc (n)$, and the map $\Spinc (n) \to \Spinh (n)$ is induced by the inclusion $U (1) \to \Sp (1)$ and passage to quotients. The maps $\Spin (n) \times_{\left\{ \pm 1 \right\}} G \to \SO (n)$ are induced by the composition of the projection $\Spin (n) \times_{\left\{ \pm 1 \right\}} G \to \Spin (n)$ and the double cover $\Spin (n) \to \SO (n)$ and passage to the quotient group.

We now recall the definition of a Spin$^h$ structure, which was first introduced by Nagase \cite[p.~94]{Nag95}.

\begin{defn}
A \textit{Spin$^h$ structure} on a principal $\SO(n)$-bundle $P_{\SO(n)}$ consists of
\begin{enumerate}[(i)]
\item a principal $\SO(3)$-bundle $P_{\SO(3)}$,
\item a principal $\Spinh(n)$-bundle $P_{\Spinh(n)}$,
\item and a double cover $P_{\Spinh(n)}\to P_{\SO(n)}\times P_{\SO(3)}$ that is equivariant with respect to $\Spinh(n)\to\SO(n)\times\SO(3)$.
\end{enumerate}
A \textit{Spin$^h$ manifold} is a manifold whose tangent bundle admits a Spin$^h$ structure.
\end{defn}

\subsection{$\Spinh$-cobordism}\label{sec:spinh cobordism}
Now that we have a sequence of topological groups $\Spinh(n)$, we can speak of cobordisms of manifolds with stable Spin$^h$ structure. The resulting cobordism groups are encoded as the homotopy groups of the $\Spinh$-cobordism spectrum. $\Spinh$-cobordism and the quaternionic Atiyah--Bott--Shapiro map were developed independently by Hu \cite{Hu22} and the seminal work of Freed and Hopkins on invertible topological phases~\cite{FH21}.

\begin{defn}
Let $\BSpinh$ be the classifying space of stable $\Spinh$-vector bundles. Then the \textit{$\Spinh$-cobordism spectrum} is the Thom spectrum $\MSpinh$, whose $n^\text{th}$ space is the Thom space of the universal bundle over $\BSpinh(n)$.
\end{defn}

The maps between the $\Spin$, $\Spinc$, and $\Spinh$ groups induce a homotopy commutative diagram of classifying spaces
\[
    \begin{tikzcd}
        \arrow[from=1-1, to=1-2]
        \arrow[from=1-2, to=1-3]
        \arrow[from=1-1, to=2-2]
        \arrow[from=1-2, to=2-2]
        \arrow[from=1-3, to=2-2]
        \BSpin & \BSpinc & \BSpinh \\
        & \BSO
    \end{tikzcd}
\]
and therefore a diagram of Thom spectra
\[
    \begin{tikzcd}
        \arrow[from=1-1, to=1-2]
        \arrow[from=1-2, to=1-3]
        \MSpin & \MSpinc & \MSpinh.
    \end{tikzcd}
\]

In contrast to $\MSpin$ and $\MSpinc$, the spectrum $\MSpinh$ does not admit a ring structure. This comes from the fact that there is no ``quaternionic tensor product'' of vector spaces. That is, the tensor product of two quaternionic vector spaces need not be quaternionic, so the product of two Spin$^h$ manifolds need not be Spin$^h$. However, the tensor product of a real vector space and a quaternionic vector space is again quaternionic, which suggests that $\MSpinh$ might be an $\MSpin$-module. This was proved by Freed--Hopkins using an explicit \textit{shearing map} \cite[Equation (10.20)]{FH21}, but we will recall the relevant details.

\begin{setup}\label{setup:shearing}
Note that the data of a $\Spinh(n)$-bundle is equivalent to a pair $(E_n,E_3)$, where $E_n$ is a principal $\SO(n)$-bundle and $E_3$ is a principal $\SO(3)$-bundle such that $w_2(E_n)=w_2(E_3)$, where $w_i$ denotes the $i^\text{th}$ mod 2 Stiefel--Whitney class. Recall that $w_1(P)=w_2(P)=0$ for any principal $\Spin(n)$-bundle $P$. Indeed, $w_1$ vanishes on all $\SO(n)$-bundles. For $w_2$, the short exact sequence $1\to\mb{Z}/2\to\Spin(n)\to\SO(n)\to 1$ induces an exact sequence on cohomology 
\[H^1(-;\Spin(n))\to H^1(-;\SO(n))\xrightarrow{w_2} H^2(-;\mb{Z}/2),\]
so an $\SO(n)$-bundle lifts to a $\Spin(n)$-bundle if and only if $w_2$ vanishes. It follows that $(P\oplus E_3,E_3)$ corresponds to a $\Spinh(n+3)$-bundle, since $P\oplus E_3$ is a principal $\SO(n+3)$-bundle and
\[w_2(P\oplus E_3)=w_2(P)+w_1(P)w_1(E_3)+w_2(E_3)=w_2(E_3).\]
This gives us the shearing map on classifying spaces:
\begin{align*}
\BSpin(n)\times\BSO(3)&\to\BSpinh(n+3)\\
(P,E_3)&\mapsto(P\oplus E_3,E_3).
\end{align*}
Applying (homotopy) colimits, we get a map $\BSpin\times\BSO(3)\to\BSpin^h$. This map admits a homotopy inverse $(R,E_3)\mapsto(R\oplus(-E_3),E_3)$, where $-E_3$ is the virtual bundle associated to $E_3$ (which exists since we are working stably).
\end{setup}

\begin{lem}[Freed--Hopkins]\label{lem:shearing}
The map $\BSpin(n)\times\BSO(3)\to\BSpinh(n+3)$ of classifying spaces over $\BO$ given in Setup~\ref{setup:shearing} induces a homotopy equivalence $\Sigma^{-3}\MSpin\wedge\MSO(3)\to\MSpinh$.
\end{lem}
\begin{proof}
Because $\BSpin\times\BSO(3)\to\BSpinh$ is a homotopy equivalence, the result follows by taking Thom spectra. The shift by $-3$ can be seen at the level of Thom spaces, since the Thom space $\MSpin(n)\wedge\MSO(3)$ maps to $\MSpinh(n+3)$.
\end{proof}

Anderson--Brown--Peterson prove a 2-local splitting of the Thom spectra $\MSpin$ and $\MSpinc$~\cite{ABP67}. Since the homotopy groups of $\MSpin$ and $\MSpinc$ have no odd torsion \cite[p.~336]{Sto68}, it follows that one can completely determine the additive structure of the $\Spin$- and $\Spinc$-cobordism groups from the Anderson--Brown--Peterson splitting. We will prove an analogous 2-local splitting for $\MSpinh$ in Section~\ref{sec:abp splitting}. In order to determine the additive structure of $\pi_*\MSpinh$, we need to show that $\Spinh$-cobordism groups are odd torsion-free.

\begin{prop}\label{prop:no odd torsion}
Let $p$ be an odd prime. Then $\pi_*\MSpinh$ is finitely generated in each degree and has no $p$-torsion.
\end{prop}
\begin{proof}
By Lemma~\ref{lem:shearing}, it suffices to show that
\[\pi_*(\Sigma^{-3}\MSpin\wedge\MSO(3))\cong\MSpin_*\Sigma^{-3}\MSO(3)\]
has no $p$-torsion. We will argue via the Atiyah--Hirzebruch spectral sequence.\footnote{We learned this argument from Proposition~3.1 of Debray's lecture notes on Spin-$U_2$ bordism \cite{Deb21}.} In the present context, this has signature
\begin{equation}\label{eq:sseq for p torsion}
E^2_{s,t}=H_s(\Sigma^{-3}\MSO(3);\MSpin_t)\Longrightarrow\MSpin_{s+t}\Sigma^{-3}\MSO(3).
\end{equation}
We will show that there is no $p$-torsion on the $E^\infty$ page of this spectral sequence, which will imply that $\MSpin_*\Sigma^{-3}\MSO(3)$ has no $p$-torsion.
\begin{enumerate}[(i)]
\item $\MSpin_*$ is finitely generated and has no $p$-torsion by \cite[p.~336]{Sto68}.
\item Let $G$ be a finitely generated abelian group with no $p$-torsion. Since $\MSO(3)$ is defined as the Thom space of the universal bundle over $\BSO(3)$, the Thom isomorphism induces an isomorphism 
\[\tilde{H}_s(\Sigma^{-3}\MSO(3);G)\cong H_s(\BSO(3);G).\]
Since $H^*(\BSO(3);\mb{Z})$ has no $p$-torsion \cite[\S 30.5]{BH59}, the universal coefficient theorem implies that $H_*(\BSO(3);G)$ has no $p$-torsion.
\item The free summands of $\MSpin_*$ all lie in even degrees \cite[p.~340]{Sto68}. Similarly, the free summands of $H^*(\BSO(3);\mb{Z})$ all lie in even degrees \cite[Proposition~30.3]{BH59}. If $G$ is a finitely generated abelian group, the universal coefficient theorem thus implies that the free summands of $H_*(\BSO(3);G)$ all lie in even degrees. By the Thom isomorphism, the free summands of $H_*(\Sigma^{-3}\MSO(3);G)$ likewise lie in even degrees.
\end{enumerate}
Steps (i) and (ii) imply that there is no $p$-torsion on the $E^2$ page of Equation~\ref{eq:sseq for p torsion}. Any $p$-torsion on the $E^\infty$ page must therefore arise from a differential between free summands. Steps (i) and (iii) imply that no so such differentials exist, since either the source or target of any differential lies in odd degree.

Also, there are only finitely many nonzero groups on the $E^\infty$ page for a given total degree, and each group is finitely generated, so $\pi_* \MSpinh$ is finitely generated.
\end{proof}

\subsection{Atiyah--Bott--Shapiro map}\label{sec:abs map}
A critical aspect of Atiyah--Bott--Shapiro's work \cite{ABS64} on spin geometry are the Atiyah--Bott--Shapiro orientations
\begin{align*}
    \vphi^r&:\MSpin\to\KO,\\
    \vphi^c&:\MSpinc\to\KU.
\end{align*}

In analogy with $\vphi^r$ and $\vphi^c$, one might hope for an Atiyah--Bott--Shapiro orientation
\[\vphi^h:\MSpinh\to\KSp.\]
However, the lack of quaternionic tensor product prevents $\MSpinh$ and $\KSp$ from being ring spectra, so a map $\MSpinh\to\KSp$ cannot be an orientation. Nevertheless, Hu \cite[\S 1.3]{Hu22} and Freed--Hopkins \cite[\S 9.2.2]{FH21} independently constructed an Atiyah--Bott--Shapiro \textit{map} $\vphi^h$ that is a module map over the real Atiyah--Bott--Shapiro orientation $\vphi^r$.

\begin{prop}[Freed--Hopkins, Hu]\label{prop:abs map}
There is a \emph{quaternionic} Atiyah--Bott--Shapiro map $\vphi^h:\MSpinh\to\KSp$ that is a module map over $\vphi^r:\MSpin\to\KO$.
\end{prop}
\begin{proof}
    See \cite[p.~36]{Hu22}.
\end{proof}

\begin{rem}
    The spectrum maps $\vphi^r$, $\vphi^c$, and $\vphi^h$ are sometimes denoted in the literature by $\hat{\A}$, $\hat{\A}^c$, and $\hat{\A}^h$, since the real Atiyah--Bott--Shapiro orientation is the spectrum-level lift of the $\hat{\A}$-genus.
\end{rem}

Traditionally, the Atiyah--Bott--Shapiro orientations (or map in the quaternionic case) are constructed in terms of Clifford algebras. Joachim gave a purely homotopical construction of the real and complex ABS orientations \cite{Joa04}, which implies that the maps $\vphi^r$ and $\vphi^c$ are $\mr{E}_\infty$-ring maps. It would be interesting to give an analogous construction for $\vphi^h$.

\begin{prob}
    Give a homotopical construction of $\vphi^h:\MSpinh\to\KSp$, and prove that $\vphi^h$ is an $\mr{E}_\infty$-module map over the $\mr{E}_\infty$-ring map $\vphi^r:\MSpin\to\KO$.
\end{prob}

\section{Summary of the Anderson--Brown--Peterson splitting}\label{sec:abp}
Anderson, Brown, and Peterson's 2-local splittings of $\MSpin$ and $\MSpinc$ \cite{ABP67} involve extensive calculations, many of which are omitted from their write-up. In this section, we will attempt to summarize the proof strategy of Anderson--Brown--Peterson. Our proof of Theorem~\ref{thm:main} is largely inspired by the strategy outlined here, as we will discuss in Section~\ref{sec:proof outline}.

\begin{notn}\label{notn:partitions}
For $n \in \N$, let $\Part(n)$ denote the set of all partitions of $n$, and let $\Part_1(n)$ denote the set of all partitions of $n$ that do not have $1$ as a summand. Let $\Part:=\bigcup_{n = 0}^\infty \Part(n)$ be the set of all partitions, and let $\Part_1$ be the set of all partitions that do not have $1$ as a summand. Let $\Part_\text{even}:=\bigcup_{n=0}^\infty\Part(2n)$ be the set of all even partitions, and let $\Part_\text{odd}:=\bigcup_{n=0}^\infty\Part(2n+1)$ be the set of all odd partitions.

If $I =(i_1, \ldots, i_r)$ is a partition, we let $|I| = i_1 + \ldots + i_r$ denote the sum of $I$.
\end{notn}

\begin{notn}
    Unless otherwise specified, whenever we write $H^*$ in this article, we mean cohomology with $\mb{Z}/2\mb{Z}$-coefficients. Given a ring $R$, we write $HR$ to denote the associated Eilenberg--Mac Lane spectrum.
\end{notn}

\begin{defn}
    Given any spectrum $X$ and any integer $n$, there is a spectrum $X\langle n\rangle$ with $\pi_kX\langle n\rangle\cong 0$ for $k<n$ and a map $X\langle n\rangle\to X$ that induces an isomorphism $\pi_kX\langle n\rangle\cong\pi_k X$ for $k\geq n$. The spectrum $X\langle n\rangle$ equipped with the map $X\langle n\rangle\to X$ is called the \textit{$n$-connective cover} of $X$, and is unique up to unique isomorphism in the stable homotopy category.
\end{defn}

\begin{ex}
    The spectra $\ko$, $\ku$, and $\ksp$ are the 0-connective covers (or just \textit{connective covers}) $\KO\langle 0\rangle$, $\KU\langle 0\rangle$, and $\KSp\langle 0\rangle$, respectively.
\end{ex}

We can now state the Anderson--Brown--Peterson splitting of $\MSpin$.

\begin{thm}[Anderson--Brown--Peterson]\label{thm:spin splitting}
    There is a collection of (homogeneous) cohomology classes $Z \subset H^* \MSpin$ and a map of spectra
    \[
        \MSpin \to \bigvee_{k = 0}^\infty \left( \bigvee_{\Part_1(2k)} \ko\langle 8k\rangle \vee \bigvee_{\Part_1(2k+1)} \ko\langle 8k + 2\rangle \right) \vee \bigvee_{z \in Z} \Sigma^{\deg z} H\Z / 2\Z
    \]
    that is a $2$-local homotopy equivalence.
\end{thm}

Similarly, there is a splitting for $\MSpinc$:

\begin{thm}[Anderson--Brown--Peterson]\label{thm:spinc splitting}
    There is a set of (homogeneous) cohomology classes $Z \subset H^* \MSpinc$ and a map of spectra
    \begin{equation}
        \label{eq:spinc splitting}
        \MSpinc \to \bigvee_{I\in\Part} \ku\langle 4| I|\rangle \vee \bigvee_{z \in Z} \Sigma^{\deg z} H\Z / 2\Z
    \end{equation}
    that is a $2$-local homotopy equivalence.
\end{thm}

The proof strategy for these theorems boils down to the following four steps.

\begin{enumerate}[(i)]
    \item Use characteristic classes to construct the maps of spectra
    \begin{talign*}
        \Phi&:\MSpin\to\bigvee_k(\bigvee_{\Part_1(2k)}\ko\langle 8k\rangle\vee\bigvee_{\Part_1(2k+1)}\ko\langle 8k+2\rangle),\\
        \Phi^c&:\MSpinc\to\bigvee_\Part\ku\langle 4|I|\rangle.
    \end{talign*}
    The maps $\MSpin\to\ko\langle d\rangle$ come from \textit{$\KO$-Pontryagin classes}, whose definition we will recall in a moment. The maps $\MSpinc\to\ku\langle d\rangle$ are not explicitly discussed in \cite{ABP67}, but these come from $\KU$-characteristic classes. Both $\KO$- and $\KU$-characteristic classes are indexed by integer partitions, which accounts for the role of partitions in the 2-local splitting theorems.
    \item Assuming that there are maps
    \begin{talign*}
        \Psi&:\MSpin\to\bigvee_k(\bigvee_{\Part_1(2k)}\ko\langle 8k\rangle\vee\bigvee_{\Part_1(2k+1)}\ko\langle 8k+2\rangle)\vee\bigvee_z\Sigma^{\deg{z}}H\Z/2\Z,\\
        \Psi^c&:\MSpinc\to\bigvee_\Part\ku\langle 4|I|\rangle\vee\bigvee_z\Sigma^{\deg{z}}H\Z/2\Z.
    \end{talign*}
    inducing isomorphisms on mod 2-cohomology, deduce that $\Psi$ and $\Psi^c$ are 2-local equivalences.
    
    A map of spectra $X\to Y$ that induces an isomorphism on mod 2 cohomology is a 2-complete equivalence. If the homotopy groups of $X$ and $Y$ are all finitely generated (as is the case for all spectra that we will consider), then a 2-complete equivalence is a 2-local equivalence.
    \item Prove that $\Phi$ and $\Phi^c$ induce isomorphisms on certain Margolis homologies.

    In general, Margolis homology is easier to compute than mod 2 cohomology. Knowing that $\Phi$ and $\Phi^c$ induce isomorphisms on Margolis homology acts as the base case of an induction argument to prove that $\Phi$ and $\Phi^c$ induce isomorphisms on mod 2 cohomology.
    \item By identifying a suitable collection of ordinary cohomology classes of $\MSpin$ and $\MSpinh$, form the maps $\Psi$ and $\Psi^c$ and prove that these induce isomorphisms on mod 2 cohomology.
    
    Surjectivity is the easier part of this step. For injectivity, filter the source and target cohomologies by degree and show that if $\Psi$ and $\Psi^c$ induce isomorphisms on cohomology in degrees at most $n$, then $\Psi$ and $\Psi^c$ are injective on cohomology in degrees at most $n+1$.
\end{enumerate}

\subsection{$\KO$-Pontryagin classes}
The key to splitting $\MSpin$ and $\MSpinc$ are $\KO$-Pontryagin classes, since these give us maps from $\MSpin$ and $\MSpinc$ to the various $K$-theoretic summands in the splitting. These were first introduced in \cite[\S 4]{ABP66}, but we will recall the definition here. 

\begin{defn}
    The $i^\text{th}$ \textit{$\KO$-Pontryagin class} of an oriented vector bundle $V$ on $X$ is the unique class $\pi^i(V)\in\KO^0(X)$ such that
    \begin{enumerate}[(i)]
        \item $\pi^i$ is natural in $V$ for all $i$;
        \item for each complex line bundle $L$, we have
        \begin{itemize}
            \item $\pi^0(L)=1$,
        \item $\pi^1(L)=L-2$, and
        \item $\pi^i(L)=0$ for $i\geq 2$;
        \end{itemize}
        \item for any oriented bundles $V$ and $W$, we have
        \[\sum_{i\geq 0}\pi^i(V\oplus W)t^i=\left(\sum_{j\geq 0}\pi^j(V)t^j\right)\left(\sum_{k\geq 0}\pi^k(W)t^k\right).\]
    \end{enumerate}
    Given a partition $I=(i_1,\ldots,i_n)$, the $I^\text{th}$ $\KO$-Pontryagin class is the product $\pi^I:=\pi^{i_1}\cdots\pi^{i_n}$.
\end{defn}

The fact these three properties characterize $\pi^i$ (and hence $\pi^I$) follows from \cite[Proposition 4.4]{ABP66}. The classes $\pi^I\in\KO^0(\BSpin)$ determine maps $\MSpin\to\KO\langle d\rangle$ by multiplication with $\varphi^r : \MSpin \to \KO$, where the degree $d$ of connectivity is determined by the degree of $\pi^I$ (which are given in \cite[Theorem 2.1]{ABP67}).

\subsection{Proof strategy for splitting $\MSpinh$}\label{sec:proof outline}
Here is our strategy for proving Theorem \ref{thm:main}:
\begin{enumerate}[(i)]
    \item Compute the homotopy groups and cohomology of the spectrum $F$. Then use $\KO$-Pontryagin classes to build maps $\MSpinh \to \ksp\langle 4|I|\rangle$ for each partition $I$. For odd partitions $I$, show that these maps lift to maps $\MSpinh \to \Sigma^{4|I|} F$. Using some spectral sequence and characteristic class computations, describe what each of these maps does in cohomology (after choosing the correct lifts with some obstruction theory). Then, take wedge sums to form the map
    \[
        \Phi^h : \MSpinh \to \bigvee_{k = 0}^\infty \left( \bigvee_{\Part(2k)} \ksp \langle 8k \rangle \vee \bigvee_{\Part(2k + 1)} \Sigma^{8k + 4} F \right).
    \]
    \item Prove that $\Phi^h$ induces isomorphisms on Margolis homology by computing the Margolis homology of the cohomology of each summand and of $H^* \MSpinh$. As in the $\MSpin$ and $\MSpinc$ cases, this is almost everything we need to get an isomorphism in cohomology.
    \item Find a set of cohomology classes $Z \subset H^* \MSpinh$ such that the map
    \[
        \MSpinh \to \bigvee_{k = 0}^\infty \left( \bigvee_{\Part(2k)} \ksp \langle 8k \rangle \vee \bigvee_{\Part(2k + 1)} \Sigma^{8k + 4} F \right) \vee \bigvee_{z \in Z} \Sigma^{\deg z} H\Z / 2\Z
    \]
    induces a surjection in cohomology. Then, using the fact that $\Phi^h$ gives isomorphisms on Margolis homology, filter the Steenrod modules on both sides of this map by the lowest degree in which summands are nonzero and use this to show that $\Phi^h$ induces an injection as well, giving us an isomorphism in mod 2 cohomology.
    \item The isomorphism on mod 2 cohomology gives an equivalence of spectra in the 2-complete category, and this is a 2-local equivalence due to finitely generated homotopy groups.
\end{enumerate}
The overall plan is analogous to the strategy used in \cite{ABP67}. In steps (ii) and (iii), we have to make a few adjustments to deal with the fact that $\MSpinh$ is not a ring spectrum, but instead a module spectrum over $\MSpin$.

\section{Cohomology of $\BSpinh$ and $\MSpinh$}\label{sec:cohomology}
In Section~\ref{sec:ksp classes}, we will construct characteristic classes that realize the non-Eilenberg--Mac Lane summands of our splitting map. To do this, we need a few cohomological computations, which we collect in this section.

First, we present the cohomology of the classifying spaces $\BSpin$, $\BSpinc$, and $\BSpinh$.

\begin{prop}\label{prop:bspin cohomology}
    The cohomology of $\BSpin$ is the ring
    \[
        H^* \BSpin \cong \Z / 2\Z\left[ w_i \ \middle| \ i \geq 2, i \neq 2^k + 1 \text{ for }k\geq 0\right],
    \]
    where $w_i$ is the $i^\text{th}$ Stiefel--Whitney class of the canonical oriented bundle $\BSpin \to \BSO$.
\end{prop}
\begin{proof}
    See \cite[p.~292]{Sto68}.
\end{proof}

\begin{prop}\label{prop:bspinc cohomology}
    The cohomology of $\BSpinc$ is the ring
    \[
        H^* \BSpinc \cong \Z / 2\Z\left[ w_i \ \middle| \ i \geq 2, i \neq 2^{k + 1} + 1 \text{ for }k\geq 0\right],
    \]
    where $w_i$ is the $i^\text{th}$ Stiefel--Whitney class of the canonical oriented bundle $\BSpinc \to \BSO$.
\end{prop}
\begin{proof}
    See \cite[p.~293]{Sto68}.
\end{proof}

\begin{prop}\label{prop:bspinh cohomology}
    The cohomology of $\BSpinh$ is the ring
    \[
        H^* \BSpinh \cong \Z / 2\Z\left[ w_i \ \middle| \ i \geq 2, i \neq 2^{k + 2} + 1 \text{ for }k\geq 0\right],
    \]
    where $w_i$ is the $i^\text{th}$ Stiefel--Whitney class of the canonical oriented bundle $\BSpinh \to \BSO$. The Stiefel--Whitney class $w_5$ vanishes.
\end{prop}
\begin{proof}
    See \cite[Proposition~2.31]{Hu22}.
\end{proof}

\begin{rem}
Note that the classes $w_{2^k + 1}$ do not vanish in general, but are non-zero polynomials in lower Stiefel--Whitney classes. For $\BSpin$, one can find these relations by noting that $w_2 = 0$ for degree reasons, imposing the relation $\Sq^{2^{k - 1}} \cdots \Sq^2 \Sq^1 w_2 = 0$, and applying the Wu formula. For $\BSpinh$, the classes $w_{2^k+1}$ are determined by $\Sq^{2^{k - 1}} \cdots \Sq^4 w_5 = 0$.
\end{rem}

By pulling back the $\KO$-Pontryagin class $\pi^I$ under $\BSpinh\to\BSO$, we get a $\KO$-Pontryagin class $\pi^I_h$ for $\BSpinh$. A fact we will need later is that the associated map $\BSpinh\to\KO$ admits a lift to $\ko\langle 4|I|\rangle$ if $I\in\Part_\mr{even}$ or to $\ko\langle 4|I|-2\rangle$ if $I\in\Part_\mr{odd}$.

\begin{prop}\label{prop:lift of pontryagin class}
The map $\BSpinh\to\BSO\xrightarrow{\pi^I_h}\KO$ admits a lift to $\ko\langle 4|I|\rangle$ if $I\in\Part_\mr{even}$ or to $\ko\langle 4|I|-2\rangle$ if $I\in\Part_\mr{odd}$.
\end{prop}
\begin{proof}
Since all torsion in the integral cohomology is order two (see \cite[Corollary 2.36]{Hu22}), we see that the Pontryagin class $p_I = p_{i_1} \ldots p_{i_r}$ corresponding to a partition $I$ is non-torsion, since its reduction modulo two is $w_{2i_1}^2 \ldots w_{2i_r}^2$ and we know this is not zero. So $p_I$ is nonzero after rationalization. Moreover, there is no integral class $x$ such that $2x = p_I$ after rationalization, since this would imply $p_I - 2x$ is a torsion class, which can then be written as $\delta y$ for some mod $2$ cohomology class $y$, where $\delta$ is the Bockstein homomorphism. Reducing mod 2, we see that $w_{2i_1}^2 \ldots w_{2i_r}^2 = \Sq^1 y$. This contradicts Lemma \ref{lem:pontryagin class not q0 boundary}, the proof of which we save for our discussion of Margolis homology.

Hence the hypotheses of the proposition of \cite[p. 303, 304]{Sto68} are met, so for $| I |$ even, $\pi_R^I$ admits a lift to $\ko \langle 4 | I | \rangle$ with $x_{4 | I |}$ mapping to $p_I + \Sq^3 \Sq^1 \alpha$ for some $\alpha$ after reduction mod 2, and for $| I |$ odd, $\pi_R^I$ admits a lift to $\ko \langle 4 | I | - 2 \rangle$ such that if $x$ is the image of $x_{4 | I | - 2}$, $\Sq^2 x = p_I$ (\cite[p. 314]{Sto68}).
\end{proof}

Propositions~\ref{prop:bspin cohomology}, \ref{prop:bspinc cohomology}, and \ref{prop:bspinh cohomology} immediately determine the cohomology of $\MSpin$, $\MSpinc$, and $\MSpinh$ via the Thom isomorphism. The action of the Steenrod algebra on each of these modules is determined by the rule $\Sq^i u = w_i u$, where $u$ is the Thom class of any bundle (see \cite[p. 91]{MS74}). The maps between the Thom spectra induce maps
\[
    \begin{tikzcd}
        \arrow[from=1-1, to=1-2]
        \arrow[from=1-2, to=1-3]
        H^*\MSpinh & H^*\MSpinc & H^*\MSpin
    \end{tikzcd}
\]
of cohomology, with Thom classes mapping to Thom classes. Since the maps of classifying spaces are maps over $\BSO$, Stiefel--Whitney classes map to the corresponding Stiefel--Whitney classes.

\subsection{Steenrod modules}
Modules over the Steenrod algebra are ubiquitous in \cite{ABP67}, as well as the present paper. Indeed, if a map of spectra $X\to Y$ is to be a 2-local equivalence, then one needs to show that the induced map $H^*Y\to H^*X$ is an isomorphism of modules over the mod 2 Steenrod algebra. In this section, we collect a few results about the cohomology of various connective covers of $\ko$, $\ku$, and $\ksp$ in terms of Steenrod modules.

\begin{notn}
    Throughout this article, $\A$ will denote the mod 2 Steenrod algebra.
\end{notn}

\begin{prop}\label{prop:cohomology of ko<k>}
    Suppose $k = 0, 1, 2, 4 \pmod{8}$. Then there is a class $x_k \in H^k \ko\langle k\rangle$ such that the map
    \begin{align*}
        \A&\to H^*\ko\langle k\rangle\\
        1&\mapsto x_k
    \end{align*}
    induces an isomorphism $\A / I_k \to H^* \ko\langle k\rangle$, where $I_k \subset \A$ is the left ideal
    \[
        I_k =
        \begin{cases}
            \A \Sq^1 + \A \Sq^2 & k = 0 \pmod{8}, \\
            \A \Sq^2 & k = 1 \pmod{8}, \\
            \A \Sq^3 & k = 2 \pmod{8}, \\
            \A \Sq^1 + \A \Sq^5 & k = 4 \pmod{8}.
        \end{cases}
    \]
\end{prop}
\begin{proof}
    See \cite[p.~295]{Sto68}.
\end{proof}

Using the Bott periodicity isomorphism $\KSp \cong \Sigma^4 \KO$ and the uniqueness of connective covers, we see that $\Sigma^4 \ko\langle k\rangle \cong \ksp\langle k + 4\rangle$, giving us the following result in cohomology:

\begin{cor}\label{cor:cohomology of ksp}
    If $k = 0, 4, 5, 6 \pmod{8}$, there is a class $y_k \in H^k \ksp\langle k\rangle$ such that the map
    \begin{align*}
    \A&\to H^*\ksp\langle k\rangle\\
    1&\mapsto y_k
    \end{align*}
    induces an isomorphism $\A / I_k \to H^*\ksp\langle k\rangle$, where $I_k \subset \A$ is the left ideal
    \[
        I_k =
        \begin{cases}
            \A \Sq^1 + \A \Sq^5 & k = 0 \pmod{8}, \\
            \A \Sq^1 + \A \Sq^2 & k = 4 \pmod{8}, \\
            \A \Sq^2 & k = 5 \pmod{8}, \\
            \A \Sq^3 & k = 6 \pmod{8}.
        \end{cases}
    \]
\end{cor}
\begin{proof}
    By Bott periodicity, this is a degree 4 shift of Proposition~\ref{prop:cohomology of ko<k>}.
\end{proof}

We also describe the cohomology of $\ku$ and its role in the splitting of $\MSpinc$, as it will be relevant later.

\begin{prop}
    For each $k$, there is a class $z_{2k} \in H^{2k} \ku \langle 2k \rangle$ such that the map
    \begin{align*}
    \A&\to H^* \ku \langle 2k \rangle\\
    1&\mapsto z_{2k}
    \end{align*}
    induces an isomorphism $\A /(\A \Sq^1 + \A \Sq^3)\to H^* \ku \langle 2k \rangle$.
\end{prop}
\begin{proof}
    See \cite[p.~295]{Sto68}.
\end{proof}

The image of $z_{4k}$ under the map $H^*\ku\langle 4k\rangle\to H^*\MSpinc$ is particularly tractable.

\begin{lem}
    The splitting in Equation \ref{eq:spinc splitting} If $I \in \Part$ is a partition, then the map $\MSpinc \to \ku\langle 4|I|\rangle$ in Equation~\ref{eq:spinc splitting} induces $z_{4|I|} \mapsto p_I U_c$ in cohomology, where $U_c \in H^* \MSpinc$ is the Thom class and $p_I$ is the $I^\text{th}$ Pontryagin class.
\end{lem}
\begin{proof}
    It is shown that the complexification of the $\KO$-Pontryagin classes can be chosen so that $z_{4 | I |} \mapsto p_I$ in \cite[p.~304]{Sto68}. The fact that multiplying with the orientation $\MSpinc \to \KU$ induces $z_{4 | I |} \mapsto p_I U_c$ is shown in \cite[p.~317]{Sto68}. The splitting map with this property is assembled at \cite[p.~319]{Sto68}.
\end{proof}

We now define three Steenrod modules that will show up when we compute the cohomology of various spectra. Proposition~\ref{prop:cohomology of ko<k>} states that $H^*\ko$ is a module over the subalgebra in $\A$ generated by $\Sq^0$, $\Sq^1$, and $\Sq^2$.

\begin{defn}
    Let $\A_1$ denote the subalgebra of $\A$ generated by $\Sq^0,\Sq^1$, and $\Sq^2$. Note that $\A_1$ is often denoted $\A(1)$ in the literature. Our choice of notation is both an homage to the notation used in \cite{ABP67} and an effort to declutter many equations in the sequel.
\end{defn}

The following $\A_1$-module occurs as a summand of $H^*\MSpinh$ (as will be discussed in Section~\ref{sec:ksp classes}).

\begin{defn}\label{def:elephant}
    The \textit{elephant} $E$ is the $\A_1$-submodule of $\Sigma^{-1} \A_1$ generated by $\Sq^1$ and $\Sq^2$ (see Figure~\ref{fig:two-elephants}).\footnote{The elephant appears in \cite{BC18} under the name $R_2$.}
\end{defn}

The next $\A_1$-module is well-known.

\begin{defn}\label{def:upside down question mark}
    The \textit{upside-down question mark} is the $\A_1$-module
    \[\udq:=\A_1 /(\A_1 \Sq^1 + \A_1(\Sq^5 + \Sq^4 \Sq^1))\]
    (see Figure~\ref{fig:question mark}).
\end{defn}

Finally, we also define a module that generates the summands of $H^*\MSpinc$.

\begin{defn}\label{def:C}
    Let $C$ denote the $\A_1$-module $\A_1 /(\A_1 \Sq^1 + \A_1 \Sq^3)$ (see Figure~\ref{fig:C}).
\end{defn}

\begin{figure}
\centering
\begin{subfigure}{0.3\textwidth}
\centering
\begin{tikzpicture}[scale=0.6]
    \fill[white] (-1,-0.5) -- (-1,5) -- (4,5) -- (4,-0.5) -- cycle;
    \draw[fill=black] (0,0) circle (3pt);
    \draw[fill=black] (0,1) circle (3pt);
    \draw[fill=black] (0,2) circle (3pt);
    \draw[fill=black] (2,2) circle (3pt);
    \draw[fill=black] (2,3) circle (3pt);
    \draw[fill=black] (2,4) circle (3pt);
    \draw[fill=black] (2,5) circle (3pt);
    \draw[thick] (0,1) -- (0,2);
    \draw[thick] (2,2) -- (2,3);
    \draw[thick] (2,4) -- (2,5);
    \draw[thick] (0,0) .. controls (2,1) and (0,1) .. (2,2);
    \draw[thick] (0,1) .. controls (2,2) and (0,2) .. (2,3);
    \draw[thick] (0,2) .. controls (2,3) and (0,3) .. (2,4);
    \draw[thick] (2,3) .. controls (3.555,3) and (3.555,5) .. (2,5);
\end{tikzpicture}
\caption{The elephant $E$}
\label{fig:elephant-steenrod}
\end{subfigure}%
\begin{subfigure}{0.3\textwidth}
\centering
\includegraphics[width=.8\textwidth]{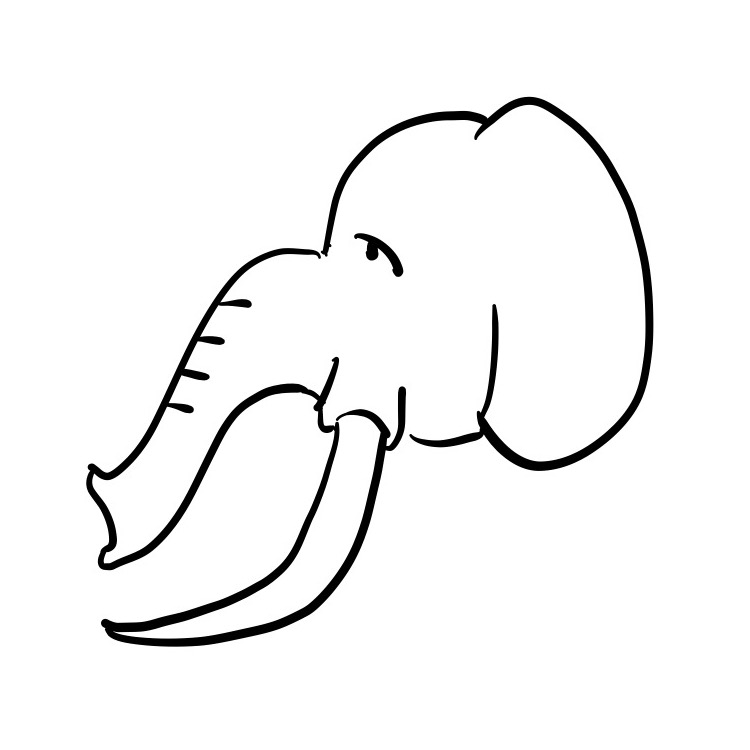}
\caption{An elephant, unnamed}
\label{fig:elephant-cartoon}
\end{subfigure}
\caption{The $\A_1$-module $E$ and its namesake}
\label{fig:two-elephants}
\end{figure}

\begin{figure}
\centering
\begin{subfigure}{0.3\textwidth}
\centering
\begin{tikzpicture}[scale=0.6]
    \fill[white] (-1,-0.5) -- (-1,3.5) -- (1,3.5) -- (1,-0.5) -- cycle;
    \draw[fill=black] (0,0) circle (3pt);
    \draw[fill=black] (0,2) circle (3pt);
    \draw[fill=black] (0,3) circle (3pt);
    \draw[thick] (0,0) .. controls (-1.555,0) and (-1.555,2) .. (0,2);
    \draw[thick] (0,2) -- (0,3);
\end{tikzpicture}
\caption{The module $\udq$}
\label{fig:question mark}
\end{subfigure}%
\begin{subfigure}{0.3\textwidth}
\centering
\begin{tikzpicture}[scale=0.6]
    \fill[white] (-1,-0.5) -- (-1,3.5) -- (1,3.5) -- (1,-0.5) -- cycle;
    \draw[fill=black] (0,0) circle (3pt);
    \draw[fill=black] (0,2) circle (3pt);
    \draw[thick] (0,0) .. controls (-1.555,0) and (-1.555,2) .. (0,2);
\end{tikzpicture}
\caption{The module $C$}
\label{fig:C}
\end{subfigure}
\caption{The $\A_1$-modules $\udq$ and $C$}
\end{figure}

\section{$\KSp$-Pontryagin and elephant classes}\label{sec:ksp classes}
We now begin constructing the map given in Theorem~\ref{thm:main}. In this section, we will give the maps to the summands $\ksp\langle 8n\rangle$ and $\Sigma^{8n+4}F$ (see Definition~\ref{def:elephant spectrum}) by defining characteristic classes for the cohomology theories defined by these spectra.

Since $\KSp$ is a $\KO$-module, there is a map $\KO \wedge \KSp \to \KSp$ satisfying the usual axioms in the homotopy category. The smash product $\ko\langle n\rangle \wedge \ksp\langle m\rangle$ is $(n + m - 1)$-connected, so there is a unique map $\ko\langle n\rangle \wedge \ksp\langle m\rangle \to \ksp\langle n + m\rangle$ fitting into the diagram
\begin{equation}\label{eq:truncated module structure}
    \begin{tikzcd}
        \arrow[from=1-1, to=1-2, dashed]
        \arrow[from=1-1, to=2-1]
        \arrow[from=1-2, to=2-2]
        \arrow[from=2-1, to=2-2]
        \ko\langle n\rangle \wedge \ksp\langle m \rangle & \ksp\langle n + m\rangle \\
        \KO \wedge \KSp & \KSp.
    \end{tikzcd}
\end{equation}
Recall that there is a map $\vphi^h: \MSpinh \to \ksp$ such that $y_0 \mapsto U_h$ in cohomology \cite[Remark~3.26]{Hu22}, where $y_0\in H^*\ksp$ is the class mentioned in Corollary~\ref{cor:cohomology of ksp} and $U_h\in H^*\MSpinh$ is the Thom class. For any partition $I$, the $\KO$-Pontryagin class $\pi^I_h\in\ko\langle n\rangle^0 \BSpinh$ determines a class on $\MSpinh$ through the composite
\begin{equation}\label{eq:ksp-pontryagin}
    \begin{tikzcd}
        \arrow[from=1-1, to=1-2]
        \arrow["{\pi^I_h\wedge\vphi^h}", from=1-2, to=1-3]
        \arrow[from=1-3, to=1-4]
        \MSpinh & \BSpinh \wedge \MSpinh & \ko\langle n\rangle \wedge \ksp & \ksp\langle n\rangle.
    \end{tikzcd}
\end{equation}
Here, the map $\MSpinh \to \BSpinh \wedge \MSpinh$ is the Thom diagonal. While $
\vphi^h$ is not a Thom class in the sense of an orientation with respect to a ring spectrum, the principle of transferring classes from a base space to the Thom spectrum via multiplication is the same. By looking for copies of $H^*\ksp\langle n\rangle$ in $H^*\MSpinh$, we get a sense of what classes $\pi^I_h\in\ko\langle n\rangle^0\BSpinh$ we need. This is the method we will use to generate all of the maps that decompose $\MSpinh$, besides those to Eilenberg--Mac Lane spectra, which originate in ordinary cohomology. 

\subsection{Module structure for $\ksp\langle n\rangle$}
The crux of understanding Equation~\ref{eq:ksp-pontryagin} is the behavior of the maps $\ko \langle n\rangle \wedge \ksp \to \ksp \langle n\rangle$ in cohomology. It turns out that we will only need the cases $n = 8k$ and $n = 8k + 2$ in order to prove Theorem~\ref{thm:main}. 

For the case $n=8k$, we can use the $\KO$-module structure of $\KSp$.

\begin{lem}\label{lem:8k multiplication}
    The map $\ko\langle 8k\rangle \wedge \ksp \to \ksp\langle 8k\rangle$ induces $y_{8k} \mapsto x_{8k} \otimes y_0$ in cohomology.
\end{lem}
\begin{proof}
    Since $\ko$ is a ring spectrum, its cohomology $A = H^* \ko$ can be equipped with the structure of a coalgebra. In particular, the diagram
    \[
        \begin{tikzcd}
            \arrow["{\Delta}", from=1-1, to=1-2]
            \arrow["{\id_A}", from=1-1, to=2-1]
            \arrow["{\epsilon \otimes \id_A}", from=1-2, to=2-2]
            \arrow[from=2-2, to=2-1]
            A & A \otimes A \\
            A & \Z / 2\Z \otimes A
        \end{tikzcd}
    \]
    commutes, where $\Delta:A\to A\otimes A$ and $\epsilon:A\to\Z/2\Z$ are $\Z/2\Z$-linear maps and the bottom arrow is the canonical isomorphism. Since the only element of $A$ of degree zero is $x_0$, we can write $\Delta x_0 = a x_0 \otimes x_0$ for some coefficient $a \in \Z / 2\Z$. The diagram above says $x_0 = a \epsilon(x_0) x_0$, so $a = \epsilon( x_0 ) = 1$. Next, let $B = H^* \ksp$. Since $\ksp$ is a $\ko$-module spectrum, $B$ can be given the structure of an $A$-comodule. In particular, the diagram
    \[
        \begin{tikzcd}
            \arrow["{\mu}", from=1-1, to=1-2]
            \arrow["{\id_B}", from=1-1, to=2-1]
            \arrow["{\epsilon \otimes \id_B}", from=1-2, to=2-2]
            \arrow[from=2-2, to=2-1]
            B & A \otimes B \\
            B & \Z / 2\Z \otimes B
        \end{tikzcd}
    \]
    commutes. Since $y_0 \in \ksp$ is the only element of degree zero, we have $\mu y_0 = b x_0 \otimes y_0$ for some $b \in \Z / 2\Z$. The diagram above then says that $y_0 = b \epsilon(x_0) y_0$, so $b = 1$, and therefore the map $\ko \wedge \ksp \to \ksp$ has $y_0 \mapsto x_0 \otimes y_0$ in cohomology.

    Finally, using $\Sigma^{8k} \ko \cong \ko\langle 8k\rangle$ and $\Sigma^{8k} \ksp \cong \ksp\langle 8k\rangle$, taking suspensions of the map $\ko\wedge\ksp\to\ksp$ gives us a map $\ko\langle 8k\rangle \wedge \ksp \to \ksp\langle 8k\rangle$ with $y_{8k} \mapsto x_{8k} \otimes y_0$ in cohomology, since $\Sigma^{8k}( \ko \wedge \ksp) \cong \Sigma^{8k} \ko \wedge \ksp$. We just have to check that this is the original map we were concerned with. Recall that the desired map $\ko\langle 8k\rangle\wedge\ksp\to\ksp\langle 8k\rangle$ is the unique map making Diagram~\ref{eq:truncated module structure} commute. It thus suffices to show that the diagram
    \begin{equation}\label{eq:suspended module}
        \begin{tikzcd}
            \arrow[from=1-1, to=1-2]
            \arrow[from=1-1, to=2-1]
            \arrow[from=1-2, to=2-2]
            \arrow[from=2-1, to=2-2]
            \Sigma^{8k}\ko \wedge \ksp & \Sigma^{8k}\ksp \\
            \KO \wedge \KSp & \KSp
        \end{tikzcd}
    \end{equation}
    commutes. When $k=0$, Diagram~\ref{eq:suspended module} is a special case of Diagram~\ref{eq:truncated module structure} and hence commutes. Suspending $8k$ times gives us the desired diagram, except we must verify that the bottom edge is still the module multiplication map. But this is true because $\KSp \cong \Sigma^4 \KO$ and $\Sigma^8 \KO \cong \KO$ as $\KO$-module spectra (Proposition~\ref{prop:ko module}).
\end{proof}

\subsection{The elephant spectrum}
The case $n = 8k + 2$ is considerably more complicated. For $n=8k$, one can find $H^*\ksp\langle n\rangle$ summands in the cohomology of $\MSpinh$, but it appears that the cohomology of a different spectrum arises at $n=8k+2$. This leads us to the following definition.

\begin{defn}\label{def:elephant spectrum}
    Consider the map $\ko \to H\Z$ inducing an isomorphism on $\pi_0$. Composing with the quotient map $H\Z \to H\Z / 2\Z$, we get a map $\ko \to H\Z / 2\Z$. Define the \textit{elephant spectrum}\footnote{This name will be justified in Lemma~\ref{lem:cohomology of F}.} $F:=\mr{fib}(\ko\to H\Z/2\Z)$ as the fiber of this map.
\end{defn}

Shifting by $8k + 4$, we observe fiber sequences
\begin{equation}\label{eq:fiber sequence for F}
    \begin{tikzcd}
        \arrow[from=1-1, to=1-2]
        \arrow[from=1-2, to=1-3]
        \Sigma^{8k + 4} F & \ksp \langle 8k + 4 \rangle & \Sigma^{8k + 4} H\Z / 2\Z.
    \end{tikzcd}
\end{equation}

We can now readily compute the homotopy and cohomology of $F$.

\begin{lem}\label{lem:homotopy of F}
For $k < 0$, we have $\pi_k F \cong 0$. For $k \geq 0$, we have
\[
\pi_k F \cong
\begin{cases}
    \Z & k = 0, 4 \pmod{8}, \\
    \Z / 2\Z & k = 1, 2 \pmod{8}, \\
    0 & \text{otherwise}.
\end{cases}
\]
\end{lem}
\begin{proof}
    For all $k$, we have exact sequences
    \begin{equation}\label{eq:fibration of F}
        \begin{tikzcd}
            \arrow[from=1-1, to=1-2]
            \arrow[from=1-2, to=1-3]
            \arrow[from=1-3, to=1-4]
            \arrow[from=1-4, to=1-5]
            \pi_{k + 1} \ko & \pi_{k + 1} H\Z / 2\Z & \pi_k F & \pi_k \ko & \pi_k H\Z / 2\Z
        \end{tikzcd}
    \end{equation}
    from the long exact sequence of a fibration.
    \begin{enumerate}[(i)]
    \item If $k$ is not $0$ or $-1$, then $\pi_{k + 1} H\Z / 2\Z \cong \pi_k H\Z / 2\Z \cong 0$ and thus $\pi_k F \to \pi_k \ko$ is an isomorphism. 
    \item If $k = 0$, then $\pi_1 H\Z / 2\Z \cong 0$, so $\pi_0 F$ is the kernel of the quotient map $\Z\cong\pi_0 \ko \to \pi_0 H\Z / 2\Z\cong\Z/2\Z$. Since $\pi_0F$ is the kernel of the quotient map $\Z\to\Z/2\Z$, it follows that $\pi_0 F \to \pi_0 \ko$ can be identified with the inclusion $2\Z \to \Z$. 
    \item Finally, suppose $k = -1$. The map $\pi_0 \ko \to \pi_0  H\Z / 2\Z$ is an epimorphism, so $\pi_0 H\Z / 2\Z \to \pi_{-1} F$ is zero and $\pi_{-1} H\Z / 2\Z \cong 0$. Exactness of Equation~\ref{eq:fibration of F} implies that $\pi_{-1} F \to \pi_{-1} \ko$ is an isomorphism.\qedhere
    \end{enumerate}
\end{proof}

\begin{lem}\label{lem:cohomology of F}
    The cohomology of $F$ is given by $H^* F \cong \A \otimes_{\A_1} E$, where $E$ is the elephant (see Definition~\ref{def:elephant}).
\end{lem}
\begin{proof}
    By the definition of $E$, we have a short exact sequence
    \[
        \begin{tikzcd}
            \arrow[from=1-1, to=1-2]
            \arrow[from=1-2, to=1-3]
            \arrow[from=1-3, to=1-4]
            \arrow[from=1-4, to=1-5]
            0 & \Sigma^1 E & \A_1 & \Z / 2\Z & 0.
        \end{tikzcd}
    \]
    Since $\A$ is flat (in fact, free) as a right $\A_1$-module, tensoring gives us a short exact sequence
    \begin{equation}\label{eq:tensored sequence}
        \begin{tikzcd}
            \arrow[from=1-1, to=1-2]
            \arrow[from=1-2, to=1-3]
            \arrow[from=1-3, to=1-4,"\phi"]
            \arrow[from=1-4, to=1-5]
            0 & \A \otimes_{\A_1} \Sigma^1 E & \A & \A \otimes_{\A_1} \Z / 2\Z & 0
        \end{tikzcd}
    \end{equation}
    of $\A$-modules. Recall that $H^* H\Z / 2\Z \cong \A$ and $H^* \ko \cong \A \otimes_{\A_1} \Z / 2\Z$. Since the map $\ko \to H\Z / 2\Z$ is non-trivial, it must represent the bottom class of $H^* \ko$ and therefore induces the map $\phi:\A \to \A \otimes_{\A_1} \Z / 2\Z$ in Equation~\ref{eq:tensored sequence}. 
    
    The fiber sequence defining $F$ gives us a long exact sequence
    \begin{equation}\label{eq:long exact cohomology}
        \begin{tikzcd}
            \arrow[from=1-1, to=1-2]
            \arrow[from=1-2, to=1-3]
            \arrow[from=1-3, to=1-4]
            \arrow[from=1-4, to=1-5]
            H^* H\Z / 2\Z & H^* \ko & H^* F & H^{* + 1} H\Z / 2\Z & H^{* + 1} \ko
        \end{tikzcd}
    \end{equation}
    in cohomology. Since $H^* H\Z / 2\Z \to H^* \ko$ is an epimorphism, Equation~\ref{eq:long exact cohomology} induces exact sequences
    \[
        \begin{tikzcd}
            \arrow[from=1-1, to=1-2]
            \arrow[from=1-2, to=1-3]
            \arrow[from=1-3, to=1-4]
            0 & H^* F & H^{* + 1} H\Z / 2\Z & H^{* + 1} \ko.
        \end{tikzcd}
    \]
    Thus $H^* F$ is the kernel of the map $H^* H\Z / 2\Z \to H^* \ko$ shifted by $-1$. That is, $H^*F\cong\Sigma^{-1}\ker\phi$, so Equation~\ref{eq:tensored sequence} implies $H^*F\cong\A \otimes_{\A_1} E$.
\end{proof}

\begin{rem}
Note that while the homotopy groups of $F$ are abstractly isomorphic to those of $\ko$, they have a different structure as a module over $\pi_* \mathbb{S}$. This can be seen in the Adams spectral sequence for $\A \otimes_{\A_1} E$ in \cite[Figure 29]{BC18} (note that $E$ is referred to as $R_2$ in \textit{loc.~cit.}).
\end{rem}

\subsection{Aside on integral cohomology}
The cohomology of $\Sigma^{8k+4}F$ arises in the cohomology of $\MSpinh$, which leads us to look for \textit{elephant classes} $\MSpinh\to\Sigma^{8k+4}F$. We will build these using $\KO$-Pontryagin classes once we know how to lift $\ko\langle 8k+2\rangle\wedge\ksp\to\ksp\langle 8k+2\rangle$ to $\ko\langle 8k+2\rangle\wedge\ksp\to\Sigma^{8k+4}F$. To do this, we need a few results in integral cohomology.

\begin{lem}\label{lem:c_4 of tensor}
Let $\mr{pr}_i:\BU\times\BU\to\BU$ denote projection onto the $i^\text{th}$ factor for $i=1,2$. Let $\gamma\to\BU$ be the classifying virtual bundle. Let $\alpha:=\mr{pr}_1^*\gamma$ and $\beta:=\mr{pr}_2^*\gamma$, so that the external tensor product $\alpha\otimes\beta$ is a virtual bundle on $\BU\times\BU$. Then in $H^*(\BU\times\BU;\Z)$, we have
    \begin{equation}
        \label{eq:c4}
        c_4(\alpha \otimes \beta) = -6 c_2(\alpha) c_2(\beta) \pmod{c_1(\alpha), c_1(\beta)}.
    \end{equation}
\end{lem}
\begin{proof}
    This can be computed using the Chern character. Recall that the Chern character of a virtual bundle $\xi$ with rank $n$ is defined to be
    \[
        \mathrm{ch}(\xi) = n + \sum_{k = 1}^\infty \frac{s_k(c(\xi))}{k!},
    \]
    where the $s_k$ are polynomials of (cohomological) degree $k$ in the Chern classes \cite[pp.~188]{MS74}. In particular, the first four $s_k$ are
    \begin{align*}
        s_1(c(\xi)) & = c_1(\xi), \\
        s_2(c(\xi)) & = c_1(\xi)^2 - 2c_2(\xi), \\
        s_3(c(\xi)) & = c_1(\xi)^3 - 3c_1(\xi) c_2(\xi) + 3c_3(\xi), \\
        s_4(c(\xi)) & = c_1(\xi)^4 - 4c_1(\xi)^2 c_2(\xi) + 2c_2(\xi)^2 + 4c_1(\xi) c_3(\xi) - 4c_4(\xi).
    \end{align*}
    Working modulo the ideal generated by $c_1(\alpha)$ and $c_1(\beta)$, we have
    \[
        \mathrm{ch}(\alpha) = -c_2(\alpha) + \frac{1}{2} c_3(\alpha) + \cdots \pmod{c_1(\alpha),c_1(\beta)},
    \]
    and similarly for $\mathrm{ch}(\beta)$. Since the Chern character is multiplicative over tensor products, we see that
    \[
        \mathrm{ch}(\alpha \otimes \beta) = c_2(\alpha) c_2(\beta) + \text{higher degree terms} \pmod{c_1(\alpha),c_1(\beta)}.
    \]
    Thus $c_1(\alpha \otimes \beta) = c_2(\alpha \otimes \beta) = c_3(\alpha \otimes \beta) = 0 \pmod{c_1(\alpha),c_1(\beta)}$, and $c_4(\alpha \otimes \beta)$ is given by the equation
    \[
        -\frac{4 c_4(\alpha \otimes \beta )}{4!} = c_2(\alpha) c_2(\beta) \pmod{c_1(\alpha),c_1(\beta)}.
    \]
    Solving for $c_4(\alpha \otimes \beta)$ gives the desired result.
\end{proof}

Next, we compute $H^*(K(\Z,3);\Z)$ in a small range. We will do this with the Serre spectral sequence (see \cite[Example 5.20]{Hat04}), but one can alternatively apply the universal coefficient theorem to Cartan's computation of $H_*(K(G,n);\mb{Z})$ \cite{Car55}.

\begin{lem}\label{lem:cohomology of k(z,3)}
    In degrees at most 8, the integral cohomology of $K(\Z, 3)$ is
    \[
        H^i(K( \Z, 3) ; \Z) \cong
        \begin{cases}
            \Z & i = 0, 3, \\
            0 & i = 1, 2, 3, 4, 5, 7, \\
            \Z / 2\Z & i = 6, \\
            \Z / 3\Z & i = 8.
        \end{cases}
    \]
\end{lem}
\begin{proof}
    We know that $H^0(K(\Z, 3) ; \Z) \cong \Z$ and $H^1(K( \Z, 3) ; \Z) \cong H^2(K(\Z, 3) ; \Z) \cong 0$ by the Hurewicz theorem and universal coefficient theorem. The path-loop fibration gives us a fiber sequence
    \[
        \begin{tikzcd}
            \arrow[from=1-1, to=1-2]
            \arrow[from=1-2, to=1-3]
            \CP^\infty\cong K(\Z,2) & * & K(\Z, 3).
        \end{tikzcd}
    \]
    The Serre spectral sequence associated with this fibration has signature
    \[
        E_2^{p, q}=H^p( K(\Z, 3) ; H^q( \CP^\infty ; \Z))\Longrightarrow H^{p+q}(*;\Z).
    \]
    Since $H^*(\CP^\infty ; \Z)$ is a polynomial ring generated by an element $x\in H^2(\CP^\infty;\Z)$, the row $E_2^{*, 2k}$ is given by $H^*(K(\Z, 3) ; \Z)$ times $x^k$, and the odd rows $E_2^{*, 2k + 1}$ vanish (see Figure~\ref{fig:k(z,3)}).
    \begin{enumerate}[(i)]
    \item Since the spectral sequence converges to the cohomology of a contractible space, the class $x \in E_2^{0, 2}$ must be nonzero under some differential. Moreover, $H^1(K(\Z,3);\Z)\cong H^2(K(\Z,3);\Z)\cong 0$ implies that $\ker{d_3}\cong\operatorname{coker}{d_3}\cong 0$, so $d_3:E_3^{0,2}\to E_3^{3,0}$ is an isomorphism. In particular, $H^3(K(\Z,3);\Z)\cong\Z$. Denote $\iota:=d_3(x)$, which generates $H^3(K(\Z,3);\Z)$.
    \item Note that $E_2^{4,0}\cong H^4(K(\Z,3);\Z)$ must be 0, because any differential hitting $E_r^{4,0}$ has domain 0 and $E^{4,0}$ converges to $H^4(*;\Z)=0$.
    \item Similarly, the only possible non-zero differential hitting $E_r^{5,0}$ is $d_5:E^{0,4}_5\to E^{5,0}_5$. But $E^{0,4}_2$ is generated by $x^2$, and $d_3(x^2)=2xd_3(x)=2x\iota$ is non-zero and non-torsion. Thus $E^{0,4}_{\geq 4}\cong\ker{d_3}\cong 0$, so $H^5(K(\Z,3);\Z)\cong E^{5,0}\cong 0$.
    \item If $H^6(K(\Z,3);\Z)\cong 0$, then we would have $\ker(d_3:E^{3,2}_3\to E^{6,0}_3)=E^{3,2}_3$. We have already seen that $\mr{im}(d_3:E^{0,4}_3\to E^{3,2}_3)=2x\iota$, so it would follow that $E^{3,2}_4$ is isomorphic to $\Z/2\Z$ (generated by $x\iota$). But now there are no other differentials hitting $E_r^{3,2}$, which would imply that this $\Z/2\Z$ survives to $H^5(*;\Z)$.
    
    By contradiction, we deduce that $y:=d_3(x\iota)\in H^6(K(\Z,3);\Z)$ is non-zero. But $2y=d_3(2x\iota)=0$, as $2x\iota=d_3(x^2)$. Thus $E^{6,0}_3$ contains a $\Z/2\Z$ subgroup, and $E^{6,0}_4$ is the quotient by this $\Z/2\Z$. There are no other differentials with non-zero domain hitting $E^{6,0}$, so we conclude that $E^{6,0}_4\cong 0$ and thus $E^{6,0}_3\cong H^6(K(\Z,3);\Z)\cong\Z/2\Z$.
    \item Degree 7 is analogous to degree 5. The only possible non-zero differential is $d_7:E^{0,6}_7\to E^{7,0}_7$, but $E^{0,6}_3$ is generated by $x^3$. Since $d_3(x^3)=3x^2\iota$ is non-zero and non-torsion, we find that $E^{0,6}_{\geq 4}\cong\ker{d_3}\cong 0$ and hence $H^7(K(\Z,3);\Z)\cong E^{7,0}\cong 0$.
    \item Consider $x^2\iota\in E^{3,4}_3$. Using the product rule, we have
    \begin{align*}
        d_3(x^2\iota)&=xd_3(x\iota)+d_3(x)x\iota\\
        &=xy+x\iota^2\in E^{6,2}_3.
    \end{align*}
    Since $3x^2\iota=d_3(x^3)$, we see that $3d_3(x^2\iota)=3xy+3x\iota^2=0$. Our previous computations and the ring structure on $E^{p,q}$ imply that $E^{6,2}_3$ is isomorphic to $\Z/2\Z$ with generator $xy$, so $x\iota^2=xy$ and $d_3(x^2\iota)=0$. It follows that $E^{3,4}_4\cong\Z/3\Z$ with generator $x^2\iota$.
    
    The only other possible non-zero differential out of $E^{3,4}_2$ is $d_5:E^{3,4}_5\to E^{8,0}_5$. There are no differentials into $E^{3,4}_5$, so $d_5$ is injective. Moreover, there are no other differentials into $E^{8,0}_5$, so $d_5$ is an isomorphism. Thus $E^{8,0}_5\cong H^8(K(\Z,3);\Z)\cong\Z/3\Z$ (whose generator is denoted by $z$ in Figure~\ref{fig:k(z,3)}).\qedhere
    \end{enumerate}
\end{proof}

\begin{figure}
\centering
\begin{sseqdata}[name = example1,title = Page \page,xscale=.5,yscale=.5,class labels = {font=\small},tick style={font = \small},differential labels={font=\tiny}]
    \class["1"](0, 0)
        \class["\iota"](3, 0)
        \class["y"](6, 0)
        \class["z"](8, 0)
        \class["x"](0, 2)
        \class["x \iota"](3, 2)
        \class["x y"](6, 2)
        \classoptions[page = 5, ""]
        \class["x^2"](0, 4)
        \class["x^2 \iota"](3, 4)
        \class["x^3"](0, 6)
        \d3(0, 2)(3, 0)
        \d["\cdot 2"' {pos=.3,yshift=-4pt}]3(0, 4)(3, 2)
        \d["\cdot 3"' {pos=.3,yshift=-4pt}]3(0, 6)(3, 4)
        \replacetarget["x\iota^2"]
        \d3(3, 2)(6, 0)
        \d5(3, 4)(8, 0)
\end{sseqdata}
\printpage[name = example1, page = 3] \quad
\printpage[name = example1, page = 5]
\caption{The $\CP^\infty\to *\to K(\Z,3)$ spectral sequence}
\label{fig:k(z,3)}
\end{figure}

The third fact we need is that the second Pontryagin class of the canonical bundle $\gamma\to\BO\langle 8\rangle$ is $\pm 6\in H^8(\BO\langle 8\rangle;\Z)$.

\begin{lem}\label{lem:p_2 is divisible by 6}
    Let $\gamma$ be the canonical bundle on $\BO\langle 8\rangle$. Then there exists a generator $a\in H^8(\BO\langle 8\rangle;\Z)\cong\Z$ such that $p_2(\gamma)=6a$, where $p_2$ is the second Pontryagin class.
\end{lem}
\begin{proof}
    The space $\BO\langle 8\rangle$ can be obtained as the homotopy fiber of the map $\BSpin \to K( \Z, 4)$ inducing an isomorphism on $\pi_4$. Extending to the left, we get a fiber sequence of spaces
    \[
        \begin{tikzcd}
            \arrow[from=1-1, to=1-2]
            \arrow[from=1-2, to=1-3]
            K(\Z, 3) & \BO\langle 8\rangle & \BSpin.
        \end{tikzcd}
    \]
    The Serre spectral sequence for this fibration has signature
    \[
        E_2^{p, q} = H^p(\BSpin ; H^q(K( \Z, 3) ; \Z))\Longrightarrow H^{p+q}(\BO\langle 8\rangle;\Z)
    \]
    (see Figure~\ref{fig:bo<8>}, in which $\bullet=\Z$, $\circ=\Z/2\Z$, and $\triangleleft=\Z/3\Z$). We computed $H^*(K(\Z,3);\Z)$ for $*\leq 8$ in Lemma~\ref{lem:cohomology of k(z,3)}, which gives us $E^{0,q}_2$ for $q\leq 8$. Let $\iota\in H^3(K(\Z,3);\Z)$ be a generator.

    Next, we need to recall the integral cohomology $\BSpin$ in low degrees, which we can read out of \cite[Theorem 9.1]{Dua19}. These are given by
    \begin{equation}\label{eq:cohomology of bspin}
        H^i(\BSpin;\Z)\cong\begin{cases}
            \Z & i=0,4\\
            0 & i=1,2,3,5,6,\\
            \Z/2\Z & i=7,9,\\
            \Z^2 & i=8.
        \end{cases}
    \end{equation}
    In \textit{loc.~cit.}, it is also shown that there is a generator $\sigma_1\in H^4(\BSpin;\Z)$ such that $2\sigma_1=p_1(\beta)$ (the first Pontryagin class of the canonical bundle $\beta$ on $\BSpin$). Moreover, there is a class $\sigma_2\in H^8(\BSpin;\Z)$ such that $\sigma_1^2,\sigma_2$ freely generate $H^8(\BSpin;\Z)$ and $\sigma_1^2+2\sigma_2=p_2(\beta)$ (the second Pontryagin class).
    
    Because $\BO\langle 8\rangle$ is $7$-connected, its cohomology must vanish in degrees seven and below. The only way for $E_r^{0,3}$ and $E_r^{4,0}$ to die is if $d_4(\iota)=\pm\sigma_1$. Thus $d_4(\iota\sigma_1)=\pm\sigma_1^2$, so quotienting by this image leaves us with $E^{8,0}_5\cong\Z$, which is generated by $\sigma_2$. There are no other differentials into $E_r^{8,0}$, so we find that $H^8(\BO\langle 8\rangle;\Z)\cong\Z$ is generated by $\sigma_2$. Since $\sigma_1^2=0$ in this group, we have the relation $p_2(\gamma)=2\sigma_2$.

\begin{figure}[b]
\centering
\begin{sseqdata}[name = example2,title = Page \page,xscale=.5,yscale=.5,class labels = {font=\small},tick style={font = \small},differential labels={font=\tiny}]
    \class[fill](0, 0)
    \class["\sigma_1"](4, 0)
    \class["\circ"](7, 0)
    \class[fill](8, 0)
    \class[fill,name=\sigma_2,show name = { above, pin }](8, 0)
    %\classoptions[page = 6,"\sigma_2"]
    \class["\iota"](0, 3)
    \class["\iota\sigma_1"](4, 3)
    \class["\circ"](0, 6)
    \class["\triangleleft"](0, 8)
    \d4(0, 3)(4, 0)
    \d4(4, 3)(8, 0)
    \d6(0, 6)(7, 0)
\end{sseqdata}
\printpage[name = example2, page = 4] \quad
\printpage[name = example2, page = 6]
\caption{The $K(\Z,3)\to\BO\langle 8\rangle\to\BSpin$ spectral sequence}
\label{fig:bo<8>}
\end{figure}

    By looking at the low degree cohomology groups of $K(\Z,3)$ and $\BSpin$, we see that the only other group of total degree 8 is $E^{0,8}_\infty\cong H^8(K(\Z,3);\Z)\cong\Z/3\Z$. The only possible non-zero differential out of $E^{0,8}$ is $d_9:E^{0,8}_9\to E^{9,0}_9\cong H^9(\BSpin;\Z)$. However, this differential must be zero, because all the previous differentials hitting $E^{9,0}$ have trivial domain, and all torsion in $H^*(\BSpin;\Z)$ has order 2 \cite[p.~316]{Sto68}. By convergence of this spectral sequence, there is a subgroup $A\subseteq H^8(\BO\langle 8\rangle;\Z)\cong\Z$ such that $E^{8,0}_\infty\cong A$ and $E^{0,8}_\infty\cong H^8(\BO\langle 8\rangle;\Z)/A$. Thus $A$ can be identified with $3\Z\subset\Z\cong H^8(\BO\langle 8\rangle;\Z)$. In particular, $\sigma_2=3a$ for some generator $a\in H^8(\BSpin;\Z)$, so $p_2(\gamma)=2\sigma_2=6a$. 
\end{proof}

Finally, we need to know a little bit about the cohomology of the tensor product maps $\BSO\wedge\BSpin\to\BO\langle 8\rangle$.

\begin{lem}\label{lem:c_4 on 0 spaces}
    Let $\gamma$ be the classifying bundle on $\BO\langle 8\rangle$. Under the product map $\BSO \wedge\BSpin \to\BO\langle 8\rangle$, the image of a generator of a generator of $H^8( \BO\langle 8\rangle ; \Z ) \cong \Z$ under the induced map
    \[H^8(\BO\langle 8\rangle;\Z)\to H^8(\BSO\wedge\BSpin;\Z)\]
    is of the form $2s+t$, where $t$ is a torsion class and the class $s$ has the property that $s+t'$ is not divisible by 2 for any torsion class $t'$.
\end{lem}
\begin{proof}
    Let $\alpha$ and $\beta$ be the classifying bundles on $\BSO$, and $\BSpin$, respectively. Recall that $H^i(\BSO;\Z)\cong 0$ for $1\leq i\leq 3$ and $H^4(\BSO;\Z)\cong\Z$, generated by the first Pontryagin class $p_1(\alpha)$ \cite{Bro82}. Also $H^i(\BSpin;\Z)\cong 0$ for $1\leq i\leq 3$ and $H^4(\BSpin;\Z)\cong\Z$ with generator $\sigma_1$ satisfying $2\sigma_1=p_1(\beta)$ \cite[Theorem~9.1]{Dua19}. The product map $\BSO \wedge \BSpin \to \BO\langle 8\rangle$ has the class $c_4(\gamma_\C)$ mapping to $c_4(\alpha_\C\otimes\beta_\C)$ in cohomology because Chern classes are natural, and this is equal to $-6 c_2(\alpha_\C) c_2(\beta_\C)$ by Lemma~\ref{lem:c_4 of tensor}, as $H^2( \BSO ; \Z) \cong H^2( \BSpin ; \Z) \cong 0$ which forces the first Chern classes to vanish for $\alpha_\C$ and $\beta_\C$. Since $c_2(\alpha_\C)=-p_1(\alpha)$ and $c_2(\beta_\C)=-p_1(\beta)=-2\sigma_1$, we find that $c_4(\gamma_\C)$ maps to $-12p_1(\alpha)\sigma_1$. 
    
    Let $a\in H^8(\BO\langle 8\rangle;\Z)$ be the generator such that $c_4(\gamma_\C)=p_2(\gamma)$ is equal to $6a$, as given by Lemma~\ref{lem:p_2 is divisible by 6}. Then $a$ maps to $-2 p_1(\alpha)\sigma_1+t$ under $H^*(\BO\langle 8\rangle;\Z)\to H^*(\BSO\wedge\BSpin;\Z)$, where $t\in H^8(\BSO\wedge\BSpin;\Z)$ is some element satisfying $6t=0$.
    
    It thus remains to show that $s:=-p_1(\alpha)\sigma_1$ is such that $s+t'$ is not divisible by 2 for any torsion class $t'$. Using the K\"unneth formula for cohomology (see \cite[Proposition VI.12.16]{Dol72}), we get a split short exact sequence
    \[
        \begin{tikzcd}
            \arrow[from=1-1, to=1-2]
            \arrow[from=1-2, to=1-3]
            \arrow[from=1-3, to=1-4]
            \arrow[from=1-4, to=1-5]
            0 & H^4(\BSO;\Z) \otimes H^4(\BSpin;\Z) & H^8(\BSO\wedge\BSpin;\Z) & A & 0
        \end{tikzcd}
    \]
    where $A$ is some $\mathrm{Tor}$ term. So we have a direct sum decomposition
    \[
        H^8(\BSO\wedge\BSpin;\Z) \cong (H^4(\BSO;\Z) \otimes H^4(\BSpin;\Z)) \oplus A.
    \]
    We know that $s$ belongs to the first summand because it is a product of a class of $\BSO$ and a class of $\BSpin$. Also $t'$ must belong to $A$ since it is a torsion class and $H^4(\BSO;\Z) \otimes H^4(\BSpin;\Z) \cong \Z \otimes \Z \cong \Z$ is torsion-free. Since $s$ is a generator of $H^4(\BSO;\Z) \otimes H^4(\BSpin;\Z)$ and $t'$ lives in the other summand, $s+t'$ cannot be a multiple of $2$.
\end{proof}

\subsection{Module structure for suspended elephants}
We are now ready to return to our goal of lifting the multiplication
\begin{equation}\label{eq:mult on ko}
\Diamond:\ko\langle 8k+2\rangle\wedge\ksp\to\ksp\langle 8k+2\rangle
\end{equation}
to $\Sigma^{8k+4}F$. We will first show that such a lift exists, after which we will compute its effect on cohomology.

\begin{prop}\label{prop:mult lifts to F}
    The multiplication map $\Diamond$ of Equation~\ref{eq:mult on ko} lifts to $\Sigma^{8k+4}F$:
    \[\begin{tikzcd}
    & \Sigma^{8k+4}F\arrow[d]\\
    \ko\langle 8k+2\rangle\wedge\ksp\arrow[r,"\Diamond"]\arrow[ur,dashed,"\tilde{\Diamond}"] & \ksp\langle 8k+2\rangle.
    \end{tikzcd}\]
\end{prop}
\begin{proof}
Since $\pi_{8k+2}\KSp\cong\pi_{8k+3}\KSp\cong 0$, we have $\ksp\langle 8k+2\rangle\cong\ksp\langle 8k+4\rangle$. Thus $\Diamond$ induces a map $\ko\langle 8k+2\rangle\wedge\ksp\to\ksp\langle 8k+4\rangle$. On cohomology, this map is determined by the image of $y_{8k+4}$, which generates $H^*\ksp\langle 8k+4\rangle$ (see Corollary~\ref{cor:cohomology of ksp}). The action of $\A_1$ on $y_{8k+4}$ is trivial, since $\Sq^1y_{8k+4}=\Sq^2y_{8k+4}=0$. 

Proposition~\ref{prop:cohomology of ko<k>} and Corollary~\ref{cor:cohomology of ksp} imply that $H^{8k + 4}(\ko\langle 8k+2\rangle\wedge\ksp)$ is generated by $\Sq^2x_{8k+2}\otimes y_0$ and $x_{8k+2}\otimes\Sq^2y_0$, since the cohomology over a field of a smash product is the tensor product of the cohomology of its factors, and the Steenrod algebra acts by the Cartan formula. But the action of $\A_1$ on any non-zero combination of these generators is non-trivial, since
\begin{align*}
    \Sq^1( \Sq^2 x_{8k + 2} \otimes y_0) & = 0, \\
    \Sq^2( \Sq^2 x_{8k + 2} \otimes y_0) & = \Sq^3 \Sq^1 x_{8k + 2} \otimes y_0 + \Sq^2 x_{8k + 2} \otimes \Sq^2 y_0, \\
    \Sq^1( x_{8k + 2} \otimes \Sq^2 y_0) & = \Sq^1 x_{8k + 2} \otimes \Sq^2 y_0 + x_{8k + 2} \otimes \Sq^3 y_0, \\
    \Sq^2( x_{8k + 2} \otimes \Sq^2 y_0) & = \Sq^2 x_{8k + 2} \otimes \Sq^2 y_0 + \Sq^1 x_{8k + 2} \otimes \Sq^3 y_0.
\end{align*}
The action of $\A_1$ on the image of $y_{8k+4}$ must be trivial, so we deduce that $H^*\ksp\langle 8k+4\rangle\to H^*(\ko\langle 8k+2\rangle\wedge\ksp)$ is given by $y_{8k+4}\mapsto 0$.

Due to the fiber sequence given in Equation~\ref{eq:fiber sequence for F}, the map $\ko\langle 8k+2\rangle\wedge\ksp\to\ksp\langle 8k+4\rangle$ lifts to a map $\tilde{\Diamond}:\ko\langle 8k+2\rangle\wedge\ksp\to\Sigma^{8k+4}F$ if the composite
\[\ko\langle 8k + 2\rangle \wedge \ksp \to \ksp\langle 8k + 4\rangle\to \Sigma^{8k + 4} H\Z / 2\Z\]
is nullhomotopic. Since the map $\ksp \langle 8k + 4 \rangle \to \Sigma^{8k + 4} H\Z / 2\Z$ represents $y_{8k + 4}$ and $y_{8k + 4}$ maps to zero in the cohomology of $\ko \langle 8k + 2 \rangle \wedge \ksp$, it follows that this map is nullhomotopic and we get a lift.
\end{proof}

Our next objective is to understand what the map $\tilde{\Diamond}:\ko\langle 8k+2\rangle\wedge\ksp\to\Sigma^{8k+4}F$ does in cohomology. After introducing some notation, we will study this map for $k=0$.

\begin{notn}
    Recall from Lemma~\ref{lem:cohomology of F} that $H^*F\cong\A\otimes_{\A_1}E$ is generated by $\Sq^1$ and $\Sq^2$. We denote these by $e_0:=\Sq^1$ and $e_1:=\Sq^2$.
\end{notn}

\begin{lem}
    Given any lift
    \[\begin{tikzcd}
        & \Sigma^8F\arrow[d]\\
        \ko\langle 2\rangle\wedge\ko\langle 4\rangle\arrow[r,"\Diamond"]\arrow[ur,"\tilde{\Diamond}"] & \ko\langle 8\rangle,
    \end{tikzcd}\]
    the image of $e_0 \in H^8 \Sigma^8 F$ under the induced map $H^8\Sigma^8F\to H^8(\ko\langle 2\rangle\wedge\ko\langle 4\rangle)$ is non-zero.
\end{lem}
\begin{proof}
    Recall that if $\Phi\to E\to B$ is a fibration of spaces, where $\Phi$ is $m$-connected and $B$ is $n$-connected, then there is a Serre exact sequence
    \[0\to H^0(B;\Z)\to H^0(E;\Z)\to H^0(\Phi;\Z)\to\cdots\to H^{m+n+1}(B;\Z).\]
    In particular, if $B$ is $n$-connected, then $\Omega B$ is $(n-1)$-connected, so the path-loop fibration gives us maps $H^i(\Omega B;\Z)\to H^{i+1}(B;\Z)$ that are isomorphisms for $i<2n-1$ and an injection for $i=2n-1$.

    Now if $X$ is a CW spectrum, we can write $X$ as the union of the subspectra generated by each level $X_n$. By \cite[Part III, Proposition 8.1]{Ada74}, we can thus compute $H^*(X;\Z)$ via the Milnor exact sequence
    \[0\to\textstyle\lim^1 H^*(X_n;\Z)\to H^*(X;\Z)\to\lim H^*(X_n;\Z)\to 0.\]
    So if $X$ is a connective $\Omega$-spectrum with $n$-connected $X_0$, then $H^i(X_0;\Z)\to H^i(X;\Z)$ is an isomorphism for $i\leq 2n$ by the Serre exact sequence above, as the connectivity of each loop space will always be at least $n$, so the maps making up the limit are all isomorphisms in this range by the Serre exact sequence. The zeroth spaces of $\ko\langle 8\rangle$, $\ko\langle 4\rangle$, and $\ko\langle 2\rangle$ are $\BO\langle 8\rangle$, $\BSpin$ and $\BSO$, respectively. Since $\BO\langle 8\rangle$ is 7-connected (by construction) and $\BSpin$ is 3-connected (as $\Spin$ is 2-connected), we have isomorphisms
    \begin{align*}
        H^i(\ko\langle 8\rangle;\Z)&\cong H^i(\BO\langle 8\rangle;\Z) & (\text{for }i\leq 14),\\
        H^j(\ko\langle 4\rangle;\Z)&\cong H^j(\BSpin;\Z) &  (\text{for }j\leq 6).
    \end{align*}
    For $\BSO$ and $\ko\langle 2 \rangle$ we need more care because $\BSO$ is only 1-connected. However, since the integral cohomology of $\BSO$ is trivial in degrees less than four, if $\Phi$ is a delooping of $\BSO$, we have a fiber sequence $\Phi \to * \to \BSO$. The Serre spectral sequence of this fiber sequence implies that the integral cohomology of $\Phi$ is trivial for degrees less than five and the transgression $d_5 : E_5^{0, 4} \to E_5^{5, 0}$ is an isomorphism. Similarly, for any higher delooping of $\BSO$, the transgression must be an isomorphism on these bottom cohomology groups for the same reason, and therefore
    \begin{align*}
        H^j(\ko\langle 2\rangle;\Z)&\cong H^j(\BSO;\Z) &  (\text{for }j\leq 4).
    \end{align*}
    It follows that the generators
    \begin{align*}
        a&\in H^8(\BO\langle 8\rangle;\Z),\\
        \sigma_1&\in H^4(\BSpin;\Z),\\
        p_1(\alpha)&\in H^4(\BSO;\Z)
    \end{align*}
    determine generators in $H^8(\ko\langle 8\rangle;\Z)$, $H^4(\ko\langle 4\rangle;\Z)$, and $H^4(\ko\langle 2\rangle;\Z)$, respectively.

    Using one of Adams's models of spectra and smash products \cite{Ada74}, the zeroth space of $\ko\langle 2\rangle\wedge\ko\langle 4\rangle$ is $\BSO\wedge\BSpin$, which is at least $(1+3+1)$-connected. Thus $H^i(\ko\langle 2\rangle\wedge\ko\langle 4\rangle;\Z)\cong H^i(\BSO\wedge\BSpin;\Z)$ for $i\leq 10$. It now follows from Lemma~\ref{lem:c_4 on 0 spaces} that any generator of $H^*(\ko\langle 8\rangle;\Z)$ is sent to an element of the form $2s+t\in H^8(\ko\langle 2\rangle\wedge\ko\langle 4\rangle;\Z)$, where $t$ is a torsion class and $s$ has the property that $s+t'$ is not a multiple of two for any torsion class $t'$.

    Finally, either generator of $\pi_8\Sigma^8 F\cong\Z$ is sent to twice a generator of $\pi_8\ko\langle 8\rangle\cong\Z$ under the map induced by $\Sigma^8 F\to\ko\langle 8\rangle$ since we defined $F \to \ko$ to be the inclusion $2\Z \to \Z$. So the Hurewicz theorem implies that in homology, either generator of $H_8(\Sigma^8 F;\Z)\cong\Z$ maps to twice a generator of $H_8(\ko\langle 8\rangle;\Z)\cong\Z$. By the universal coefficient theorem, either generator of $H^8(\ko\langle 8\rangle;\Z)\cong\Z$ goes to twice a generator of $H^8(\Sigma^8 F;\Z)$. By assumption, the map $\ko\langle 2\rangle\wedge\ko\langle 4\rangle\to\ko\langle 8\rangle$ factors as
    \[
        \begin{tikzcd}
            \arrow[from=2-1, to=1-2]
            \arrow[from=1-2, to=2-2]
            \arrow[from=2-1, to=2-2]
            & \Sigma^8 F \\
            \ko\langle 2\rangle \wedge \ko\langle 4\rangle & \ko\langle 8\rangle,
        \end{tikzcd}
    \]
    so the map $H^8(\ko\langle 8\rangle;\Z)\to H^8(\ko\langle 2\rangle\wedge\ko\langle 4\rangle;\Z)$ must factor through $H^8(\Sigma^8 F;\Z)\to H^8(\ko\langle 2\rangle\wedge\ko\langle 4\rangle;\Z)$. So if $a\in H^8(\ko\langle 8\rangle;\Z)$ is a generator, then we have
    \[\begin{tikzcd}[row sep = tiny]
        H^8(\ko\langle 8\rangle;\Z)\arrow[r] & H^8(\Sigma^8 F;\Z)\arrow[r] & H^8(\ko\langle 2\rangle\wedge\ko\langle 4\rangle;\Z)\\
        a\arrow[r,maps to] & 2b\arrow[r,maps to] & 2s+t,
    \end{tikzcd}\]
    where $b\in H^8(\Sigma^8F;\Z)$ is a generator. Thus $b\mapsto s+t'$, where $s$ is not divisible by 2 and $t'$ is a torsion class. The mod 2 reduction of $b$ is $e_0\in H^8\Sigma^8F$, and the mod 2 reduction of $s+t'$ is a non-zero element of $H^8(\ko\langle 2\rangle\wedge\ko\langle 4\rangle)$ since it is not a multiple of two.
\end{proof}

Now that we know that the image of $e_0$ is non-zero, we can explicitly determine what value this image takes.

\begin{lem}\label{lem:image of e_0}
    Any lift $\ko\langle 2\rangle \wedge \ko\langle 4\rangle \to \Sigma^8 F$ induces $e_0 \mapsto \Sq^2 x_2 \otimes x_4$ in cohomology.
\end{lem}
\begin{proof}
    Since $H^* \ko\langle 2\rangle \cong \A / \A \Sq^3$ and $H^* \ko\langle 4\rangle \cong \A /(\A \Sq^1 + \A \Sq^5)$, we can read off the possible images that $e_0$ might have. The $\A$-module $H^* \ko\langle 2\rangle$ is generated by $x_2$ in degree two, $\Sq^1 x_2$ in degree three, and $\Sq^2 x_2$ in degree four, and $H^* \ko\langle 4\rangle$ has $x_4$ in degree four, nothing in degree five, and $\Sq^2 x_4$ in degree six.
    
    So we know that $e_0 \mapsto A \Sq^2 x_2 \otimes x_4 + B x_2 \otimes \Sq^2 x_4$ for some $A,B\in\Z/2\Z$ with $(A,B)\neq(0,0)$. However, note that $\Sq^1 e_0 = 0$ and
    \begin{align*}
        &\Sq^1(A \Sq^2 x_2 \otimes x_4 + B x_2 \otimes \Sq^2 x_4)\\
        =\ & A \Sq^3 x_2 \otimes x_4 + A \Sq^2 x_2 \otimes \Sq^1 x_4 + B \Sq^1 x_2 \otimes \Sq^2 x_4 + B x_2 \otimes \Sq^3 x_4 \\
        =\ & B \Sq^1 x_2 \otimes \Sq^2 x_4 + B x_2 \otimes \Sq^3 x_4,
    \end{align*}
    This is only zero if $B$ is zero, so $H^*\Sigma^8 F\to H^*(\ko\langle 2\rangle\wedge\ko\langle 4\rangle)$ is a homomorphism of $\A$-modules if and only if $A = 1$ and $B = 0$. Thus $e_0 \mapsto \Sq^2 x_2 \otimes x_4$.
\end{proof}

So far, we know that $\ko\langle 2\rangle\wedge\ko\langle 4\rangle\to\ko\langle 8\rangle$ lifts to $\ko\langle 2\rangle\wedge\ko\langle 4\rangle\to\Sigma^8 F$, and we know the image of $e_0\in H^8\Sigma^8F$ under any such lift. Next, we show that there is a lift such that $e_1\mapsto x_2\otimes\Sq^3x_4$. This is a key property that the maps in the splitting must have in order to get an isomorphism in cohomology with $\Z / 2\Z$ coefficients later.

\begin{lem}\label{lem:lift at k=0}
    There exists a lift $\ko\langle 2 \rangle \wedge \ko\langle 4\rangle \to \Sigma^8 F$ such that $e_1 \mapsto x_2 \otimes \Sq^3 x_4$.
\end{lem}
\begin{proof}
    We use obstruction theory to obtain a lift with the desired properties. Since $H^* \ko\langle 2\rangle$ has $\Sq^2 \Sq^1 x_2$ in degree five and $H^* \ko\langle 4\rangle$ has $\Sq^3 x_4$ in degree seven, we have
    \[
        e_1 \mapsto C \Sq^2 \Sq^1 x_2 \otimes x_4 + D \Sq^1 x_2 \otimes \Sq^2 x_4 + E x_2 \otimes \Sq^3 x_4
    \]
    for $C,D,E\in\Z/2\Z$. We can eliminate some possibilities using the relation $\Sq^2 e_1 = \Sq^3 e_0$. On the right hand side, Lemma~\ref{lem:image of e_0} implies that
    \begin{align*}
        & \Sq^3( \Sq^2 x_2 \otimes x_4) \\
        =\ & \Sq^3 \Sq^2 x_2 \otimes x_4 + \Sq^2 \Sq^2 x_2 \otimes \Sq^1 x_4 + \Sq^1 \Sq^2 x_2 \otimes \Sq^2 x_4 + \Sq^2 x_2 \otimes \Sq^3 x_4 \\
        =\ & \Sq^3 x_2 \otimes \Sq^2 x_4 + \Sq^2 x_2 \otimes \Sq^3 x_4 \\
        =\ & \Sq^2 x_2 \otimes \Sq^3 x_4.
    \end{align*}
    For the left hand side, we compute
    \begin{align*}
        & \Sq^2(C \Sq^2 \Sq^1 x_2 \otimes x_4 + D \Sq^1 x_2 \otimes \Sq^2 x_4 + E x_2 \otimes \Sq^3 x_4) \\
        =\ & C \Sq^2 \Sq^2 \Sq^1 x_2 \otimes x_4 + C \Sq^1 \Sq^2 \Sq^1 x_2 \otimes \Sq^1 x_4 + C \Sq^2 \Sq^1 x_2 \otimes \Sq^2 x_4 \\
        & + D \Sq^2 \Sq^1 x_2 \otimes \Sq^2 x_4 + D \Sq^1 \Sq^1 x_2 \otimes \Sq^1 \Sq^2 x_4 + D \Sq^1 x_2 \otimes \Sq^2 \Sq^2 x_4 \\
        & + E \Sq^2 x_2 \otimes \Sq^3 x_4 + E \Sq^1 x_2 \otimes \Sq^1 \Sq^3 x_4 + E x_2 \otimes \Sq^2 \Sq^3 x_4 \\
        =\ & C \Sq^2 \Sq^1 x_2 \otimes \Sq^2 x_4 + D \Sq^2 \Sq^1 x_2 \otimes \Sq^2 x_4 + E \Sq^2 x_2 \otimes \Sq^3 x_4.
    \end{align*}
    In order for this to equal $\Sq^2x_2\otimes\Sq^3x_4$, we must have $E = 1$ and either $C = D = 0$ or $C = D = 1$. We are done if $C = D = 0$, so we may assume $C = D = 1$. Let $f : \ko\langle 2\rangle \wedge \ko\langle 4\rangle \to \Sigma^8 F$ be the lift inducing
    \[e_1\mapsto\Sq^2\Sq^1x_2\otimes x_4+\Sq^1x_2\otimes\Sq^2x_4+x_2\otimes\Sq^3x_4.\]
    Rotating the fiber sequence given in Equation~\ref{eq:fiber sequence for F}, we see that there is a fiber sequence
    \begin{equation}\label{eq:shifted fiber sequence}
        \begin{tikzcd}
            \arrow[from=1-1, to=1-2]
            \arrow[from=1-2, to=1-3]
            \Sigma^7 H\Z / 2\Z & \Sigma^8 F & \ko\langle 8\rangle.
        \end{tikzcd}
    \end{equation}
    The image of $H^* \ko\langle 8\rangle \to H^* \Sigma^8 F$ is zero by Proposition~\ref{prop:cohomology of ko<k>} and Lemma~\ref{lem:cohomology of F}, so the long exact sequence associated to Equation~\ref{eq:shifted fiber sequence} implies that the map $H^* \Sigma^8 F \to H^* \Sigma^7 H\Z / 2\Z$ is injective. This forces the generators $e_0$ and $e_1$ to map to $\Sq^1$ and $\Sq^2$, respectively, and hence completely determines the map $H^*\Sigma^8F\to H^*\Sigma^7 H\Z/2\Z$.
    
    Now consider the map $\ko\langle 2\rangle \wedge \ko\langle 4\rangle \to \Sigma^7 H\Z / 2\Z$ classifying $\Sq^1 x_2 \otimes x_4 \in H^7(\ko\langle 2\rangle \wedge \ko\langle 4\rangle)$. Composing with the map $\Sigma^7 H\Z / 2\Z \to \Sigma^8 F$ from Equation~\ref{eq:shifted fiber sequence}, we get a map $g : \ko\langle 2\rangle \wedge \ko\langle 4\rangle \to \Sigma^8 F$ such that
    \begin{align*}
        g^* e_0 & = \Sq^1(\Sq^1 x_2 \otimes x_4) \\
        & = \Sq^1 \Sq^1 x_2 \otimes x_4 + \Sq^1 x_2 \otimes \Sq^1 x_4 \\
        & = 0, \\
        g^* e_1 & = \Sq^2(\Sq^1 x_2 \otimes x_4) \\
        & = \Sq^2 \Sq^1 x_2 \otimes x_4 + \Sq^1 \Sq^1 x_2 \otimes \Sq^1 x_4 + \Sq^1 x_2 \otimes \Sq^2 x_4 \\
        & = \Sq^2 \Sq^1 x_2 \otimes x_4 + \Sq^1 x_2 \otimes \Sq^2 x_4.
    \end{align*}
    Since the composite
    \begin{equation}\label{eq:composite with g}
    \ko\langle 2\rangle\wedge\ko\langle 4\rangle\xrightarrow{g}\Sigma^8F\to\ko\langle 8\rangle
    \end{equation}
    factors through the fiber sequence given in Equation~\ref{eq:shifted fiber sequence} (by the definition of $g$),
    a nullhomotopy of the fiber sequence yields a nullhomotopy of Equation~\ref{eq:composite with g}. Since $f:\ko\langle 2\rangle\wedge\ko\langle 4\rangle\to\Sigma^8F$ is a lift of the product map $\ko\langle 2\rangle\wedge\ko\langle 4\rangle\to\ko\langle 8\rangle$, so is the sum $f+g:\ko\langle 2\rangle\wedge\ko\langle 4\rangle\to\Sigma^8F$. In cohomology, we compute
    \begin{align*}
        (f + g)^* e_0 & = f^* e_0 + g^* e_0 \\
        & = \Sq^2 x_2 \otimes x_4 + 0 \\
        & = \Sq^2 x_2 \otimes x_4, \\
        (f + g)^* e_1 & = f^* e_1 + g^* e_1 \\
        & = (\Sq^2 \Sq^1 x_2 \otimes x_4 + \Sq^1 x_2 \otimes \Sq^2 x_4 + x_2 \otimes \Sq^3 x_4) \\
        & \quad + (\Sq^2 \Sq^1 x_2 \otimes x_4 + \Sq^1 x_2 \otimes \Sq^2 x_4) \\
        & = x_2 \otimes \Sq^3 x_4.
    \end{align*}
    Thus $f+g$ is the desired lift of $\ko\langle 2\rangle\wedge\ko\langle 4\rangle\to\ko\langle 8\rangle$.
\end{proof}

We now suspend this lift at $k=0$ to obtain the desired lift for all $k$.

\begin{lem}\label{lem:lift to F}
    There is a lift $\ko\langle 8k + 2\rangle \wedge \ksp \to \Sigma^{8k+4}F$ of the multiplication map $\ko\langle 8k + 2\rangle \wedge \ksp \to \ksp\langle 8k + 2\rangle$ such that $e_0 \mapsto \Sq^2 x_{8k + 2} \otimes y_0$ and $e_1 \mapsto x_{8k + 2} \otimes \Sq^3 y_0$ in cohomology.
\end{lem}
\begin{proof}
    Consider the product map $\ko\langle 2\rangle \wedge \ksp \to \ksp\langle 4\rangle$, which is the unique top arrow making the square
    \[
        \begin{tikzcd}
            \arrow[from=1-1, to=1-2]
            \arrow[from=1-1, to=2-1]
            \arrow[from=1-2, to=2-2]
            \arrow[from=2-1, to=2-2]
            \ko \langle 2 \rangle \wedge \ksp & \ksp \langle 4 \rangle \\
            \KO \wedge \KSp & \KSp
        \end{tikzcd}
    \]
    commute (using Diagram~\ref{eq:truncated module structure} and $\ksp\langle 2\rangle\cong\ksp\langle 4\rangle$). Suspending four times, we get the commutative diagram
    \[
        \begin{tikzcd}
            \arrow[from=1-1, to=1-2]
            \arrow[from=1-1, to=2-1]
            \arrow[from=1-2, to=2-2]
            \arrow[from=2-1, to=2-2]
            \ko\langle 2\rangle \wedge \ko\langle 4\rangle & \ko\langle 8\rangle \\
            \KO \wedge \KO & \KO.
        \end{tikzcd}
    \]
    Using the isomorphism $\KO \cong \Sigma^4 \KSp$ as $\KO$-module spectra, the bottom arrow is the $\KO$ multiplication map. Thus the top arrow is the product map $\ko \langle 2 \rangle \wedge \ko \langle 4 \rangle \to \ko \langle 8 \rangle$ appearing in Lemma~\ref{lem:lift at k=0}. Now let 
    \begin{equation}\label{eq:lift in ko}
    \begin{tikzcd}
    & \Sigma^8F\arrow[d]\\
    \ko\langle 2\rangle\wedge\ko\langle 4\rangle\arrow[r]\arrow[ur] & \ko\langle 8\rangle
    \end{tikzcd}
    \end{equation}
    be a lift such that $e_0\mapsto\Sq^2x_2\otimes x_4$ and $e_1\mapsto x_2\otimes\Sq^3x_4$ in cohomology. Then the fourfold desuspension
    \begin{equation}\label{eq:lift in ksp}
        \begin{tikzcd}
            \arrow[from=2-1, to=1-2]
            \arrow[from=1-2, to=2-2]
            \arrow[from=2-1, to=2-2]
            & \Sigma^4 F \\
            \ko \langle 2 \rangle \wedge \ksp & \ksp \langle 4 \rangle
        \end{tikzcd}
    \end{equation}
     of Diagram~\ref{eq:lift in ko} satisfies $e_0\mapsto\Sq^2x_2\otimes y_0$ and $e_1\mapsto x_2\otimes\Sq^3y_0$ in cohomology, since $y_0$ is the fourfold desuspension of $x_4$ (see Proposition~\ref{prop:cohomology of ko<k>} and Corollary~\ref{cor:cohomology of ksp}). Now we suspend Diagram~\ref{eq:lift in ksp} another $8k$ times to get the diagram
    \[
        \begin{tikzcd}
            \arrow[from=2-1, to=1-2]
            \arrow[from=1-2, to=2-2]
            \arrow[from=2-1, to=2-2]
            \arrow[from=2-1, to=3-1]
            \arrow[from=2-2, to=3-2]
            \arrow[from=3-1, to=3-2]
            & \Sigma^{8k + 4} F \\
            \ko \langle 8k + 2 \rangle \wedge \ksp & \ksp \langle 8k + 4 \rangle \\
            \Sigma^{8k} \KO \wedge \KSp & \Sigma^{8k} \KSp.
        \end{tikzcd}
    \]
    Indeed, the bottom arrow is still the product map because $\Sigma^8 \KSp \cong \KSp$ as $\KO$-modules. So the top horizontal arrow is still the product. The map in cohomology induced by $\ko \langle 8k + 2 \rangle \wedge \ksp \to \Sigma^{8k + 4} F$ is given by $e_0 \mapsto \Sq^2 x_{8k + 2} \otimes y_0$ and $e_1 \mapsto x_{8k + 2} \otimes \Sq^3 y_0$.
\end{proof}

\subsection{Defining the $\KSp$-Pontryagin and elephant classes}
Our next goal is to give maps $\MSpinh\to\ksp\langle 8k\rangle$ and $\MSpinh\to\Sigma^{8k+4}F$ that will constitute some of the summands in the 2-local splitting of $\MSpinh$. These maps arise from characteristic classes associated to $\ksp\langle 8k\rangle$ and $\Sigma^{8k+4}F$.

\begin{setup}\label{setup:ksp classes}
    For each partition $I=(i_1,\ldots,i_r)$, there is a $\KO$-Pontryagin class $\pi^I_h\in\KO^0(\BSpinh)$, obtained by pulling back the $\KO$-Pontryagin class $\pi^I\in\KO^0(\BSO)$ under $\BSpinh\to\BSO$. The associated map $\BSpinh\to\KO$ admits a lift to $\ko\langle 4|I|\rangle$ if $I\in\Part_\mr{even}$ or to $\ko\langle 4|I|-2\rangle$ if $I\in\Part_\mr{odd}$ by Proposition~\ref{prop:lift of pontryagin class}.

    Let $\vphi^h$ be the Atiyah--Bott--Shapiro map (Section~\ref{sec:abs map}), and let $\KSp\to\ksp$ denote the canonical map to the connective cover. Smashing the map $\BSpinh\to\ko\langle n\rangle$ with $\MSpinh\xrightarrow{\vphi^h}\KSp\to\ksp$, we get a map of the form
    \begin{equation}\label{eq:smash with ABS}
    \BSpinh\wedge\MSpinh\to\ko\langle n\rangle\wedge\ksp,
    \end{equation}
    where $n=4|I|$ or $4|I|-2$ (depending on whether $|I|$ is even or odd). We now precompose Equation~\ref{eq:smash with ABS} with the Thom diagonal $\MSpinh\to\BSpinh\wedge\MSpinh$ and postcompose with the multiplication $\ko\langle 4|I|\rangle\wedge\ksp\to\ksp\langle 4|I|\rangle$ (given in Lemma~\ref{lem:8k multiplication}) or with the lift $\ko\langle 4|I|-2\rangle\wedge\ksp\to\Sigma^{4|I|}F$ (given in Lemma~\ref{lem:lift to F}) of the multiplication $\ko\langle 4|I|-2\rangle\wedge\ksp\to\ko\langle 4|I|-2\rangle$. 
    
    When $I\in\Part_\mr{even}$, the composite takes the form
    \begin{equation}\label{eq:ksp class}
    \MSpinh\to\BSpinh\wedge\MSpinh\to\ko\langle 4|I|\rangle\wedge\ksp\to\ksp\langle 4|I|\rangle.
    \end{equation}
    When $I\in\Part_\mr{odd}$, the composite takes the form
    \begin{equation}\label{eq:elephant class}
    \MSpinh\to\BSpinh\wedge\MSpinh\to\ko\langle 4|I|-2\rangle\wedge\ksp\to\Sigma^{4|I|}F.
    \end{equation}
\end{setup}

\begin{defn}\label{def:ksp classes}
    Given an even partition $I$, the \textit{$I^\text{th}$ $\KSp$-Pontryagin class} is the class $\kappa^I\in\ksp\langle 4|I|\rangle^0(\MSpinh)$ determined by Equation~\ref{eq:ksp class}. Given an odd partition $I$, the \textit{$I^\text{th}$ elephant class} is the class $\eps^I\in\Sigma^{4|I|}F^0(\MSpinh)$ determined by Equation~\ref{eq:elephant class}. We refer to $\kappa^I$ and $\eps^I$ collectively as \textit{$\KSp$-characteristic classes}.
\end{defn}

\begin{rem}\label{rem:odd ksp-pontryagin}
    When $I$ is an odd partition, we still have a map
    \[\BSpinh\wedge\MSpinh\to\ko\langle 4|I|\rangle\wedge\ksp\to\ksp\langle 4|I|\rangle\]
    coming from Diagram~\ref{eq:ksp-pontryagin}. In particular, we have $I^\textsuperscript{th}$ $\KSp$-Pontryagin classes $\kappa^I$ for odd partitions as well, although we have not computed their effect on cohomology. These classes will not be needed for Theorem~\ref{thm:main}, but they will become relevant in Section~\ref{sec:characteristic classes}.
\end{rem}

We wish to compute the maps on cohomology induced by $\kappa^I$ and $\eps^I$. To this end, we need to compute $\MSpinh\to\ksp$ in cohomology.

\begin{lem}\label{lem:abs in cohomology}
    Let $\vphi^h:\MSpinh\to\KSp$ be the Atiyah--Bott--Shapiro map. Then the composite $\MSpinh\xrightarrow{\vphi^h}\KSp\to\ksp$ induces the map
    \begin{align*}
        H^*\ksp&\to H^*\MSpinh\\
        y_0&\mapsto U_h,
    \end{align*}
    where $U_h\in H^*\MSpin$ denotes the Thom class.
\end{lem}
\begin{proof}
    By \cite[Theorem 3.23]{Hu22}, the induced map $\pi_0 \MSpinh \to \pi_0 \KSp$ is surjective. But $\pi_0 \MSpinh \cong \Z$ (see \cite[Theorem 9.97]{FH21}) and $\pi_0 \KSp \cong \Z$, so this must be an isomorphism. By the Hurewicz theorem, this means that $\MSpinh \to \ksp$ must also give an isomorphism in degree zero integral homology, and then in degree zero mod 2 homology by reduction. Dualizing, we see that the map $\MSpinh \to \ksp$ induces an isomorphism in mod 2 cohomology in degree zero, and therefore $y_0 \mapsto U_h$.
\end{proof}

\begin{rem}
    The real and complex analogs of Lemma~\ref{lem:abs in cohomology} can be proved by utilizing the fact that the real and complex Atiyah--Bott--Shapiro maps $\vphi^r$ and $\vphi^c$ are orientations and therefore preserve units. We have no such guarantee quaternionic Atiyah--Bott--Shapiro map $\vphi^h$, but it is plausible that the $\vphi^r$-module structure of $\vphi^h$ enables a more conceptually parsimonious proof than the one we found.
\end{rem}

Now we can compute the maps on cohomology induced by $\kappa^I$ and $\eps^I$.

\begin{prop}\label{prop:ksp/elephant class in cohomology}
Given a partition $I$, let $p_I\in H^*\BSpinh$ denote the corresponding Pontryagin class. Let $U_h\in H^*\MSpinh$ denote the Thom class. If $I\in\Part_\mr{even}$, then the map $H^*\ksp\langle 4|I|\rangle\to H^*\MSpinh$ induced by $\kappa^I$ is given by 
\[y_{8k}\mapsto(p_I+\Sq^3\Sq^1\alpha_I)U_h,\]
where $\alpha_I\in H^*\BSpinh$ is some class.

If $I\in\Part_\mr{odd}$, then the map $H^*\Sigma^{4|I|}F\to H^*\MSpinh$ induced by $\eps^I$ is given by
\begin{align*}
e_0&\mapsto p_IU_h,\\
e_1&\mapsto \beta_Iw_3U_h,
\end{align*}
where $\beta_I\in H^*\BSpinh$ is some class satisfying $\Sq^2\beta_I=p_I$.
\end{prop}
\begin{proof}
    Recall that the Thom diagonal $\MSpinh\to\BSpinh\wedge\MSpinh$ induces $a\otimes U_h\mapsto aU_h$ in cohomology. Since we have characterized the image of $y_{8k}$ under $H^*\ksp\langle 4|I|\rangle\to H^*(\ko\langle 4|I|\rangle\wedge\ksp)$ (Lemma~\ref{lem:8k multiplication}) and the images of $e_0$ and $e_1$ under $H^*\Sigma^{4|I|}F\to H^*(\ko\langle 4|I|-2\rangle\wedge\ksp)$ (Lemma~\ref{lem:lift to F}), it suffices to show that
    \begin{align*}
        H^*(\ko\langle 4|I|\rangle\wedge\ksp)&\to H^*(\BSpinh\wedge\MSpinh)\\
        x_{4|I|}\otimes y_0&\mapsto (p_I+\Sq^3\Sq^1\alpha_I)\otimes U_h
    \end{align*}
    for $I\in\Part_\mr{even}$ and
    \begin{align*}
        H^*(\ko\langle 4|I|-2\rangle\wedge\ksp) &\to H^*(\BSpinh\wedge\MSpinh)\\
        \Sq^2x_{4|I|-2}\otimes y_0&\mapsto p_I\otimes U_h,\\
        x_{4|I|-2}\otimes\Sq^3y_0 &\mapsto \beta_Iw_3\otimes U_h
    \end{align*}
    for $I\in\Part_\mr{odd}$. These are the maps induced by lifting $\KO$-Pontryagin classes and smashing with the Atiyah--Bott--Shapiro map (see Setup~\ref{setup:ksp classes}). By \cite[p.~304]{Sto68}, these lifts of $\KO$-Pontryagin classes induce $x_{4|I|} \mapsto p_I + \delta \Sq^2 \Sq^1 \alpha_I$ and $\Sq^2x_{4|I|-2}\mapsto p_I$ in integral cohomology,\footnote{We abuse notation by denoting integral classes and their mod 2 reductions by the same symbols.} where $\delta$ is the Bockstein. Let $\rho_2:H^*(-;\Z)\to H^*(-;\Z/2\Z)$ denote mod 2 reduction. Since $\rho_2\circ\delta=\Sq^1$, we find that $\rho_2(\delta \Sq^2 \Sq^1 \alpha_I)=\Sq^3\Sq^1\alpha_I$. Thus $x_{4|I|}\mapsto p_I+\Sq^3\Sq^1\alpha_I$
    for $I\in\Part_\mr{even}$.
    
    For $I\in\Part_\mr{odd}$, we have $\Sq^2x_{4|I|-2}\mapsto p_I$ in integral cohomology and hence $x_{4|I|-2}\mapsto\beta_I\in H^*\BSpinh$. Lemma~\ref{lem:abs in cohomology} states that $\MSpinh\to\ksp$ induces $y_0\mapsto U_h$ in cohomology, so we have $\Sq^3y_0\mapsto \Sq^3U_h=w_3 U_h$ (by definition of the Stiefel--Whitney classes). Thus
    \begin{align*}
        x_{4|I|}\otimes y_0&\mapsto (p_I+\Sq^3\Sq^1\alpha_I)\otimes U_h,\\
        \Sq^2x_{4|I|-2}\otimes y_0&\mapsto p_I\otimes U_h,\\
        x_{4|I|-2}\otimes\Sq^3y_0&\mapsto\beta_I\otimes\Sq^3U_h=\beta_Iw_3\otimes U_h,
    \end{align*}
    as desired.
\end{proof}

Adding together the various $\KSp$-Pontryagin classes and elephant classes gives us the first part of our eventual 2-local splitting of $\MSpinh$.

\begin{prop}
    There exists a map
    \[
        \MSpinh \to \bigvee_{I \in \Part_{\mathrm{even}}} \ksp \langle 4 | I | \rangle \vee \bigvee_{I \in \Part_{\mathrm{odd}}} \Sigma^{4 | I |} F
    \]
    such that $y_{8k} \mapsto (p_I + \Sq^3 \Sq^1 \alpha_I) U_h$ for some $\alpha_I \in H^* \BSpinh$ (when $I\in\Part_\mr{even})$ and $e_0 \mapsto p_I U_h$ and $e_1 \mapsto \beta_I w_3 U_h$ for some $\beta_I \in H^* \BSpinh$ satisfying $\Sq^2 \beta_I = p_I$ (when $I\in\Part_\mr{odd}$).
\end{prop}
\begin{proof}
    Taking the product of $\kappa^I$ and $\eps^I$ over all partitions gives us a map
    \[\MSpinh\to\prod_{I\in\Part_\mr{even}}\ksp\langle 4|I|\rangle\times\prod_{I\in\Part_\mr{odd}}\Sigma^{4|I|}F.\]
    Since there are only finitely many factors of this product with non-zero homotopy groups in a given degree, the map
    \[\bigvee_{I\in\Part_\mr{even}}\ksp\langle 4|I|\rangle\vee\bigvee_{I\in\Part_\mr{odd}}\Sigma^{4|I|}F\to\prod_{I\in\Part_\mr{even}}\ksp\langle 4|I|\rangle\times\prod_{I\in\Part_\mr{odd}}\Sigma^{4|I|}F\]
    is an equivalence. This induces the desired map
    \[\MSpinh\to\bigvee_{I\in\Part_\mr{even}}\ksp\langle 4|I|\rangle\vee\bigvee_{I\in\Part_\mr{odd}}\Sigma^{4|I|}F.\]
    The effect of this map on cohomology follows from Proposition~\ref{prop:ksp/elephant class in cohomology}.
\end{proof}

\section{Margolis homology of $H^* \MSpinh$}\label{sec:margolis homology}
In the preceding section, we constructed a map
\[
    \MSpinh \to \bigvee_{I \in \Part_{\mathrm{even}}} \ksp \langle 4|I| \rangle \vee \bigvee_{I \in \Part_{\mathrm{odd}}} \Sigma^{4|I|} F.
\]
The induced map on mod 2 cohomology takes the form
\begin{equation}\label{eq:theta bar}\bar{\theta}:\bigoplus_{I\in\Part_\mr{even}}\Sigma^{4|I|}(\A\otimes_{\A_1}\udq)\oplus\bigoplus_{I\in\Part_\mr{odd}}\Sigma^{4|I|}(\A\otimes_{\A_1}E)\to H^*\MSpinh,
\end{equation}
where $\udq$ and $E$ denote the upside-down question mark (Definition~\ref{def:upside down question mark} and Corollary~\ref{cor:cohomology of ksp}) and the elephant (Definition~\ref{def:elephant} and Lemma~\ref{lem:cohomology of F}), respectively. Denote by $\bar{\theta}$ the homomorphism in Equation~\ref{eq:theta bar}. 

In this section, we will show that $\bar{\theta}$ induces isomorphisms in Margolis homology, analogous to a method used in \cite{ABP67}. This will be used as a key input in Section~\ref{sec:abp splitting}, where we will show that $\bar{\theta}$ is injective and can be augmented to an isomorphism (which induces a map of spectra realizing the desired 2-local splitting).

\begin{notn}
    If $B$ is an $\A_1$-module, we will denote the $\A$-module $B_\A:=\A\otimes_{\A_1}B$. Since $\A$ is free as a right $\A_1$-module, the functor $B\mapsto B_\A$ is exact. It follows that there is automatically an injective map of $\A_1$-modules $B \to B_\A$ given by $b \mapsto 1 \otimes b$.
\end{notn}

\begin{notn}\label{notn:theta bar}
    Because mod 2 cohomology of $\MSpin$, $\MSpinc$, and $\MSpinh$ will show up so frequently later in this section, we will use the notation 
    \begin{align*}
    M&:=H^*\MSpin,\\
    M_c&:=H^*\MSpinc,\\
    M_h&:=H^*\MSpinh.
    \end{align*}
    We will also write
    \[\overline{N}:=\bigoplus_{I\in\Part_\mr{even}}\Sigma^{4|I|}\udq_\A\oplus\bigoplus_{I\in\Part_\mr{odd}}\Sigma^{4|I|}E_\A,\]
    so that Equation~\ref{eq:theta bar} can be written as $\bar{\theta}:\overline{N}\to M_h$.
\end{notn}

\begin{setup}\label{setup:Q_i homology}
Let $Q_0 = \Sq^1$ and $Q_1 = \Sq^3 + \Sq^2 \Sq^1$. These satisfy $Q_0^2 = Q_1^2 = 0$, so we can view multiplication by $Q_0$ or $Q_1$ as a differential of a chain complex on any $\A_1$-module (and by extension, any $\A$-module). Also, $Q_0$ and $Q_1$ are primitive elements of $\A$, so they act on products $xy$ by $Q_i (xy) =(Q_i x) y + x(Q_i y)$. 

Any map of $\A_1$-modules becomes a map of chain complexes with respect to the differentials $Q_0\cdot(-)$ and $Q_1\cdot(-)$. Given an $\A_1$-module $B$, we will denote by $H_*(B;Q_i)$ the homology of $B$ with respect to $Q_i$. The usual results of homological algebra apply for computing $H_*(-;Q_i)$. In particular, short exact sequences of $\A_1$-modules induce long exact sequences in homology, and there is a K\"unneth theorem for $H_*(-;Q_i)$ \cite[Chapter 18.1, Propositions 1c and 2a]{Mar83}.\footnote{In Margolis's notation, we have $Q_0=P^0_1$ and $Q_1=P^0_1P^1_2+P^1_2P^0_1$.}
\end{setup}

\subsection{The upside-down question mark and the elephant}
We will begin by recalling a few basic computations of $Q_i$-homology, which we will then use to compute the $Q_i$-homology of $\udq_\A$ and $E_\A$. To do so, we need to introduce a little more notation.

\begin{notn}
Let $\chi : \A \to \A$ denote the antipode of the Hopf algebra $\A$. We will frequently use the following properties of $\chi$:
\begin{enumerate}[(i)]
\item $\chi(ab)=\chi(b)\chi(a)$ for all $a,b\in\A$.
\item $\chi(Q_i)=Q_i$ for all $i$.
\end{enumerate}

Let $\Delta_i=(a_1,a_2,a_3,\ldots)$,
where $a_i=1$ and $a_j=0$ for $j\neq i$. Given a sequence $R$ of natural numbers with finitely many non-zero terms, let $\Sq^R\in\A$ denote the Milnor basis vector associated to $R$. Finally, given a set $V$ of vectors in some vector space, let $\langle V\rangle$ denote the span of $V$.
\end{notn}

\begin{lem}\label{lem:Q_i homology of some quotients}
We have
\begin{align*}
    H_*(\A /(\A \Sq^1 + \A \Sq^2 ) ; Q_0) & = \left\langle \chi( \Sq^{4k}) \ \middle| \ k \in \N \right\rangle, \\
    H_*(\A /(\A \Sq^1 + \A \Sq^2) ; Q_1) & = \left\langle \chi( \Sq^{2\sum_{j=1}^k\Delta_{i_j}}) \ \middle| \ k \in \N,\ i_1 > \ldots > i_k \geq 2 \right\rangle, \\
    H_*(\A /\A \Sq^3 ; Q_0) & = \left\langle \chi(\Sq^{4k}) \Sq^2 \ \middle| \ k \in \N \right\rangle, \\
    H_*( \A / \A \Sq^3 ; Q_1) & = \left\langle \chi(\Sq^{2\sum_{j=1}^k\Delta_{i_j}}) \Sq^2 \ \middle| \ k \in \N,\ i_1 > \ldots > i_k \geq 2 \right\rangle.
\end{align*}
\end{lem}
\begin{proof}
This is \cite[Theorem 6.9]{ABP67}.
\end{proof}

The following is a sort of Leibniz rule for the $Q_i$-differentials.

\begin{lem}\label{lem:Q_i Leibniz rules}
For any natural number $n$ and any distinct natural numbers $i_1,\ldots,i_k$, we have
\begin{align*}
    Q_1\Sq^{2n-2}&=Q_0\Sq^{2n}+\Sq^{2n}Q_0,\\
    0&=Q_1\Sq^{2\sum_{j=1}^k\Delta_{i_j}}+\Sq^{2\sum_{j=1}^k\Delta_{i_j}}Q_1.
\end{align*}
\end{lem}
\begin{proof}
    In the Milnor basis, we have
    \[Q_i\Sq^{2I}+\Sq^{2I}Q_i=\sum_{j=1}^\infty Q_{i+j}\Sq^{2(I-2^i\Delta_j)}\]
    for any partition $I$ \cite[Theorem~4a]{Mil58}. We also have $\Sq^n=\Sq^{n\Delta_1}$ (see \cite[Section 6]{Mil58}) and $\Sq^R=0$ if any term of $R$ is negative \cite[p.~163]{Mil58}. Setting $i=0$ and $I=(n)$, we thus compute
    \begin{align*}
        Q_0\Sq^{2n}+\Sq^{2n}Q_0&=\sum_{j=1}^\infty Q_j\Sq^{2(n\Delta_1-\Delta_j)}\\
        &=Q_1\Sq^{2n-2}+\sum_{j=2}^\infty Q_j\cdot 0.
    \end{align*}
    Setting $i=1$ and $I=\sum_{j=1}^k\Delta_{i_j}$, we compute
    \begin{align*}
        Q_1\Sq^{2\sum_{j=1}^k\Delta_{i_j}}+\Sq^{2\sum_{j=1}^k\Delta_{i_j}}Q_1&=\sum_{\ell=1}^\infty Q_{\ell+1}\Sq^{2(\sum_{j=1}^k\Delta_{i_j}-2\Delta_\ell)} \\
        & = \sum_{\ell = 1}^\infty Q_{\ell + 1} \cdot 0,
    \end{align*}
    since $\sum_{j=1}^k\Delta_{i_j}-2\Delta_\ell$ always contains a negative term.
\end{proof}

Using Lemmas~\ref{lem:Q_i homology of some quotients} and~\ref{lem:Q_i Leibniz rules}, we are able to compute the homologies of both $\udq_\A$ and $E_\A$. In order to simplify the presentation of $H_*(\udq_\A;Q_i)$ and $H_*(E_\A;Q_i)$, we need another lemma.

\begin{lem}\label{lem:multiple of Sq1}
    We have $\chi( \Sq^{4k - 2}) \Sq^2 \in \A \Sq^1$ for any $k\in\N$.
\end{lem}
\begin{proof}
    If $k = 0$, then $\chi( \Sq^{4k - 2}) = 0$. If $k = 1$, then
    \begin{align*}
        \chi( \Sq^{4k - 2}) \Sq^2 &= \chi(\Sq^2) \Sq^2\\
        &=( \Sq^2)^2\\
        &= \Sq^3 \Sq^1 \in \A \Sq^1.
    \end{align*}
    Finally, suppose $k \geq 2$. We have the Adem relations
    \[
        \Sq^i \Sq^j = \sum_{t = 0}^{\left\lfloor \frac{i}{2} \right\rfloor} \binom{j - t - 1}{i - 2t} \Sq^{i + j - t} \Sq^t
    \]
    when $i < 2j$. We then compute
    \begin{align*}
        \Sq^2 \Sq^n & = \sum_{t = 0}^1 \binom{n - t - 1}{2 - 2t} \Sq^{n + 2 - t} \Sq^t \\
        & = \binom{n - 1}{2} \Sq^{n + 2} + \Sq^{n + 1} \Sq^1 \\
        & =
        \begin{cases}
            \Sq^{n + 2} + \Sq^{n + 1} \Sq^1 & n = 0, 3 \pmod{4}, \\
            \Sq^{n + 1} \Sq^1 & n = 1, 2 \pmod{4}
        \end{cases}
    \end{align*}
    if $2 < 2n$. This implies that
    \begin{align*}
        \Sq^2 \Sq^{4k - 4} + \Sq^2 \Sq^{4k - 5} \Sq^1 & = \Sq^{4k - 2} + \Sq^{4k - 3} \Sq^1 +( \Sq^{4k - 3} + \Sq^{4k - 4} \Sq^1) \Sq^1 \\
        & = \Sq^{4k - 2}.
    \end{align*}
    Therefore
    \begin{align*}
        \chi( \Sq^{4k - 2}) \Sq^2 & = \chi( \Sq^2( \Sq^{4k - 4} + \Sq^{4k - 5} \Sq^1)) \Sq^2 \\
        & = \chi( \Sq^{4k - 4} + \Sq^{4k - 5} \Sq^1) \chi( \Sq^2) \Sq^2 \\
        & = \chi( \Sq^{4k - 4} + \Sq^{4k - 5} \Sq^1) \Sq^2 \Sq^2 \\
        & = \chi( \Sq^{4k - 4} + \Sq^{4k - 5} \Sq^1) \Sq^3 \Sq^1.\qedhere
    \end{align*}
\end{proof}

\begin{prop}\label{prop:Q homology}
    The Margolis homology of $\udq_\A$ has the following presentation:
    \begin{align*}
        H_*( \udq_\A ; Q_0) & \cong \left\langle \chi( \Sq^{4k}) q_0 \ \middle| \ k \in \N \right\rangle, \\
        H_*( \udq_\A ; Q_1) & \cong \left\langle \chi( \Sq^{2\sum_{j=1}^k\Delta_{i_j}}) \Sq^2 q_0 \ \middle| \ k \in \N,\ i_1 > \ldots > i_k \geq 2 \right\rangle.
    \end{align*}
\end{prop}
\begin{proof}
    There is a short exact sequence
    \[0 \longrightarrow \A_1 \Sq^3 \longrightarrow \A_1 \longrightarrow \A_1 / \A_1 \Sq^3 \longrightarrow 0,
    \]
    and an isomorphism $\A_1 \Sq^3 \cong \Sigma^3\udq$, where $\Sq^3$ corresponds to $q_0$. So tensoring with $\A$ gives us a short exact sequence
    \[0 \longrightarrow \Sigma^3 \udq_\A \longrightarrow \A \longrightarrow \A / \A \Sq^3 \longrightarrow 0,
    \]
    which induces a long exact sequence in homology:
    \[H_*( \A ; Q_i)\longrightarrow H_*( \A / \A \Sq^3 ; Q_i)\longrightarrow H_{* + \deg Q_i}( \Sigma^3 \udq_\A ; Q_i) \longrightarrow H_{* + \deg Q_i}( \A ; Q_i).\]
    The homologies of $\A$ vanish \cite[p.~331, Proposition 1]{Mar83}, so the connecting homomorphism
    \begin{equation}\label{eq:connecting for question mark}
    H_*( \A / \A \Sq^3 ; Q_i) \to H_{* + \deg Q_i}( \Sigma^3 \udq_\A ; Q_i)
    \end{equation}
    is an isomorphism. Since the homology of $\Sigma^3\udq_\A$ is a degree 3 shift of the homology of $\udq_\A$, there is an isomorphism $H_*( \A / \A \Sq^3 ; Q_i) \to H_{* + \deg Q_i - 3}( \udq_\A ; Q_i)$. 
    
    A presentation of $H_*(\A/\A\Sq^3;Q_i)$ is given in Lemma~\ref{lem:Q_i homology of some quotients}. All that remains is to give a formula for the connecting homomorphism (Equation~\ref{eq:connecting for question mark}). The connecting homomorphism can be computed for a class in $H_*( \A / \A \Sq^3 ; Q_i)$ by choosing a representative cycle in $\A / \A \Sq^3$, lifting it to an element of $\A$, acting by $Q_i$ to get a cycle in $\udq_\A$, and taking the resulting homology class. 
    
    For $Q_0$, we lift $\chi(\Sq^{4k})\Sq^2$ and compute
    \begin{align*}
        Q_0 \chi( \Sq^{4k}) \Sq^2 & = \chi( Q_0) \chi( \Sq^{4k}) \Sq^2 \\
        & = \chi ( \Sq^{4k} Q_0 ) \Sq^2 \\
        & = \chi ( Q_0 \Sq^{4k} + Q_1 \Sq^{4k - 2} ) \Sq^2\tag{Lemma~\ref{lem:Q_i Leibniz rules}} \\
        & = \chi ( \Sq^{4k} ) Q_0 \Sq^2 + \chi ( \Sq^{4k - 2} ) Q_1 \Sq^2 \\
        & = \chi ( \Sq^{4k} ) \Sq^3 + \chi ( \Sq^{4k - 2} ) \Sq^2 \Sq^3.\tag{Setup~\ref{setup:Q_i homology}}
    \end{align*}
    Lemma~\ref{lem:multiple of Sq1} implies that $\chi ( \Sq^{4k - 2} ) \Sq^2 \in \A \Sq^1$, so the second term vanishes (because $\Sq^1 \Sq^3 = 0$). Finally, the isomorphism $\A \Sq^3 \cong \Sigma^3 \udq_\A$ has $\Sq^3$ in correspondence with $q_0$, so the $Q_0$-homology of $\udq_\A$ is generated by $\chi(\Sq^{4k}) q_0$. 
    
    For $Q_1$, we lift $\chi(\Sq^{2\sum_{j=1}^k\Delta_{i_j}})\Sq^2$ and compute
    \begin{align*}
        Q_1 \chi ( \Sq^{2\sum_{j=1}^k\Delta_{i_j}} ) \Sq^2 & = \chi ( \Sq^{2\sum_{j=1}^k\Delta_{i_j}} Q_1 ) \Sq^2 \\
        & = \chi ( Q_1 \Sq^{2\sum_{j=1}^k\Delta_{i_j}} ) \Sq^2\tag{Lemma~\ref{lem:Q_i Leibniz rules}} \\
        & = \chi ( \Sq^{2\sum_{j=1}^k\Delta_{i_j}} ) Q_1 \Sq^2 \\
        & = \chi ( \Sq^{2\sum_{j=1}^k\Delta_{i_j}} ) \Sq^2 \Sq^3.\tag{Setup~\ref{setup:Q_i homology}}
    \end{align*}
    Therefore the $Q_1$-homology of $\udq_\A$ is generated by $\chi ( \Sq^{2\sum_{j=1}^k\Delta_{i_j}} ) \Sq^2 q_0$.
    \end{proof}

\begin{prop}\label{prop:E homology}
    The Margolis homology of $E_\A$ has the following presentation:
    \begin{align*}
        H_*( E_\A ; Q_0) & \cong \left\langle \chi( \Sq^{4k}) e_0 + \chi( \Sq^{4k - 2}) \Sq^1 e_1 \ \middle| \ k \in \N \right\rangle, \\
        H_*( E_\A ; Q_1) & \cong \left\langle \chi( \Sq^{2\sum_{j=1}^k\Delta_{i_j}})( \Sq^2 e_0 + \Sq^1 e_1) \ \middle| \ k \in \N,\ i_1 > \ldots > i_k \geq 2 \right\rangle.
    \end{align*}
\end{prop}
    \begin{proof}
    There is a short exact sequence
    \[
       0 \longrightarrow\A_1 \Sq^1 + \A_1 \Sq^2 \longrightarrow \A_1 \longrightarrow \A_1 /( \A_1 \Sq^1 + \A_1 \Sq^2) \longrightarrow 0
    \]
    and an isomorphism $\A_1 \Sq^1 + \A_1 \Sq^2\cong \Sigma E$, where $\Sq^1$ corresponds to $e_0$ and $\Sq^2$ corresponds to $e_1$. Tensoring with $\A$ thus gives us a short exact sequence
    \[
        0 \longrightarrow \Sigma E_\A \longrightarrow \A \longrightarrow \A /( \A \Sq^1 + \A \Sq^2) \longrightarrow 0.
    \]
    Since $H_*(\A;Q_i)=0$ \cite[p.~331, Proposition 1]{Mar83}, we see that the connecting homomorphism in the induced long exact sequence on $Q_i$-homology is an isomorphism
    \[H_*( \A /( \A \Sq^1 + \A \Sq^2) ; Q_i ) \to H_{* + \deg Q_i - 1}( E_\A ; Q_i).\]
    We have already calculated $H_*(\A/(\A\Sq^1+\A\Sq^2);Q_i)$ in Lemma~\ref{lem:Q_i homology of some quotients}, so we can compute the connecting homomorphism for $Q_0$-homology by
    \begin{align*}
        Q_0 \chi ( \Sq^{4k} ) & = \chi ( \Sq^{4k} Q_0 ) \\
        & = \chi ( Q_0 \Sq^{4k} + Q_1 \Sq^{4n - 2} )\tag{Lemma~\ref{lem:Q_i Leibniz rules}} \\
        & = \chi ( \Sq^{4k} ) Q_0 + \chi ( \Sq^{4n - 2} ) Q_1 \\
        & = \chi ( \Sq^{4k} ) \Sq^1 + \chi ( \Sq^{4n - 2} ) \Sq^2 \Sq^1 + \chi ( \Sq^{4n - 2} ) \Sq^1 \Sq^2.\tag{Setup~\ref{setup:Q_i homology}}
    \end{align*}
    Since $\chi ( \Sq^{4n - 2} ) \Sq^2 \in \A \Sq^1$ by Lemma~\ref{lem:multiple of Sq1}, the term $\chi ( \Sq^{4n - 2} ) \Sq^2 \Sq^1$ vanishes because $( \Sq^1 )^2 = 0$. Since the isomorphism $\A\Sq^1+\A\Sq^2\to\Sigma E_\A$ satisfies $\Sq^1\mapsto e_0$ and $\Sq^2\mapsto e_1$, we find that $H_* ( E_\A ; Q_0 )$ is generated by $\chi ( \Sq^{4k} ) e_0 + \chi ( \Sq^{4k - 2} ) \Sq^1 e_1$. 
    
    For $Q_1$, we calculate
    \begin{align*}
        Q_1 \chi ( \Sq^{2\sum_{j=1}^k\Delta_{i_j}} ) & = \chi ( \Sq^{2\sum_{j=1}^k\Delta_{i_j}} Q_1 ) \\
        & = \chi ( Q_1 \Sq^{2\sum_{j=1}^k\Delta_{i_j}} )\tag{Lemma~\ref{lem:Q_i Leibniz rules}} \\
        & = \chi ( \Sq^{2\sum_{j=1}^k\Delta_{i_j}} ) Q_1 \\
        & = \chi ( \Sq^{2\sum_{j=1}^k\Delta_{i_j}} ) ( \Sq^2 \Sq^1 + \Sq^1 \Sq^2 ),\tag{Setup~\ref{setup:Q_i homology}}
    \end{align*}
    so $H_* ( E_\A ; Q_1 )$ is generated by $\chi ( \Sq^{2\sum_{j=1}^k\Delta_{i_j}} ) ( \Sq^2 e_0 + \Sq^1 e_1 )$.
\end{proof}

\subsection{$Q_i$-homology of $H^*\MSpinh$}
Our next goal is to calculate the $Q_i$-homology of $M_h:=H^*\MSpinh$. To begin, recall the Wu formula:
\begin{equation}\label{eq:wu}
    \Sq^i w_j = \sum_{t = 0}^i \binom{j + t - i - 1}{t} w_{i - t} w_{j + t}.
\end{equation}
Evaluating the action of $\Sq^1$ and using the fact that $w_1 = 0$ in $H^*\BSpinh$, we can compute $Q_0w_j$ and $Q_1w_j$.

\begin{lem}\label{lem:Q_i times Stiefel--Whitney}
    For any $j \in \N$,
    \[
        Q_0 w_j =
        \begin{cases}
            w_{j + 1} & j \text{ is even}, \\
            0 & j \text{ is odd},
        \end{cases}
        \qquad\text{and}\qquad
        Q_1 w_j =
        \begin{cases}
            w_{j + 3} + w_3 w_j & j \text{ is even}, \\
            w_3 w_j & j \text{ is odd}
        \end{cases}
    \]
    in $H^* \BSpinh$.
\end{lem}
\begin{proof}
    Using Equation~\ref{eq:wu}, we find
    \begin{align*}
        Q_0 w_j & = \binom{j - 2}{0} w_1 w_j + \binom{j - 1}{1} w_0 w_{j + 1} \\
        & =( j - 1) w_{j + 1}.
    \end{align*}
    For $Q_1$, we need to evaluate $\Sq^2$ and $\Sq^3$ as well:
    \begin{align*}
        \Sq^2 w_j & = w_2 w_j + \binom{j - 1}{2} w_{j + 2}, \\
        \Sq^3 w_j & = w_3 w_j +( j - 3) w_2 w_{j + 2} + \binom{j - 1}{3} w_{j + 3}.
    \end{align*}
    Putting these together gives us
    \begin{align*}
        Q_1 w_j & =( \Sq^3 + \Sq^2 \Sq^1) w_j \\
        & = w_3 w_j+( j - 3) w_2 w_{j + 2} + \binom{j - 1}{3} w_{j + 3} +( j - 1) \Sq^2 w_{j + 1} \\
        & = w_3 w_j +( j - 3) w_2 w_{j + 2} + \binom{j - 1}{3} w_{j + 3} +( j - 1) w_2 w_{j + 1} +( j - 1) \binom{j}{2} w_{j + 3} \\
        & = w_3 w_j + \left( \binom{j - 1}{3} +( j - 1) \binom{j}{2} \right) w_{j + 3}.
    \end{align*}
    It remains to determine the parity of $\binom{j-1}{3}+(j-1)\binom{j}{2}$. Note $\binom{j}{2} = \frac{1}{2} j( j - 1)$ is even if and only if $j = 0, 1 \pmod{4}$, and $\binom{j - 1}{3} = \frac{1}{6}( j - 1)( j - 2)( j - 3)$ is even if and only if $j = 1, 2, 3 \pmod{4}$. Thus $\binom{j-1}{3}+(j-1)\binom{j}{2}$ is odd if $j$ is even and is even if $j$ is odd.
\end{proof}

As an application of Lemma~\ref{lem:Q_i times Stiefel--Whitney}, we prove the following lemma used in the proof of Proposition \ref{prop:lift of pontryagin class}.

\begin{lem}\label{lem:pontryagin class not q0 boundary}
    The class $w_{2i_1}^2 \cdots w_{2i_s}^2\in\BSpinh$ is not in the image of $\Sq^1$.
\end{lem}
\begin{proof}
    Let $B = H^* \BSpinh$. By \cite[Corollary 2.35]{Hu22},
    \[
        H_*(B ; Q_0) = \Z / 2\Z[w_2^2, w_{2k}^2, \nu_{2^r} \ | \ k \neq 2^r, r > 1],
    \]
    where $\nu_i$ is the $i$th Wu class. For our purposes, the only fact about the Wu classes we need is that $\nu_{2^r}$ is $w_{2^r}$ plus products of lower degree classes \cite[p. 316]{Sto68}. In particular, we can write
    \[
        \nu_{2^r} = w_{2^r} + q_r (w_2, \ldots, w_{2^r - 2}) + \sum_{j = 1}^{2^r - 1} w_j x_{r, j}
    \]
    for some polynomial $q_r$ and some $x_{r, j}$. Since we are working over $\Z/2\Z$, the freshman's dream gives us
    \begin{equation}\label{eq:w^2}
    \begin{aligned}
        w_{2^r}^2 & = \nu_{2^r}^2 + q_r(w_2^2, \ldots, w_{2^r - 2}^2) + \sum_{j = 1}^{2^r - 1} w_j^2 x_{r, j}^2 \\
        & = \nu_{2^r}^2 + q_r(w_2^2, \ldots, w_{2^r - 2}^2) + \sum_{j = 1}^{2^r - 1} Q_0 (w_{j - 1} w_j x_{r, j}^2).
    \end{aligned}
    \end{equation}
    Expanding out $w = w_{2i_1}^2 \cdots w_{2i_s}^2$ (using Equation~\ref{eq:w^2} if necessary), we see that $w$ cannot be in the image of $\Sq^1$. Indeed, the expansion of $w$ is a sum of monomials in $R:=\Z / 2\Z[w_2^2, w_{2k}^2, \nu_{2^r}]$ and products of monomials of $R$ and terms of the form
    \begin{equation}\label{eq:Q_0 terms}
    Q_0 (w_{j - 1} w_j x_{r, j}^2).
    \end{equation}
    Modulo terms of the form in Equation~\ref{eq:Q_0 terms} (which lie in the image of $Q_0=\Sq^1$), the class $w$ is a non-zero sum of linearly independent monomials that generate $H_*(B;Q_0)$. As generators of $_*(B;Q_0)$, such monomials do not lie in the image of $Q_0\cdot -$, so $w\not\in\mr{im}(Q_0\cdot -)$.
\end{proof}

To calculate $H_*(M_h;Q_i)$, we will use methods similar to those of \cite{ABP67}.\footnote{One could alternatively use the shearing map (Lemma~\ref{lem:shearing}) for this calculation.} The difficult part of this computation is managing the Stiefel--Whitney classes in $H^* \BSpinh$ that hit a decomposable Stiefel--Whitney class after applying the $Q_i$-differential. The resolution is that there is always a way to replace these classes. To construct our replacement generators, we need another lemma.

\begin{lem}
The map
\begin{align*}
    \A&\to M_h\\
    1&\mapsto U_h
\end{align*}
factors through $\udq_\A \cong \A /( \A \Sq^1 + \A( \Sq^5 + \Sq^4 \Sq^1))$.
\end{lem}
\begin{proof}
This is true because $\Sq^1 U_h = w_1 U_h = 0$ and
\begin{align*}
    ( \Sq^5 + \Sq^4 \Sq^1) U_h & = \Sq^2 \Sq^3 U_h \\
    & = \Sq^2( w_3 U_h) \\
    & = ( \Sq^2 w_3) U_h + ( \Sq^1 w_3 ) ( \Sq^1 U_h) + w_3( \Sq^2 U_h) \\
    & = ( w_2 w_3 + w_5) U_h + w_3 w_2 U_h\tag{Equation~\ref{eq:wu}} \\
    & = w_2 w_3 U_h + w_2 w_3 U_h\tag{$w_5\in H^5\BSpinh$ vanishes} \\
    & = 0.\qedhere
\end{align*}
\end{proof}

\begin{cor}\label{cor:Q_i cycles}
    Each cycle in $H_*(\udq_\A;Q_i)$ maps to a cycle in $H_*(M_h;Q_i)$.
\end{cor}
\begin{proof}
    This follows by applying $H_*(-;Q_i)$ to the factorization $\A\to\udq_\A\to M_h$.
\end{proof}

The final ingredient we need before computing $H_*(M_h;Q_0)$ is an alternative presentation of $H^*\BSpinh$.

\begin{lem}\label{lem:alternative cohomology of bspinh}
    For each $k\geq 2$, there are cohomology classes $f_{2^k} \in H^{2^k} \BSpinh$ such that $Q_0 f_{2^k} = 0$ and
    \[
        H^* \BSpinh \cong \Z / 2\Z \left[ w_i, f_{2^k} \ \middle| \ k \geq 2, i \geq 2, i \neq 2^r\text{ or } 2^r + 1 \text{ for } r \geq 2 \right].
    \]
\end{lem}
\begin{proof}
    There are classes $f_{2^k} \in H^{2^k} \BSpinh$ such that $\chi( \Sq^{2^k}) U_h = \chi( \Sq^{4 \cdot 2^{k - 2}}) U_h = f_{2^k} U_h$, as the Thom isomorphism theorem implies every element of $H^* \MSpinh$ can be written in the form $f U_h$ for some $f \in H^* \BSpinh$. Since $Q_0 U_h = 0$ by Lemma~\ref{lem:Q_i times Stiefel--Whitney}, we have
    \begin{align*}
        0 & = Q_0( f_{2^k} U_h) \\
        & = ( Q_0 f_{2^k}) U_h + f_{2^k}( Q_0 U_h) \\
        & =( Q_0 f_{2^k}) U_h.
    \end{align*}
    This implies $Q_0 f_{2^k} = 0$. From the usual presentation of $H^*\BSpinh$ (Proposition~\ref{prop:bspinh cohomology}), it remains to show that $f_{2^k}\equiv w_{2^k}\mod{(w_1,\ldots,w_{2^k-1})}$. To prove this, write $\chi(\Sq^{2^k})=\Sq^{2^k}+a$ for some $a\in\A$, so that
    \begin{align*}
        f_{2^k}U_h&=(\Sq^{2^k}+a)U_h\\
        &=w_{2^k}U_h+aU_h.
    \end{align*}
    If we write $aU_h$ in the monomial basis, we want to show that the coefficient of $w_{2^k}U_h$ is zero. Since $\chi( \Sq^{2^k})\equiv\Sq^{2^k}\mod{(\Sq^1,\ldots,\Sq^{2^k-1})}$ \cite[Section 7]{Mil58}, we know that $a\in(\Sq^1,\ldots,\Sq^{2^k-1})$. Now expand out the terms of $aU_h$ by using the action of the Steenrod squares on $U_h$ and the Wu formula. None of the resulting monomials can have degree equal to that of $w_{2^k}$, so $a \in (w_1, \ldots, w_{2^k - 1})$.
\end{proof}

We are now set to compute $H_*(M_h;Q_0)$.

\begin{lem}\label{lem:Q0 homology of M_h}
    Let $f_{2^k}\in H^{2^k}\BSpinh$ be as in Lemma~\ref{lem:alternative cohomology of bspinh}. Let
    \[
        R:=\Z / 2\Z \left[ w_{2i}^2, f_{2^k} \ \middle| \ k \geq 2, i
        \geq 3, i \neq 2^{r - 1} \text{ for } r \geq 2 \right].
    \]
    Then $H_*(M_h;Q_0)$ is the free $R$-module generated by $U_h\in M_h$.
\end{lem}
\begin{proof}
    We first use the K\"unneth theorem to break up the calculation into manageable pieces. By Lemma~\ref{lem:alternative cohomology of bspinh}, we can decompose $H^* \BSpinh$ the following tensor product:
    \[
        H^* \BSpinh \cong \bigotimes_{k \geq 2} \Z / 2\Z[ f_{2^k}] \otimes \bigotimes_{i \neq 2^{r - 1}, r \geq 2} \Z / 2\Z[ w_{2i}, w_{2i + 1}].
    \]
    Each factor is well-defined as a module over the exterior algebra generated by $Q_0$, since $Q_0 f_{2^k} = Q_0 w_{2i+1}= 0$ and $Q_0 w_{2i} = w_{2i + 1}$. Every monomial in $\Z / 2\Z[ f_{2^k}]$ is a cycle, so $H_*( \Z / 2\Z[ f_{2^k}] ; Q_0) \cong \Z / 2\Z \left[ f_{2^k} \right]$. For the other factors, we have
    \[
        Q_0( w_{2i}^a w_{2i + 1}^b) = a w_{2i}^{a - 1} w_{2i + 1}^{b + 1}.
    \]
    It follows that $\ker(Q_0\cdot-)$ is the subspace generated by those monomials having an even number of $w_{2i}$ factors, and that $\mr{im}(Q_0\cdot-)$ is the subspace generated by those monomials having an even number of $w_{2i}$ factors and at least one $w_{2i + 1}$. Hence the homology is generated by monomials having an even number of $w_{2i}$ factors and no factors of $w_{2i + 1}$, so $H_*( \Z / 2\Z \left[ w_{2i}, w_{2i + 1} \right] ; Q_0) \cong \Z / 2\Z \left[ w_{2i}^2 \right]$. Using the K\"unneth theorem over a field, we see that
    \[H_*(H^* \BSpinh;Q_0)\cong R.\] Since the Thom class $U_h$ satisfies $Q_0 U_h = 0$, the Thom isomorphism $x \mapsto x U_h$ is a map of chain complexes. It follows that $H_*( M_h ; Q_0)\cong R\cdot U_h$, as desired.
\end{proof}

To compute $H_*(M_h;Q_1)$, we again need an alternative presentation of $H^*\BSpinh$.

\begin{lem}\label{lem:another alternative cohomology of bspinh}
There are classes $t_{2j+1}\in H^{2j+1}\BSpinh$ for $j\geq 3$ and $j\neq 2^m$, and classes $g_{2^k-2}\in H^{2^k-2}\BSpinh$ for $k\geq 3$, such that $Q_1t_{2j+1}=Q_1g_{2^k-2}=0$ and
\[H^*\BSpinh\cong\Z/2\Z\left[w_3,w_{2i},t_{2j+1},g_{2^r-2}\ \middle|\ i\neq 2^{r-1}-1,r\geq 3,j\geq 3,j\neq 2^m\right].\]
\end{lem}
\begin{proof}
    Define $t_{2j + 1}:=w_{2j + 1} + w_3 w_{2j - 2}$. Since $j \neq 2^m$ for $m\geq 1$, we have $2j + 1 \neq 2^{m+1} + 1$ and hence $w_{2j + 1}$ is one of the polynomial generators of $H^* \BSpinh$. It follows from Proposition~\ref{prop:bspinh cohomology} that
    \[
        H^* \BSpinh \cong \Z / 2\Z \left[ w_{2i}, w_3, t_{2j + 1} \ \middle| \ j \neq 2^{r - 1} \right].
    \]
    To see that $Q_1 t_{2j + 1} = 0$, note that $Q_1 w_{2j - 2} = w_{2j + 1} + w_3 w_{2j - 2} = t_{2j + 1}$ and recall that $Q_1^2=0$. 

    Next, we employ similar tactics as in Lemma~\ref{lem:alternative cohomology of bspinh} to construct the classes $g_{2^k - 2}$, although slightly more work is needed due to the fact that $U_h$ is not a $Q_1$-cycle. To define the $g_{2^r - 2}$, we induct on $r \geq 3$, using the same argument for the base case and inductive step. Specifically, we assume that the $g_{2^q - 2}$ have been constructed for $q < r$, that $Q_1g_{2^q-2}=0$, and that 
    \begin{equation}\label{eq:induction}
    H^*\BSpinh\cong\Z/2\Z\left[w_3,w_{2i},t_{2j+1},g_{2^q-2} \ \middle|\ i\neq 2^{q-1}-1,3\leq q<r,j\geq 3,j\neq 2^m\right].
    \end{equation}
    By our computation of $H_*(\udq_\A;Q_1)$ (Proposition~\ref{prop:Q homology}), we see that $\chi(\Sq^{2\Delta_{r-1}})\Sq^2U_h$ is a $Q_1$-cycle. To extract a replacement of $w_{2^r-2}$ from this, first write $\chi(\Sq^{2\Delta_{r-1}})\Sq^2U_h=aw_2U_h+bU_h$, where $aw_2$ and $b$ are classes of degree $2^r$ and no monomials of $b$ (in the basis given by Equation~\ref{eq:induction}) are multiples of $w_2$. We will set $a:=g_{2^r-2}$.
    
    We first need to verify that $Q_1g_{2^r-2}=0$. To this end, we compute
    \begin{align*}
        0&=Q_1(\chi(\Sq^{2\Delta_{r-1}})\Sq^2U_h)\tag{$Q_1$-cycle}\\
        &=Q_1(aw_2U_h+bU_h)\\
        &=(Q_1a)w_2U_h+(Q_1b)U_h+bw_3U_h.\tag{$Q_1=\Sq^3+\Sq^2\Sq^1$}
    \end{align*}
    Note that in the monomial basis, no term of $(Q_1b)U_h$ are divisible by $w_2$. Indeed, no terms of $b$ are divisible by $w_2$, and the images under $Q_1$ of the basis elements are
    \begin{align*}
        Q_1w_{2i}&=t_{2i+3}=w_{2i+3}+w_3w_{2i},\\
        Q_1w_3&=w_3^2,\\
        Qt_{2j+1}&=0,\\
        Qg_{2^q-2}&=0,
    \end{align*}
    none of which are divisible by $w_2$. It follows that no terms of $bw_3U_h$ are divisible by $w_2$, so $(Q_1a)w_2U_h=0$ and therefore $Q_1a=0$.
    
    It remains to show that $a$ is $w_{2^r - 2}$ plus products of lower generators. We will revert to the Stiefel--Whitney generators of $H^* \BSpinh$ for this step, since rewriting the Stiefel--Whitney generators in terms of the new generators will not introduce monomials with $w_{2^r - 2}$ for degree reasons. By \cite[Proposition 6.2]{ABP67}, $\chi(\Sq^{2 \Delta_{r - 1}})$ is $\Sq^{2^r - 2}$ modulo admissible sequences with two or more factors. By expanding out the action of these other admissible sequences using the Wu formula, we find that the monomial $w_{2^r - 2}$ cannot arise from these admissible sequences. It follows that $a$ is indeed $w_{2^r - 2}$ modulo products of lower degree terms.
\end{proof}

Now we compute $H_*(M_h;Q_1)$.

\begin{lem}\label{lem:Q1 homology of M_h}
    Let $g_{2^k-2}\in H^{2^k-2}\BSpinh$ be as in Lemma~\ref{lem:another alternative cohomology of bspinh}. Let
    \[
        S:=\Z / 2\Z \left[ w_{2i}^2, g_{2^r - 2} \ \middle| \ i \neq 2^{r - 1} - 1, r \geq 3 \right] w_2 U.
    \]
    Then $H_*(M_h;Q_1)$ is the free $S$-module generated by $w_2U_h$.
\end{lem}
\begin{proof}
    By Lemma~\ref{lem:another alternative cohomology of bspinh}, the Thom isomorphism, and the K\"unneth formula, $M_h$ can be written as the tensor product
    \[\Z/2\Z[w_2,w_3]U_h\otimes\bigotimes_{\substack{i\geq 2\\ i\neq 2^m-1}}\Z/2\Z[w_{2i},t_{2i+3}]\otimes\bigotimes_{r\geq 3}\Z/2\Z[g_{2^r-2}].\]
    Moreover, each of these factors is closed under the action of $Q_1$ because
    \begin{equation}\label{eq:Q_1 products}
    \begin{split}
        Q_1 U_h &= w_3 U_h,\\
        Q_1 w_2 &= w_2 w_3,\\
        Q_1 w_3 &= w_3^2,\\
    \end{split}
    \qquad\qquad
    \begin{split}
        Q_1 w_{2i} &= t_{2i + 3},\\
        Q_1 t_{2i + 3} &= 0,\\
        Q_1 g_{2^r - 2} &= 0.
    \end{split}
    \end{equation}
    In order to determine $H_*(M_h;Q_1)$, it thus suffices to compute the $Q_1$-homology of each factor individually. For $\Z / 2\Z [w_2, w_3] U_h$, the action of $Q_1$ on the monomial $w_2^a w_3^b U_h$ is
    \begin{align*}
        Q_1( w_2^a w_3^b U_h) & = a w_2^a w_3^{b + 1} U_h + b w_2^a w_3^{b + 1} U_h + w_2^a w_3^{b + 1} U_h \\
        & =( a + b + 1) w_2^a w_3^{b + 1} U_h.
    \end{align*}
    Thus $\ker(Q_1\cdot-)$ is the subspace generated by all the monomials $w_2^a w_3^b U_h$ where $a + b$ is odd, and $\mr{im}(Q_1\cdot -)$ is the subspace generated by all the monomials $w_2^a w_3^{b + 1} U_h$ where $a + b$ is even. Rephrased, $\mr{im}(Q_1\cdot -)$ is the subspace generated by all the monomials $w_2^a w_3^b U_h$ where $a + b$ is odd and $b \geq 1$. It follows that $H_*(\Z/2\Z[w_2,w_3]U_h;Q_1)\cong\Z / 2\Z[ w_2^2] w_2 U_h$. 
    
    Equation~\ref{eq:Q_1 products} implies that the image of $Q_1\cdot -$ on $\Z/2\Z[g_{2^r-2}]$ is trivial, so
    \[H_*(\Z/2\Z[g_{2^r-2}];Q_1)\cong\Z/2\Z[g_{2^r-2}].\]
    Similarly, Equation~\ref{eq:Q_1 products} implies that $t_{2i+3}\in\mr{im}(Q_1\cdot -)$ on $\Z/2\Z[w_{2i},t_{2i+3}]$, while the kernel of $Q_1\cdot -$ is generated by $w_{2i}^2$ and $t_{2i+3}$. It follows that
    \[H_*(\Z/2\Z[w_{2i},t_{2i+3}];Q_1)\cong\Z/2\Z[w_{2i}^2],\]
    and we are done.
\end{proof}

\subsection{$\bar{\theta}$ induces isomorphisms on $Q_i$-homology}
Recall the map $\bar{\theta}:\overline{N}\to M_h$ from Notation~\ref{notn:theta bar}. Our next goal is to prove that $\bar{\theta}$ induces isomorphisms $H_*(\overline{N};Q_i)\to H_*(M_h;Q_i)$ for $i=0$ and 1. We will do so by comparing $\bar{\theta}$ to the analogous map coming from the Anderson--Brown--Peterson splitting of $\MSpinc$.\footnote{A posteriori, what makes this approach work is that we have a natural bijection of non-Eilenberg--Mac Lane summands in the 2-local splittings of $\MSpinc$ and $\MSpinh$. For another consequence of this observation, see Corollary~\ref{cor:rank spinc and spinh}.} For this, we need the following lemma relating $H_*(\udq_\A;Q_i)$ and $H_*(E_\A;Q_i)$ to $H_*(C_\A;Q_i)$, where $C$ is the $\A_1$-module defined in Definition~\ref{def:C}. 

\begin{notn}
    Let $e_0:=\Sq^1$ and $e_1:=\Sq^2$ be the generators of $E_\A$. Let $q_0:=\Sq^1$ and $c_0:=\Sq^1$ denote the generators of $\udq_\A$ and $C_\A$, respectively.
\end{notn}

\begin{lem}
    There are unique non-trivial maps $\udq_\A \to C_\A$ and $E_\A \to C_\A$. Moreover, these maps induce monomorphisms on $H_*( - ; Q_i)$.
\end{lem}
\begin{proof}
    Note that the only non-zero element of $C_\A$ of degree zero is $c_0$, and that $C_\A$ has no non-zero elements of degree one. Thus if non-trivial maps $\udq_\A \to C_\A$ and $E_\A\to C_\A$ exist, they must be given by $q_0 \mapsto c_0$ and
    \begin{align*}
        e_0&\mapsto c_0,\\
        e_1&\mapsto 0,
    \end{align*}
    respectively. We will show that these determine maps of $\A_1$-modules, and tensoring with $\A$ will give the desired maps of $\A$-modules. Consider the map $\udq \to C$ given by
    \begin{align*}
    q_0&\mapsto c_0,\\
    \Sq^2 q_0 &\mapsto \Sq^2 c_0,\\
    \Sq^3 q_0 &\mapsto 0.
    \end{align*}
    This map commutes with the action of $\A_1$ on $\udq$ and $C$, so this is a map of $\A_1$-modules and therefore induces a map of $\A$-modules $\udq_\A \to C_\A$. 
    
    Next, consider the map $E \to C$ given by
    \begin{align*}
        e_0 &\mapsto c_0,\\
        e_1 &\mapsto 0,\\
        \Sq^1 e_1 &\mapsto 0,\\
        \Sq^2 e_0 &\mapsto \Sq^2 c_0, 
    \end{align*}
    and sending all elements of higher degree to zero. As before, this map commutes with the action of $\A_1$ on $E$ and $C$, so this determines a map of $\A_1$-modules and hence induces the desired $\A$-module map $E_\A \to C_\A$.

    To see that these maps induce injections in homology, notice that both maps $Q \to C$ and $E \to C$ are surjective. The only element in the kernel of $Q \to C$ is $\Sq^3 q_0$, so the kernel is isomorphic to $\Sigma^3 \A_1 /(\A_1 \Sq^1 + \A_1 \Sq^2)$. This yields a short exact sequence
    \begin{equation}\label{eq:ses for Q to C}
        0 \longrightarrow \Sigma^3 \A_1 /( \A_1 \Sq^1 + \A_1 \Sq^2) \longrightarrow \udq \longrightarrow C \longrightarrow 0.
    \end{equation}
    Tensoring Equation~\ref{eq:ses for Q to C} with $\A$ gives us a short exact sequence
    \begin{equation}\label{eq:ses for Q_A to C_A}
        0 \longrightarrow \Sigma^3 \A /( \A \Sq^1 + \A \Sq^2) \longrightarrow \udq_\A \longrightarrow C_\A \longrightarrow 0.
    \end{equation}
    Equation~\ref{eq:ses for Q_A to C_A} induces a long exact sequence in $Q_i$-homology. The relevant part is
    \[
        H_j( \Sigma^3 \A /( \A \Sq^1 + \A \Sq^2) ; Q_i) \longrightarrow H_j( \udq_\A ; Q_i) \longrightarrow H_j( C_\A ; Q_i).
    \]
    Showing that the map $H_j( \udq_\A ; Q_i) \to H_j( C_\A ; Q_i)$ is injective is equivalent to showing that
    \begin{equation}\label{eq:map from quotient to Q}
    H_j( \Sigma^3 \A /( \A \Sq^1 + \A \Sq^2) ; Q_i) \to H_j( \udq_\A ; Q_i)
    \end{equation}
    is zero. Since there is an isomorphism $H_j( \Sigma^3 \A /( \A \Sq^1 + \A \Sq^2) ; Q_i) \cong H_{j - 3}( \A /( \A \Sq^1 + \A \Sq^2) ; Q_i)$ and the homology $H_*( \A /( \A \Sq^1 + \A \Sq^2) ; Q_i)$ is only nonzero in even degrees (Lemma~\ref{lem:Q_i homology of some quotients}), the map in Equation~\ref{eq:map from quotient to Q} can only be nonzero for $j$ odd. But $H_j( \udq_\A ; Q_i)\cong 0$ for $j$ odd (Proposition~\ref{prop:Q homology}), so $H_j( \udq_\A ; Q_i) \to H_j( C_\A ; Q_i)$ is injective.

    The argument for $E$ is similar. The kernel of $E \to C$ is isomorphic to $\Sigma \A_1 / \A_1 \Sq^3$, with the inclusion $\Sigma \A_1 / \A_1 \Sq^3 \to E$ given by $1 \mapsto e_1$. Tensoring by $\A$ gives us a short exact sequence
    \[
        0 \longrightarrow \Sigma \A / \A \Sq^3 \longrightarrow E_\A \longrightarrow C_\A \longrightarrow 0,
    \]
    which induces an exact sequence
    \[
        H_j( \Sigma \A / \A \Sq^3 ; Q_i) \longrightarrow H_j( E_\A ; Q_i) \longrightarrow H_j( C_\A ; Q_i).
    \]
    Again, it will suffice to show that the the map $H_j( \Sigma \A / \A \Sq^3 ; Q_i) \to H_j( E_\A ; Q_i)$ is zero. There is an isomorphism $H_j( \Sigma \A / \A \Sq^3 ; Q_i) \cong H_{j - 1}( \A / \A \Sq^3 ; Q_i)$ and $H_*( \A / \A \Sq^3 ; Q_i)$ is nonzero only in even degrees (Lemma~\ref{lem:Q_i homology of some quotients}), so this map has a non-zero domain only when $j$ is odd. But $H_*( E_\A ; Q_i)$ is zero in odd degrees (Proposition~\ref{prop:E homology}), so the codomain is trivial if $j$ is odd. Hence the map is always zero.
\end{proof}

\begin{setup}
    We now explain the comparison to $\MSpinc$ that we will use to compute $\bar{\theta}$ on $Q_i$-homology. Let $\overline{N}_c = \bigoplus_{I \in \Part} \Sigma^{4| I|} C_\A$. Then we have the commutative diagram
\[
    \begin{tikzcd}
        \arrow[from=1-1, to=2-1]
        \arrow["{\bar{\theta}_c}", from=1-1, to=2-3]
        \arrow["\psi"',from=2-1, to=2-3]
        \overline{N}_c \\
        \bigoplus_{I \in \Part} \Sigma^{4| I|} C_\A \oplus \bigoplus_{z \in Z_c} \Sigma^{\deg z} \A & & M_c,
    \end{tikzcd}
\]
where the vertical map is the inclusion of the summands on the left and $\psi$ is the isomorphism in cohomology induced by the Anderson--Brown--Peterson splitting of $\MSpinc$. Since $\psi$ is an isomorphism, it induces isomorphisms on $Q_i$-homology. The vertical map induces isomorphisms on $Q_i$-homology because the $Q_i$-homology of each $\Sigma^{\deg z} \A$ summand vanishes \cite[p.~331, Proposition 1]{Mar83}. It follows that $\bar{\theta}_c$ induces isomorphisms on $H_*(-;Q_i)$. Moreover, $\bar{\theta}_c$ takes the generator $c_0$ of the $C_\A$ summand corresponding to a partition $I \in \Part$ to $p_I U_c \in M_c$, where $U_c\in M_c:=H^*\MSpinc$ is the Thom class. 

Altogether, we have a diagram
\begin{equation}\label{eq:partial square}
    \begin{tikzcd}
        \arrow["{\bar{\theta}}", from=1-1, to=2-1]
        \arrow["{\bar{\theta}_c}", from=1-2, to=2-2]
        \arrow[from=2-1, to=2-2]
        \overline{N} & \overline{N}_c \\
        M_h & M_c
    \end{tikzcd}
\end{equation}
in the category of $\A$-modules. If we could fill this in to make a commuting square, then understanding the map $M_h \to M_c$ would give us control over $\bar{\theta}$. Unfortunately, there is no obvious way to do this. Instead, we will fill in Diagram~\ref{eq:partial square} to a non-commutative diagram that yields a commutative diagram on $Q_i$-homology.

We define a map $\overline{N} \to \overline{N}_c$ by treating even partition summands and odd partition summands separately. For $I\in\Part_\mr{even}$, set
\begin{align*}
    \Sigma^{4|I|}\udq_\A&\to\Sigma^{4|I|}C_\A\\
    q_0&\mapsto c_0.
\end{align*}
For $I\in\Part_\mr{odd}$, set
\begin{align*}
    \Sigma^{4|I|}E_\A&\to\Sigma^{4|I|}C_\A\\
    e_0&\mapsto c_0,\\
    e_1&\mapsto 0.
\end{align*}
Combined with Diagram~\ref{eq:partial square}, this gives us a non-commuting square
\begin{equation}\label{eq:filled square}
    \begin{tikzcd}
        \arrow[from=1-1, to=1-2]
        \arrow["{\bar{\theta}}", from=1-1, to=2-1]
        \arrow["{\bar{\theta}_c}", from=1-2, to=2-2]
        \arrow[from=2-1, to=2-2]
        \overline{N} & \overline{N}_c \\
        M_h & M_c.
    \end{tikzcd}
\end{equation}
\end{setup}

\begin{rem}
Note that Diagram~\ref{eq:filled square} does commute when we restrict $\overline{N}$ to the submodule generated by all the $E_\A$ summands. Indeed, the top arrow followed by $\bar{\theta}_c$ takes $e_0\mapsto p_I U_c$ and $e_1\mapsto 0$, and $\bar{\theta}$ followed by the bottom arrow takes $e_0\mapsto p_I U_c$ and $e_1$ to the product of the image of $\beta_I$ and $w_3 U_c$. Since $w_3$ vanishes in $H^* \BSpinc$, we find that $e_1\mapsto 0$.
\end{rem}

Next up, we show that $\bar{\theta}$ induces an injection on $Q_0$-homology.

\begin{lem}\label{lem:Q0 inj}
    The map $\bar{\theta} : \overline{N} \to M_h$ induces an injection on $Q_0$-homology.
\end{lem}
\begin{proof}
    Applying $H_*(-;Q_0)$ to Diagram~\ref{eq:filled square} gives us the diagram
    \begin{equation}\label{eq:Q_0 square}
        \begin{tikzcd}
            \arrow[from=1-1, to=1-2]
            \arrow["{\bar{\theta}}_*", from=1-1, to=2-1]
            \arrow["{\bar{\theta}_{c*}}", from=1-2, to=2-2]
            \arrow[from=2-1, to=2-2]
            H_* ( \overline{N} ; Q_0 ) & H_* ( \overline{N}_c ; Q_0 ) \\
            H_* ( M_h ; Q_0 ) & H_* ( M_c ; Q_0 )
        \end{tikzcd}
    \end{equation}
    of $\Z/2\Z$-vector spaces. We claim that Diagram~\ref{eq:Q_0 square} commutes, which we will prove by checking commutativity for each generator of $H_*( \overline{N} ; Q_0)$. We will then use our understanding of $\bar{\theta}_{c*}$ to prove that $\bar{\theta}_*$ is an isomorphism.
    
    If $I$ is an even partition, then the generator $\chi(\Sq^{4k})q_0\in H_*( \udq_\A ; Q_0)$ satisfies
    \[\begin{tikzcd}[row sep=0pt]
        H_*(\overline{N};Q_0)\arrow[r] & H_*(\overline{N}_c;Q_0)\arrow[r,"\bar{\theta}_{c*}"] & H_*(M_c;Q_0)\\
        \chi(\Sq^{4k})q_0\arrow[r,maps to] & \chi(\Sq^{4k})c_0\arrow[r,maps to] & \chi(\Sq^{4k})p_IU_c.
    \end{tikzcd}\]
    Taking the other path around Diagram~\ref{eq:Q_0 square}, Proposition~\ref{prop:ksp/elephant class in cohomology} implies
    \[\begin{tikzcd}[row sep=0pt]
        H_*(\overline{N};Q_0)\arrow[r,"\bar{\theta}_*"] & H_*(M_h;Q_0)\arrow[r] & H_*(M_c;Q_0)\\
        \chi(\Sq^{4k})q_0\arrow[r,maps to] & \chi(\Sq^{4k})(p_I+\Sq^3\Sq^1\alpha_I)U_h. &
    \end{tikzcd}\]
    We need to show that $\chi(\Sq^{4k})(p_I+\Sq^3\Sq^1\alpha_I)U_h$ maps to $\chi( \Sq^{4k})p_I U_c$. We first calculate
    \begin{align*}
        \Delta( \Sq^3 \Sq^1) & = ( \Delta \Sq^3)( \Delta \Sq^1) \\
        & = ( \Sq^3 \otimes 1 + \Sq^2 \otimes \Sq^1 + \Sq^1 \otimes \Sq^2 + 1 \otimes \Sq^3 ) ( \Sq^1 \otimes 1 + 1 \otimes \Sq^1 ) \\
        & = \Sq^3 \Sq^1 \otimes 1 + \Sq^2 \Sq^1 \otimes \Sq^1 + \Sq^1 \Sq^1 \otimes \Sq^2 + \Sq^1 \otimes \Sq^3 \\
        & \quad + \Sq^3 \otimes \Sq^1 + \Sq^2 \otimes \Sq^1 \Sq^1 + \Sq^1 \otimes \Sq^2 \Sq^1 + 1 \otimes \Sq^3 \Sq^1 \\
        & = \Sq^3 \Sq^1 \otimes 1 + 1 \otimes \Sq^3 \Sq^1 + \Sq^3 \otimes \Sq^1 + \Sq^1 \otimes \Sq^3 \\
        & \quad + \Sq^2 \Sq^1 \otimes \Sq^1 + \Sq^1 \otimes \Sq^2 \Sq^1.
    \end{align*}
    From this, we calculate
    \begin{align*}
        \Sq^3 \Sq^1 (\alpha_I U_h ) & = (\Sq^3 \Sq^1 \alpha_I ) U_h + (\Sq^1 \alpha_I ) w_3 U_h \\
        & = (\Sq^3 \Sq^1 \alpha_I ) U_h + \Sq^1 (\alpha_I w_3 U_h ),
    \end{align*}
    so
    \begin{equation}
        \label{eq:steenrod action thom class}
        (\Sq^3 \Sq^1 \alpha_I ) U_h = \Sq^3 \Sq^1 (\alpha_I U_h ) + \Sq^1 (\alpha_I w_3 U_h ).
    \end{equation}
    It follows that $\chi (\Sq^{4k} ) q_0$ maps to $\chi (\Sq^{4k} ) (p_I U_h + \Sq^3 \Sq^1 (\alpha_I U_h ) + \Sq^1 (\alpha_I w_3 U_h ) )$. 
    
    To deal with the $\chi(\Sq^{4k}\Sq^3\Sq^1(\alpha_IU_h)$ term, we check that $\chi(\Sq^{4k}\Sq^3\Sq^1)$ is a $Q_0$-cycle in $\A$:
    \begin{align*}
        \Sq^1 \chi (\Sq^{4k} ) \Sq^3 \Sq^1 & = \chi (\Sq^1 ) \chi (\Sq^{4k} ) \Sq^3 \Sq^1 \\
        & = \chi (\Sq^{4k} \Sq^1 ) \Sq^3 \Sq^1 \\
        & = \chi (\Sq^1 \Sq^{4k} + Q_1 \Sq^{4k - 2} ) \Sq^3 \Sq^1 \\
        & = \chi (\Sq^{4k} ) \Sq^1 \Sq^3 \Sq^1 + \chi (\Sq^{4k - 2} ) Q_1 \Sq^3 \Sq^1 \\
        & = 0.
    \end{align*}
    Since the $Q_0$-homology of $\A$ vanishes, there is some $a \in \A$ such that $\chi (\Sq^{4k} ) \Sq^3 \Sq^1 = \Sq^1 a$. Thus $\chi (\Sq^{4k} ) \Sq^3 \Sq^1 (\alpha_I U_h ) = \Sq^1 a (\alpha_I U_h )$, and this term vanishes in $Q_0$-homology.
    
    At this point, we have deduced that $\chi (\Sq^{4k} ) q_0$ maps to the element $\chi (\Sq^{4k} ) (p_I U_h + \Sq^1 (\alpha_I w_3 U_h ) )$. To deal with the $\chi(\Sq^{4k})\Sq^1(\alpha_Iw_3U_h)$ term, recall that $w_3$ vanishes in $H^* \BSpinc$. Thus $\chi (\Sq^{4k} ) \Sq^1 (\alpha_I w_3 U_h )$ maps to zero in $M_c$, and we find that $\chi(\Sq^{4k})q_0$ maps to $\chi (\Sq^{4k} ) (p_I U_c )$, as desired. Thus Diagram~\ref{eq:Q_0 square} commutes for each $\udq_\A$ summand. Since Diagram~\ref{eq:filled square} commutes for the $E_\A$ summands, Diagram~\ref{eq:Q_0 square} commutes for the $E_\A$ summands. Thus Diagram~\ref{eq:Q_0 square} commutes.

    Because $\bar{\theta}_{c*}$ is an isomorphism by Theorem~\ref{thm:spinc splitting}, and since the map
    \[H_* (\overline{N} ; Q_0 ) \to H_* (\overline{N}_c ; Q_0 )\]
    is the direct sum of injective maps (and is hence injective), the composite $H_* (\overline{N} ; Q_0 ) \to H_* (M_c ; Q_0 )$ is injective. This implies that $\bar{\theta}_*$ is injective too.
    \end{proof}

Now we strengthen Lemma~\ref{lem:Q0 inj} by showing that $\bar{\theta}_*$ is in fact an isomorphism.
    
\begin{lem}\label{lem:Q0 iso}
    The map $\bar{\theta} : \overline{N} \to M_h$ induces an isomorphism on $Q_0$-homology.
\end{lem}
    \begin{proof}
    Since $\bar{\theta}_*$ is injective, we just need to show that dimensions of $H_*(\overline{N};Q_0)$ and $H_*(M_h;Q_0)$ in each degree are equal. We will prove that these dimensions are equal in each degree by showing that $H_*(\overline{N};Q_0)$ and $H_*(M_h;Q_0)$ have the same Hilbert--Poincar\'e series. 
    
    First, we compute the Hilbert--Poincar\'e series of $H_*(M_h;Q_0)$. By Lemma~\ref{lem:Q0 homology of M_h}, we can write any monomial in $H_* (M_h ; Q_0 )$ as $ABU_h$, where $A$ is a monomial in the $w_{2i}^2$ and $f_{2^k}^2$ and $B$ is a product of $f_{2^k}$ with each factor occurring at most once. The number of monomials of the form $A$ in degree $4n$ is the number of partitions of $n$, and there are no monomials of this form in degrees not divisible by four. Since each $f_{2^k}$ has degree $2^k$ (ranging over $k \geq 2$), the degree of $B$ is the number whose binary expansion has a $1$ in the $k^\textsuperscript{th}$ place for every $f_{2^k}$ factor. There is one such $B$ for every number divisible by four, and no others. Hence the Hilbert--Poincar\'e series of $H_* (M_h ; Q_0 )$ is $\sum_{I \in \Part} t^{4|I|} (1 - t^4 )^{-1}$.
    
    The Hilbert--Poincar\'e series for $H_* (\overline{N} ; Q_0 )$ is simple to compute: the Hilbert--Poincar\'e series of both $H_* (\udq_\A ; Q_0 )$ and $H_* (E_\A ; Q_0 )$ is $(1 - t^4 )^{-1}$ (Propositions~\ref{prop:Q homology} and~\ref{prop:E homology}), and for each partition $I$, there is a single summand of $\udq_\A$ or $E_\A$ shifted by degree $4|I|$. Thus the Hilbert--Poincar\'e series of $H_*(\overline{N};Q_0)$ is also $\sum_{I\in\Part}t^{4|I|}(1-t^4)^{-1}$.
\end{proof}

Now we turn to the effect of $\bar{\theta}$ on $Q_1$-homology, employing the same strategies as before. We will again see that $\bar{\theta}_*$ is injective and even an isomorphism.

\begin{lem}\label{lem:Q1 inj}
    The map $\bar{\theta} : \overline{N} \to M_h$ induces an injection on $Q_1$-homology.
\end{lem}
\begin{proof}
    As in Lemma~\ref{lem:Q0 inj}, we will show that Diagram~\ref{eq:filled square} commutes after applying the functor $H_* (- ; Q_1 )$ by checking on generators. If $I$ is an even partition, then the generator $\chi(\Sq^{2\sum_{\ell=1}^k\Delta_{i_\ell}})\Sq^2q_0\in H_*( \udq_\A ; Q_1)$ satisfies
    \[\begin{tikzcd}[row sep=0pt]
        H_*(\overline{N};Q_1)\arrow[r] & H_*(\overline{N}_c;Q_1)\arrow[r,"\bar{\theta}_{c*}"] & H_*(M_c;Q_1)\\
        \chi(\Sq^{2\sum_{\ell=1}^k\Delta_{i_\ell}})\Sq^2q_0\arrow[r,maps to] & \chi (\Sq^{2 \sum_{\ell=1}^k\Delta_{i_\ell}} ) \Sq^2 c_0\arrow[r,maps to] & \chi (\Sq^{2 \sum_{\ell=1}^k\Delta_{i_\ell}} ) \Sq^2 (p_I U_c ).
    \end{tikzcd}\]
    For the other path (in Diagram~\ref{eq:Q1 square}), we get
    \[\bar{\theta}_*(\chi(\Sq^{2\sum_{\ell=1}^k\Delta_{i_\ell}})\Sq^2q_0)=\chi (\Sq^{2 \sum_{\ell=1}^k\Delta_{i_\ell}} ) \Sq^2 (p_I U_h + (\Sq^3 \Sq^1 \alpha_I ) U_h ).\] 
    By Equation \ref{eq:steenrod action thom class}, this maps to
    \[
        \chi (\Sq^{2 \sum_{\ell=1}^k\Delta_{i_\ell}} ) (\Sq^2 (p_I U_h ) + \Sq^2 \Sq^3 \Sq^1 (\alpha_I U_h ) + \Sq^2 \Sq^1 (\alpha_I w_3 U_h ) ).
    \]
    We then compute
    \begin{align*}
        Q_1 \chi ( \Sq^{2 \sum_{\ell=1}^k\Delta_{i_\ell}} ) \Sq^2 \Sq^3 \Sq^1 & = \chi ( Q_1 ) \chi ( \Sq^{2 \sum_{\ell=1}^k\Delta_{i_\ell}} ) \Sq^2 \Sq^3 \Sq^1 \\
        & = \chi ( \Sq^{2 \sum_{\ell=1}^k\Delta_{i_\ell}} Q_1 ) \Sq^2 \Sq^3 \Sq^1 \\
        & = \chi ( Q_1 \Sq^{2 \sum_{\ell=1}^k\Delta_{i_\ell}} ) \Sq^2 \Sq^3 \Sq^1 \\
        & = \chi ( \Sq^{2 \sum_{\ell=1}^k\Delta_{i_\ell}} ) Q_1 \Sq^2 \Sq^3 \Sq^1 \\
        & = 0.
    \end{align*}
    Since $H_* ( \A ; Q_1 ) \cong 0$, this implies there is some $a \in A$ with $\chi ( \Sq^{2 \sum_{\ell=1}^k\Delta_{i_\ell}} ) \Sq^3 \Sq^1 = Q_1 a$. Hence the term $\chi ( \Sq^{2 \sum_{\ell=1}^k\Delta_{i_\ell}} ) \Sq^2 \Sq^3 \Sq^1 ( \alpha_I U_h )$ is a boundary and vanishes in $Q_1$-homology, and therefore our generator maps to
    \[
        \chi ( \Sq^{2 \sum_{\ell=1}^k\Delta_{i_\ell}} ) ( \Sq^2 ( p_I U_h ) + \Sq^2 \Sq^1 ( \alpha_I w_3 U_h ) ) \in H_* ( M_h ; Q_1 ).
    \]
    As $w_3$ vanishes in $H^* \BSpinc$, this maps to
    \[
        \chi ( \Sq^{2 \sum_{\ell=1}^k\Delta_{i_\ell}} ) \Sq^2 ( p_I U_c ) \in H_* ( M_c ; Q_1 ).
    \]
    Hence the diagram 
    \begin{equation}\label{eq:Q1 square}
        \begin{tikzcd}
            \arrow[from=1-1, to=1-2]
            \arrow["{\bar{\theta}_*}", from=1-1, to=2-1]
            \arrow["\bar{\theta}_{c*}", from=1-2, to=2-2]
            \arrow[from=2-1, to=2-2]
            H_*( \overline{N} ; Q_1) & H_*( \overline{N}_c ; Q_1) \\
            H_*( M_h ; Q_1) & H_*( M_c ; Q_1 )
        \end{tikzcd}
    \end{equation}
    commutes for the $\udq_\A$ summands. Diagram~\ref{eq:filled square} commutes on the $E_\A$ summands, so Diagram~\ref{eq:Q1 square} commutes for the $E_\A$ summands. Thus Diagram~\ref{eq:Q1 square} commutes in general. Theorem~\ref{thm:spinc splitting} implies that $\bar{\theta}_{c*}$ is an isomorphism, and $H_*( \overline{N} ; Q_1) \to H_*( \overline{N}_c ; Q_1)$ is a direct sum of injective maps (and is hence injective), so $\bar{\theta}_*$ must be injective as well.
\end{proof}

\begin{lem}\label{lem:Q1 iso}
    The map $\bar{\theta} : \overline{N} \to M_h$ induces an isomorphism on $Q_1$-homology.
\end{lem}
\begin{proof}
    As in Lemma~\ref{lem:Q0 iso}, it suffices to show that $H_*(M_h;Q_1)$ and $H_*(\overline{N};Q_1)$ have the same Hilbert--Poincar\'e series. By Lemma~\ref{lem:Q1 homology of M_h}, we can write any monomial in $H_* ( M_h ; Q_1 )$ as $ABw_2 U_h$, where $A$ is a monomial in the $w_{2i}^2$ and $g_{2^r - 2}^2$ and $B$ is a product of $g_{2^r - 2}$ with each factor occurring at most once. The number of monomials of the form $A$ in degree $4n$ is the number of partitions of $n$, and there are no monomials of this form in degrees not divisible by four. Let $s$ be the Hilbert--Poincar\'e series for the exterior algebra $\bigwedge \left[ g_{2^r - 2} \ \middle| \ r \geq 3 \right]$. Then the Hilbert--Poincar\'e series of $H_* ( M_h ; Q_1 )$ is $t^2 \sum_{I \in \Part} t^{4 | I |} s.$
    
    For the Hilbert--Poincar\'e series of $H_* ( \overline{N} ; Q_1 )$, note that the degree of $\chi ( \Sq^{2 \sum_{\ell=1}^k\Delta_{i_\ell}} )$ for $i_1 > \ldots > i_k \geq 2$ is
    \[2\sum_{\ell=1}^k(2^{i_\ell}-1)=\sum_{\ell=1}^k(2^{i_\ell+1}-2),\]
    which is also the degree of the exterior product $g_{2^{i_1 + 1} - 2} \cdots g_{2^{i_k + 1} - 2}$. So the Hilbert--Poincar\'e series of $H_* ( \udq_\A ; Q_1 )$ and $H_* ( E_\A ; Q_1 )$ are both $t^2 s$, and hence the Hilbert--Poincar\'e series of $H_* ( \overline{N} ; Q_1 )$ is $\sum_{I \in \Part} t^{4 | I |} t^2 s$, as desired.
\end{proof}

\begin{cor}\label{cor:iso on Q_i homology}
    The map $\bar{\theta} : \overline{N} \to M_h$ induces isomorphisms on $Q_i$-homology.
\end{cor}
\begin{proof}
    This is just the combination of Lemma~\ref{lem:Q0 iso} and Lemma~\ref{lem:Q1 iso}.
\end{proof}

\section{Anderson--Brown--Peterson splitting of $\MSpinh$}\label{sec:abp splitting}
Using the $Q_i$-homology isomorphisms given in Section~\ref{sec:margolis homology}, we now prove Theorem~\ref{thm:main} (which we restate here for convenience).

\begin{thm}\label{thm:spinh splitting}
    There is a set of homogeneous classes $Z \subset H^* \MSpinh$ and a map
    \[
        \MSpinh \to \bigvee_{I \in \Part_{\mathrm{even}}} \ksp \langle 4 | I | \rangle \vee \bigvee_{I \in \Part_{\mathrm{odd}}} \Sigma^{4 | I |} F \vee \bigvee_{z \in Z} \Sigma^{\deg z} H\Z / 2\Z
    \]
    that is a $2$-local homotopy equivalence.
\end{thm}

To begin, we need to construct our class $Z\subset H^*\MSpinh$ of homogeneous classes. As before, we will use the notation $M_h:=H^*\MSpinh$.

\begin{setup}\label{setup:Z}
Let $\A_+ \subset \A$ be the (left) submodule generated by all elements of positive degree. Now form the composition
\[
    \begin{tikzcd}
        \arrow["{\bar{\theta}}", from=1-1, to=1-2]
        \arrow["{\rho}", from=1-2, to=1-3]
        \overline{N} & M_h & M_h / \A_+ M_h,
    \end{tikzcd}
\]
where $\bar{\theta}$ is the map given in Notation~\ref{notn:theta bar} and $\rho : M_h \to M_h / \A_+ M_h$ is the quotient map. Take the cokernel $c : M_h / \A_+ M_h \to R$ of $\rho\circ\bar{\theta}$. Let $Z \subset M_h$ be any collection of homogeneous elements such $c\circ\rho(Z)$ is a basis for $R$.
\end{setup}

We will show that $Z$ is (an instance of) the desired set of homogeneous classes. In order to prove Theorem~\ref{thm:spinh splitting}, we first need to expand $\overline{N}$ to include the cohomology of $\bigvee_{z\in Z}\Sigma^{\deg{z}}H\Z/2\Z$.

\begin{notn}\label{notn:theta and N}
    Define
\[
    N := \overline{N} \oplus \bigoplus_{z \in Z} \Sigma^{\deg z} \A,
\]
and let $\theta : N \to M_h$ be the map defined by $\bar{\theta}$ on $\overline{N}$ and the maps
\begin{align*}
\Sigma^{\deg z} \A &\to M_h\\
1 &\mapsto z
\end{align*}
for each $z \in Z$.
\end{notn}

Surjectivity of $\theta$ is relatively straightforward.

\begin{lem}\label{lem:theta surjective}
    The map $\theta : N \to M_h$ is surjective.
\end{lem}
\begin{proof}
    Suppose $x \in M_h$. Then there are $z_1, \ldots z_n \in Z$ such that $c\rho x = c\rho z_1 + \ldots + c\rho z_n$, since $c\rho(Z)$ forms a basis of $R$. Thus there exists $y \in \overline{N}$ such that $\rho x = \rho z_1 + \ldots + \rho z_n + \rho \bar{\theta} y$, and there is some $a \in \A_+$ and $x' \in M_h$ such that $x = z_1 + \ldots + z_n + \bar{\theta} y + ax'$. In particular, $x'$ has lower degree than $x$. Since $M_h$ is bounded below (in degree), we can repeat this procedure until $x$ is written as a sum of Steenrod squares of elements of $Z$ and elements of $\bar{\theta}(\overline{N})$. Hence $x$ is in the image of $\theta$.
\end{proof}

Showing that $\theta$ is injective requires more work. The idea is to filter $N$ and $M_h$ and show that $\theta$ induces an isomorphism at each step in the filtration.

\begin{notn}\label{notn:filtered modules}
    For $n \in \Z$, let
    \[N^{[n]} \subset N\]
    be the submodule given by the direct sum of all the $\udq_\A$, $E_\A$, and $\A$ summands that are non-zero in degrees less than or equal to $n$. Let 
    \[M_h^{[n]} := \theta( N^{[n]}).\]
    Denote the restriction of $\theta$ by $\theta_n : N^{[n]} \to M_h^{[n]}$, and let $\lambda_n : N / N^{[n - 1]}\to M_h / M_h^{[n - 1]}$ be the induced map on quotients. 
\end{notn}

Note that by our definition of $N^{[n - 1]}$, the module $N / N^{[n - 1]}$ is the direct sum of those summands of $N$ that are zero in degrees less than $n$, and $N^{[n]} / N^{[n - 1]}$ is the direct sum of those summands that are zero in degrees less than $n$ but nonzero in degree $n$. Also, each summand of $N$ is of the form $B_\A$ for some $\A_1$-module $B$ (e.g.~the free summands of $N$ take the form $\A \cong ( \A_1 )_\A$). 

\begin{defn}
    Define $P_n \subset N^{[n]} / N^{[n - 1]}$ to be the $\A_1$-submodule given by the direct sum of $B$ for each summand $B_\A$ of $N^{[n]} / N^{[n - 1]}$.
\end{defn}

\begin{lem}\label{lem:P injective}
    If $\theta_{n - 1} : N^{[n - 1]} \to M_h^{[n - 1]}$ is an isomorphism, then the restriction of $\lambda_n$ to $P_n$ is injective.
\end{lem}
\begin{proof}
    First, notice that $P_n$ can be written as $X_n \oplus Y_n \oplus Z_n$, where $X_n = \bigoplus_{\alpha \in A_X} \Sigma^n \udq$, $Y_n = \bigoplus_{\alpha \in A_Y} \Sigma^n E$, and $Z_n = \bigoplus_{\alpha \in A_Z} \Sigma^n \A_1$, and $A_X$, $A_Y$, and $A_Z$ are finite sets. Moreover, $A_X$ is non-empty only if $n = 0 \pmod{8}$ and $A_Y$ is non-empty only if $n = 4 \pmod{8}$. In particular, one or the other is empty. We have short exact sequences fitting in commutative diagrams
    \[
        \begin{tikzcd}
            \arrow[from=1-1, to=1-2]
            \arrow[from=1-2, to=1-3]
            \arrow[from=1-3, to=1-4]
            \arrow[from=1-4, to=1-5]
            \arrow[from=1-1, to=2-1]
            \arrow["{\theta_{n - 1}}", from=1-2, to=2-2]
            \arrow["{\theta}", from=1-3, to=2-3]
            \arrow["{\lambda_n}", from=1-4, to=2-4]
            \arrow[from=1-5, to=2-5]
            \arrow[from=2-1, to=2-2]
            \arrow[from=2-2, to=2-3]
            \arrow[from=2-3, to=2-4]
            \arrow[from=2-4, to=2-5]
            0 & N^{[n - 1]} & N & N / N^{[n - 1]} & 0 \\
            0 & M_h^{[n - 1]} & M_h & M_h / M_h^{[n - 1]} & 0,
        \end{tikzcd}
    \]
    so we get long exact sequences in $Q_i$-homology:
    \[
        \begin{tikzcd}
            \arrow[from=1-1, to=1-2]
            \arrow[from=1-2, to=1-3]
            \arrow[from=1-3, to=1-4]
            \arrow["{{\theta_{n - 1}}_*}", from=1-1, to=2-1]
            \arrow["{\theta_*}", from=1-2, to=2-2]
            \arrow["{{\lambda_n}_*}", from=1-3, to=2-3]
            \arrow[from=2-1, to=2-2]
            \arrow[from=2-2, to=2-3]
            \arrow[from=2-3, to=2-4]
            \arrow[from=3-1, to=3-2]
            \arrow[from=3-2, to=3-3]
            \arrow["{{\theta_{n - 1}}_*}", from=3-2, to=4-2]
            \arrow["{\theta_*}", from=3-3, to=4-3]
            \arrow[from=4-1, to=4-2]
            \arrow[from=4-2, to=4-3]
            H_j ( N^{[n - 1]} ; Q_i ) & H_j ( N ; Q_i ) & H_j ( N / N^{[n - 1]} ; Q_i ) & {} \\
            H_j ( M_h^{[n - 1]} ; Q_i ) & H_j ( M_h ; Q_i ) & H_j ( M_h / M_h^{[n - 1]} ; Q_i ) & {} \\
            {} & H_{j + \deg Q_i} ( N^{[n - 1]} ; Q_i ) & H_{j + \deg Q_i} ( N ; Q_i ) \\
            {} & H_{j + \deg Q_i} ( M_h^{[n - 1]} ; Q_i ) & H_{j + \deg Q_i} ( M_h ; Q_i ).
        \end{tikzcd}
    \]
    Each ${\theta_{n - 1}}_*$ is an isomorphism because $\theta_{n - 1}$ is an isomorphism by hypothesis. Each $\theta_*$ is an isomorphism because $\bar{\theta}_*$ is an isomorphism by Corollary~\ref{cor:iso on Q_i homology}, and the inclusion $\overline{N} \to N$ induces isomorphisms on $Q_i$-homology because the $Q_i$-homology of each free (i.e.~$\A$) summand vanishes. So by the five lemma, ${\lambda_n}_* : H_i ( N / N^{[n - 1]} ; Q_i ) \to H_i ( M_h / M_h^{[n - 1]} ; Q_i )$ is an isomorphism.

    Now, to show that the restriction of $\lambda_n$ to $P_n$ is injective, note that the modules $\udq$, $E$, and $\A_1$ are concentrated in degrees $0$ through $6$, so $P_n$ is concentrated in degrees $n$ through $n+6$. It thus suffices to show that if $v \in P_n$ is homogeneous of degree $n+s$ for $0\leq s\leq 6$, and if $\lambda_n v = 0$, then $v=0$. We will describe the proof for $s=0$. The proofs for $1\leq s\leq 6$ are similar.
    
    Suppose $v$ has degree $n$. Then we can write $v = x + y + z$ for $x \in X_n$, $y \in Y_n$, and $z \in Z_n$. Setup~\ref{setup:Z} gives us a diagram
    \begin{equation}\label{eq:diagram for deg n lemma}
        \begin{tikzcd}
            \arrow["{\theta}", from=1-1, to=1-2]
            \arrow["{\rho}", from=1-2, to=1-3]
            \arrow["{c}", from=1-3, to=1-4]
            N & M_h & M_h / \A_+ M_h & R
        \end{tikzcd}
    \end{equation}
    in which the element $v \in N$ maps to zero in $M_h$. But $\theta z$ is a linear combination of elements of $Z$ that map to basis vectors for $R$ in Equation~\ref{eq:diagram for deg n lemma}. Moreover, $c(x)=c(y)=0$, so $z = 0$ (otherwise $\theta z$ would give a relation among the basis vectors of $R$). 
    
    So $v = x + y$ and $x = 0$ or $y = 0$. In either case, if $v \neq 0$, then $v$ represents a nonzero class in $Q_0$-homology, so the assumption $\lambda_n v = 0$ contradicts our previous conclusion that $\lambda_n$ induces an isomorphism on $Q_i$-homology. Thus $v = 0$, and hence $\lambda_n$ is a monomorphism on the degree $n$ part of $P_n$.
\end{proof}

Using the structure of $\MSpinh$ as a module over $\MSpin$, we can strengthen Lemma~\ref{lem:P injective} by extending the submodule on which $\lambda_n$ is injective.

\begin{lem}\label{lem:lambda_n}
    If $\theta_{n - 1}$ is an isomorphism, then the restriction of $\lambda_n$ to $N^{[n]} / N^{[n - 1]}$ is injective.
\end{lem}
\begin{proof}
    Since $\MSpinh$ is a module spectrum over $\MSpin$, taking cohomology gives $M_h:=H^*\MSpinh$ the structure of a comodule over the coalgebra $M:=H^*\MSpin$. Specifically, the comultiplication $\mu : M_h \to M \otimes M_h$ is induced by the multiplication map $\MSpin \wedge \MSpinh \to \MSpinh$. The identity axiom for a comodule states that the diagram 
    \[
        \begin{tikzcd}
            \arrow["{\mu}", from=1-1, to=1-2]
            \arrow["{\id_{M_h}}", from=1-1, to=2-1]
            \arrow["{\epsilon \otimes \id_{M_h}}", from=1-2, to=2-2]
            \arrow[from=2-2, to=2-1]
            M_h & M \otimes M_h \\
            M_h & \Z / 2\Z \otimes M_h
        \end{tikzcd}
    \]
    commutes, where $\Z / 2\Z \otimes M_h \to M_h$ is the canonical isomorphism and $\epsilon : M \to \Z / 2\Z$ is the map induced by the unit map $\mathbb{S} \to \MSpin$. Since the Thom class $U \in M$ is the only nonzero element of degree $0$, we see that $\epsilon (U) = 1$ and $\epsilon (x) = 0$ if $x$ has degree greater than zero. It follows that for any homogeneous $m \in M_h$, we have
    \begin{equation}
        \label{eq:comultiplication}
        \mu m = U \otimes m + \sum_{i = 1}^\alpha \ell_i \otimes m_i.
    \end{equation}
    Here $m_i$ has degree strictly less than that of $m$, as $\ell_i$ has degree strictly greater than zero. Indeed, if $\mu m = \sum_{i = 1}^{\alpha'} \ell_i' \otimes m_i'$, then the diagram above implies
    \[
        m = \sum_{i = 1}^\alpha \epsilon (\ell_i') m_i'.
    \]
    As an $\A_1$-module, $P_n$ is generated by the $\Sigma^nq_0 \in \Sigma^n \udq$, $\Sigma^ne_0,\Sigma^ne_1 \in \Sigma^n E$, and $\Sigma^n1 \in \Sigma^n \A_1$ of each summand. If $w$ is one of these generators, we have
    \[
        (\id_M \otimes p_n)\mu\theta w = U \otimes \lambda_n w,
    \]
    where $p_n:M_h\to M_h/M_h^{[n-1]}$ is the quotient map. To see this, we can separate the degree $n$ (i.e.~$\Sigma^nq_0$, $\Sigma^ne_0$, and $\Sigma^n1$) and $n+1$ (i.e.~$\Sigma^ne_1$) cases and check that the $m_i$ summands of $\mu w$ (from Equation~\ref{eq:comultiplication}) are killed by $p_n$.
    \begin{enumerate}[(i)]
    \item If $w$ has degree $n$, then each $m_i$ has degree less than $n$ and is killed by $p_n$. 
    \item If $w$ has degree $n + 1$, then because $M$ vanishes in degree one \cite[Theorem 8.1]{ABP67}, there are no terms in Equation~\ref{eq:comultiplication} where $m_i$ has degree $n$. So each $m_i$ has degree less than $n$ and is killed by $p_n$, as claimed.
    \end{enumerate}
    Since we are working with $\A_1$-modules and $U \in M$ is annihilated by $\Sq^1$ and $\Sq^2$, the Cartan formula for the action of $\A_1$ on the tensor product implies that
    \[
        (\id_M \otimes p_n)\mu\theta v = U \otimes \lambda_n v
    \]
    for all $v \in P_n$ (rather than just for the generators).
    
    Next, we want to show that there is a map $\mu_n : M_h / M_h^{[n - 1]} \to M \otimes M_h / M_h^{[n - 1]}$ such that the diagram
    \begin{equation}\label{eq:mu_n}
        \begin{tikzcd}
            \arrow["{\theta}", from=1-1, to=1-2]
            \arrow["{\mu}", from=1-2, to=1-3]
            \arrow[from=1-1, to=2-1]
            \arrow["{p_n}", from=1-2, to=2-2]
            \arrow["{\id_M \otimes p_n}", from=1-3, to=2-3]
            \arrow["{\lambda_n}", from=2-1, to=2-2]
            \arrow["{\mu_n}", from=2-2, to=2-3]
            N & M_h & M \otimes M_h \\
            N / N^{[n - 1]} & M_h / M_h^{[n - 1]} & M \otimes M_h / M_h^{[n - 1]}
        \end{tikzcd}
    \end{equation}
    commutes. To prove that such a $\mu_n$ exists, it suffices to show that $(\id_M \otimes p_n)\mu y = 0$ for each $y \in M_h^{[n - 1]}$. To this end, let $y\in M_h^{[n-1]}$. Since $M_h^{[n-1]}=\theta(N^{[n-1]})$ (Notation~\ref{notn:filtered modules}), there exists $x\in N^{[n-1]}$ such that $y = \theta x$. Write $x = \sum_{i = 1}^\alpha a_i x_i$, where $a_i \in \A$ and $x_i$ are generators for the summands that constitute $N^{[n - 1]}$. In particular, each $x_i$ has degree less than or equal to $n - 1$. By Equation \ref{eq:comultiplication}, we have
    \begin{align*}
        \mu y & = \mu\theta\left( \sum_{i = 1}^\alpha a_i x_i\right) \\
        & = \sum_{i = 1}^\alpha a_i \mu\theta x_i \\
        & = \sum_{i = 1}^\alpha a_i\left( U \otimes \theta x_i + \sum_{j = 1}^{\beta_i} \ell_{i, j} \otimes m_{i, j}\right),
    \end{align*}
    where $m_{i, j}$ has degree less than or equal to $n - 2$ (as $\deg(x_i)\leq n-1$). Thus $m_{i,j}\in M_h^{[n - 1]}$ for all $i,j$, so $(\id_M \otimes p_n)\mu y = 0$ and therefore the map $\mu_n$ exists and Diagram~\ref{eq:mu_n} commutes.

    We are finally read to show that $\lambda_n|_{N^{[n]}/N^{[n-1]}}$ is injective, which we do by contradiction. Suppose that $v\in N^{[n]}/N^{[n-1]}$ is non-zero and satisfies $\lambda_n v=0$. Let $\left\{ v_i \right\}_{i \in I}$ be a homogeneous basis of $P_n$ as a $\Z / 2\Z$-vector space. Since $N^{[n]} / N^{[n - 1]}$ is generated as an $\A$-module by $P_n$, we can write $v = \sum_{i \in I} a_i v_i$ for some homogeneous $a_i \in \A$. 
    
    If $a_i \in \A \Sq^1 + \A \Sq^2$, then we have $a_i = a_i' \Sq^1 + a_i'' \Sq^2$ for some $a_i', a_i'' \in \A$. We can then write $a_i v_i = (a_i' \Sq^1 + a_i'' \Sq^2) v_i = a_i' w_i' + a_i'' w_i''$, where $w_i' = \Sq^1 v_i$ and $w_i'' = \Sq^2 v_i$. Since $w_i',w_i''\in P_n$, we can rewrite $w_i'$ and $w_i''$ as linear combinations of $\{v_j\}_{j\in I}$. Since $\Sq^1$ and $\Sq^2$ increase degree, the basis elements in the linear combination for $a_i'w_i'+a_i''w_i''$ have greater degree than that of $v_i$. If $a_i'$ or $a_i''$ is an element of $\A\Sq^1+\A\Sq^2$, repeat this procedure. Since $\udq$, $E$, $\A_1$, and hence $P_n$ are all bounded above, this procedure eventually stabilizes. It follows that we can always write $v = \sum_{i \in I} a_i v_i$ with $a_i \notin \A \Sq^1 + \A \Sq^2$. 
    
    Let $k:=\max_{i\in I}\{\deg(a_i)\}$, which exists since all but finitely many $a_i$ must be zero. Let $i_1, \ldots, i_\alpha$ be the indices such that $\deg(a_{i_j})=k$. By Diagram~\ref{eq:mu_n}, our assumption $\lambda_n v = 0$ implies that $(\id_M \otimes p_n)\mu\theta v = 0$. (Here we conflate $v\in N^{[n]}/N^{[n-1]}$ with any choice of lift $v\in N^{[n]}$, since Diagram~\ref{eq:mu_n} commutes.) Then
    \begin{equation}\label{eq:linear combo}
    \begin{aligned}
        0 & = (\id_M \otimes p_n) \mu \theta v \\
        & = (\id_M \otimes p_n) \mu \theta \left( \sum_{i \in I} a_i v_i \right) \\
        & = \sum_{i \in I} a_i (\id_M \otimes p_n) \mu \theta v_i \\
        & = \sum_{i \in I} a_i (U \otimes \lambda_n v_i) \\
        & = \sum_{j = 1}^\alpha a_{i_j} U \otimes \lambda_n v_{i_j} + x,
    \end{aligned}
    \end{equation}
    where $x$ is a sum of terms belonging to $M^\beta \otimes M_h / M_h^{[n - 1]}$ for $\beta < k$. Recall that $\lambda_n|_{P_n}$ is injective (Lemma~\ref{lem:P injective}), so $\lambda_n v_{i_1},\ldots,\lambda_n v_{i_\alpha}$ are linearly independent. It thus follows from Equation~\ref{eq:linear combo} that $a_{i_1} U = \ldots = a_{i_\alpha} U = 0$. But the submodule of $M$ generated by $U$ is isomorphic to $\A / (\A \Sq^1 + \A \Sq^2)$, and we chose $a_{i_j}\not\in\A \Sq^1 + \A \Sq^2$, which yields the desired contradiction. Hence $v = 0$.
\end{proof}

We are now ready to show that our extension $\theta:N\to M_h$ of $\bar{\theta}:\overline{N}\to M_h$ (Notation~\ref{notn:theta and N}) is indeed an isomorphism. The general idea is to use Lemma~\ref{lem:lambda_n} to inductively show that $\theta$ is injective. Paired with Lemma~\ref{lem:theta surjective}, we will find that $\theta$ is an isomorphism.

\begin{prop}
    There exists a set of homogeneous elements $Z \subset M_h$ and an isomorphism $\theta : N \to M_h$ extending $\bar{\theta} : \overline{N} \to M_h$ along the inclusion $\overline{N} \to N$, where
    \[
        N = \overline{N} \oplus \bigoplus_{z \in Z} \Sigma^{\deg z} \A.
    \]
\end{prop}
\begin{proof}
    Let $Z$ and $\theta$ be as in Setup~\ref{setup:Z} and Notation~\ref{notn:theta and N}. We will induct on $n$, with our induction hypothesis the statement that $\theta_n : N^{[n]} \to M_h^{[n]}$ is an isomorphism. To simplify, our base case is $n = -1$, so that $N^{[n]}$ and $M_h^{[n]}$ are both trivial and there is nothing to check. 
    
    Now, assuming that $\theta_{n - 1}$ is an isomorphism, it suffices to show that $\theta_n$ is injective by Lemma~\ref{lem:theta surjective}. To this end, consider the diagram
    \[
        \begin{tikzcd}
            \arrow[from=1-1, to=1-2]
            \arrow[from=1-2, to=1-3]
            \arrow[from=1-3, to=1-4]
            \arrow[from=1-1, to=2-1]
            \arrow["{\theta_{n - 1}}", from=1-2, to=2-2]
            \arrow["{\theta_n}", from=1-3, to=2-3]
            \arrow["{\lambda_n}", from=1-4, to=2-4]
            \arrow[from=2-1, to=2-2]
            \arrow[from=2-2, to=2-3]
            \arrow[from=2-3, to=2-4]
            0 & N^{[n - 1]} & N^{[n]} & N^{[n]} / N^{[n - 1]} \\
            0 & M_h^{[n - 1]} & M_h^{[n]} & M_h^{[n]} / M_h^{[n - 1]}.
        \end{tikzcd}
    \]
    The rows are exact and $0 \to 0$ is an epimorphism. By the induction hypothesis $\theta_{n - 1}$ is a monomorphism, and $\lambda_n$ is a monomorphism by Lemma~\ref{lem:lambda_n}. The four lemma implies $\theta_n$ is a monomorphism. So $\theta_n$ is injective and hence an isomorphism.

    So by induction, each $\theta_n$ is an isomorphism. If $v \in N$ is homogeneous, then $v\in N^{[n]}$ for some $n$. If $\theta v = 0$, then $\theta_n v = 0$, and therefore $v = 0$. Therefore $\theta$ is injective. It now follows from Lemma~\ref{lem:theta surjective} that $\theta$ is an isomorphism.
\end{proof}

The proof of the main theorem now follows formally.

\begin{proof}[Proof of Theorem \ref{thm:spinh splitting}]
    For each $z \in Z$, let $\MSpinh \to \Sigma^{\deg z} H\Z / 2\Z$ be the map classifying $z \in H^* \MSpinh$. Together with the $\KSp$-Pontryagin and elephant classes, we get a map
    \begin{equation}\label{eq:splitting}
        \MSpinh \to \bigvee_{I \in \Part_{\mathrm{even}}} \ksp \langle 4| I| \rangle \vee \bigvee_{I \in \Part_{\mathrm{odd}}} \Sigma^{4|I|} F \vee \bigvee_{z \in Z} \Sigma^{\deg z} H\Z / 2\Z
    \end{equation}
    inducing $\theta$ in cohomology. Since $H^* \MSpinh$ is finitely generated in each degree (by Proposition~\ref{prop:bspin cohomology} and the Thom isomorphism), we can dualize to see that Equation~\ref{eq:splitting} induces an isomorphism in homology with coefficients in $\Z / 2\Z$. Hence this map is a $2$-complete equivalence. Since $\MSpinh$ and our wedge sum both have finitely generated homotopy groups in each degree (by Proposition~\ref{prop:no odd torsion}, Lemma~\ref{lem:homotopy of F}, and Bott periodicity), Equation~\ref{eq:splitting} is a $2$-local equivalence.
\end{proof}

\section{Calculating $\Spinh$ cobordism groups}\label{sec:computing groups}
According to Milnor, calculating Spin cobordism groups is a ``formidable computation'' \cite[p.~202]{Mil63}. The $\Spin$, $\Spinc$, and $\Spinh$ cobordism groups are all 2-primary, so their splitting at $p=2$ is sufficient to compute these groups. The formidable computation arises from two calculations: (i) the combinatorics of partitions that characterize the Anderson--Brown--Peterson splitting in the real and complex cases \cite{ABP67} and Theorem~\ref{thm:spinh splitting} in the quaternionic case, and (ii) counting the Eilenberg--Mac Lane summands. We provide code at \cite{git} that performs these manipulations for us, as well as tables of $\pi_*\MSpin$ (Table~\ref{table:spin bordism groups}), $\pi_*\MSpinc$ (Table~\ref{table:spinc bordism groups}), and $\pi_*\MSpinh$ (Table~\ref{table:spinh bordism groups}) for $0\leq *\leq 99$. Tables for $0\leq *\leq 19999$ are also available at \cite{git}.

\begin{rem}
A table of $\pi_*\MSpin$ for $0\leq *\leq 127$ (with an extra column recording additional information about the torsion) appears in~\cite[Section 10]{BN14}. Nevertheless, we include Table~\ref{table:spin bordism groups} for the reader's convenience. A table of $\pi_*\MSpinc$ for $0\leq *\leq 59$ is given in~\cite[p.~5]{BG87}. Values of $\pi_*\MSpinh$ are given for $0\leq *\leq 6$ in \cite[\S 3.5]{Hu22} and for $0\leq *\leq 30$ in \cite[\S 4]{Mil23}.
\end{rem}

\subsection{Computing rank and torsion}
We used code to generate Tables~\ref{table:spin bordism groups}, \ref{table:spinc bordism groups}, and~\ref{table:spinh bordism groups}. In this section, we will explain the math behind this code.

\subsubsection{Rank}
Theorems~\ref{thm:spin splitting}, \ref{thm:spinc splitting}, and~\ref{thm:spinh splitting} tell us that the ranks of $\pi_*\MSpin$, $\pi_*\MSpinc$, and $\pi_*\MSpinh$ are determined by the combinatorics of partitions and the homotopy groups of various connective covers of $\KO$, $\KU$, and $\KSp$, respectively. Putting this all together, we can derive formulas for the ranks of these bordism groups.

\begin{notn}
Let $p(i)=|\Part(i)|$ and $p_1(i)=|\Part_1(i)|$ denote the number of partitions of $i$ and the number of partitions of $i$ not containing 1, respectively. 
\end{notn}

\begin{lem}\label{lem:rank spin}
We have 
\[\rank\pi_n\MSpin=\begin{cases} p(m) & n=4m\geq 0,\\
0 & \text{otherwise}.
\end{cases}\]
\end{lem}
\begin{proof}
By Theorem~\ref{thm:spin splitting} and Bott periodicity for $\ko$ (see Table~\ref{table:bott periodicity}), we find that 
\begin{align}\label{eq:rank mspin}
\rank\pi_{8m}\MSpin&=\sum_{i=0}^{2m}p_1(i),\\
\rank\pi_{8m+4}\MSpin&=\sum_{i=0}^{2m+1}p_1(i).\nonumber
\end{align}
Partitions of $i$ containing 1 are sums of the form $1+s$ for $s$ a partition of $i-1$, so we have $p(i)=p_1(i)+p(i-1)$. That is, $p_1(i)=p(i)-p(i-1)$. By expanding the sums in Equation~\ref{eq:rank mspin} in terms of $p(i)$, we have 
\begin{align*}
\sum_{i=0}^kp_1(i)&=p_1(0)+\sum_{i=1}^k(p(i)-p(i-1))\\
&=p(k)-p(0)+p_1(0)\\
&=p(k).
\end{align*}
Thus $\rank\pi_{8m}\MSpin=p(2m)$ and $\rank\pi_{8m+4}\MSpin=p(2m+1)$, or more simply 
\[\rank\pi_{4m}\MSpin=p(m).\]Since the free part of $\pi_*\KO$ is concentrated in degrees $4m\geq 0$, it follows that $\pi_*\MSpin$ is torsion in all other degrees.   
\end{proof}

\begin{lem}\label{lem:rank spinc}
We have
\[\rank\pi_n\MSpinc=\begin{cases} \sum_{i=0}^m p(i) & n=4m\geq 0,\\
\sum_{i=0}^m p(i) & n=4m+2\geq 0,\\
0 & \text{otherwise}.
\end{cases}\]
\end{lem}
\begin{proof}
Recall that the free part of $\pi_*\ku\langle i\rangle$ is concentrated in degrees $2j\geq i$, and that each non-trivial free summand has rank 1. Thus by Theorem~\ref{thm:spinc splitting}, the rank of $\pi_{4m}\MSpinc$ is given by the sum $\sum_{i=0}^m p(i)$. The same argument holds for $\rank\pi_{4m+2}\MSpinc$, as the connective covers in Theorem~\ref{thm:spinc splitting} proceed in multiples of 4.   
\end{proof}

\begin{lem}\label{lem:rank spinh}
We have
\[\rank\pi_n\MSpinh=\begin{cases}
    \sum_{i=0}^m p(i) & n=4m\geq 0,\\
    0 & \text{otherwise}.
\end{cases}\]
\end{lem}
\begin{proof}
The free part of $\pi_*\ksp\langle i\rangle$ is concentrated in degrees $4j\geq i$, and each non-trivial free summand has rank 1. The same is true of the spectra $\Sigma^{8k+4}F$, since $\pi_*\ksp\cong\pi_*F$ (Lemma~\ref{lem:homotopy of F}). Theorem~\ref{thm:spinh splitting} thus implies that the rank of $\pi_{4m}\MSpinh$ is given by the sum 
\[\sum_{i=0}^{\lfloor m/2\rfloor}p(2i)+\sum_{i=0}^{\lceil m/2\rceil-1}p(2i+1)=\sum_{i=0}^mp(i),\]
as desired.
\end{proof}

Note that we have just shown that $\rank\pi_{4n}\MSpinh=\rank\pi_{4n}\MSpinc$.

\begin{cor}\label{cor:rank spinc and spinh}
We have $\rank\pi_{4n}\MSpinh=\rank\pi_{4n}\MSpinc=\rank\pi_{4n+2}\MSpinc$ for all $n$.
\end{cor}
\begin{proof}
This follows directly from Lemmas~\ref{lem:rank spinc} and~\ref{lem:rank spinh}.
\end{proof}

\subsubsection{Torsion}
Besides the partition numbers, one needs to count the Eilenberg--Mac Lane summands in order to determine these groups. To do this, we can use Hilbert--Poincar\'e series representing the dimension of various $\A$-modules in each degree. If $M$ is an $\A$-module, let $P(M)$ denote its Hilbert--Poincar\'e series.
\begin{prop}
    We have the following Hilbert--Poincar\'e series:
    \begin{align*}
        P(H^* \MSpinh) & = \prod_{n\geq 2}(1-t^n)^{-1}\cdot\prod_{r\geq 2}(1-t^{2^r+1}), \\
        P(\A) & = \prod_{n \geq 1}( 1 - t^{2^n - 1})^{-1}, \\
        P(H^*\ksp\langle 8k\rangle) & = \prod_{n\geq 3}(1-t^{2^n-1})^{-1}\cdot\frac{t^{8k}(1+t^2+t^3)}{(1-t^4)(1-t^6)}, \\
        P(H^* \Sigma^{8k + 4} F) & = \prod_{n\geq 3}(1-t^{2^n-1})^{-1}\cdot \frac{t^{8k + 4}(1 + t + 2t^2 + t^3 + t^4 + t^5)}{(1 - t^4)(1 - t^6)}.
    \end{align*}
\end{prop}
\begin{proof}
    For $P(H^*\MSpinh)=P(H^*\BSpinh)$, recall that the cohomology of $\BSpinh$ is a polynomial ring, so its Hilbert--Poincar\'e series is a product with a factor of $( 1 - t^n )^{-1}$ for each generator of degree $n$. There is a generator in degrees $i\geq 2$ such that $i\neq 2^{k+2}+1$ (for $k\geq 0$) by Proposition~\ref{prop:bspinh cohomology}.

    The series for $\A$ is given in \cite[Theorem 1.11]{ABP66}.

    Since $H^* \ksp\langle 8k\rangle \cong \Sigma^{8k} H^* \ksp \cong \Sigma^{8k} \udq_\A$, we can use the exact sequence
    \[
        0 \longrightarrow \Sigma^3 \udq_\A \longrightarrow \A \longrightarrow \A / \A \Sq^3 \longrightarrow 0,
    \]
    which implies $t^3 P(\udq_\A) = P(\A) - P(\A/\A\Sq^3)$. From \cite[Theorem 1.11]{ABP67}, we know
    \[
        P(\A/\A\Sq^3) = \prod_{n \geq 3} (1 - t^{2^n - 1})^{-1} \cdot (1 - t^4)^{-1} (1 - t^6)^{-1} (1 + t + t^2 + t^3 + t^4). 
    \]
    It follows that
    \begin{align*}
        P(H^*\ksp\langle 8k\rangle)&=t^{8k}P(\udq_\A)\\
        & = \frac{t^{8k}}{t^3} \prod_{n \geq 3} (1 - t^{2^n - 1})^{-1} \cdot \left( \frac{1}{(1 - t)(1 - t^3)} - \frac{1+t+t^2+t^3+t^4}{(1 - t^4)(1 - t^6)} \right) \\
        & = \prod_{n \geq 3} (1 - t^{2^n - 1})^{-1} \cdot \frac{t^{8k}(1 + t^2 + t^3)}{(1 - t^4)(1 - t^6)}.
    \end{align*}

    Finally, for $E_\A$ we use the exact sequence
    \[
        0 \longrightarrow \Sigma E_\A \longrightarrow \A \longrightarrow \A / (\A\Sq^1 + \A\Sq^2) \longrightarrow 0
    \]
    to get the equation $t P(E_\A) = P(\A) - P(\A/(\A\Sq^1 + \A\Sq^2))$. From \cite[Theorem 1.11]{ABP66}, we have
    \[
        P(\A/(\A\Sq^1 + \A\Sq^2)) = \prod_{n \geq 3} (1 - t^{2^n - 1})^{-1} \cdot (1 - t^4)^{-1} (1 - t^6)^{-1}.
    \]
    Therefore
    \begin{align*}
        P(E_\A) & = \frac{1}{t} \prod_{n \geq 3} (1 - t^{2^n - 1})^{-1} \cdot \left( (1 - t)^{-1} (1 - t^3)^{-1} - (1 - t^4)^{-1} (1 - t^6)^{-1} \right) \\
        & = \prod_{n \geq 3} (1 - t^{2^n - 1})^{-1} \cdot \frac{1 + t + 2t^2 + t^3 + t^4 + t^5}{(1 - t^4)(1 - t^6)}.\qedhere
    \end{align*}
\end{proof}

We can now describe the generating function for the number of $H\Z/2\Z$ summands in each degree.

\begin{cor}\label{cor:em generating series}
    Let $R(t):=\sum_{k\geq 0}r_kt^k$, where $r_k$ is the number of $\Sigma^k H\Z/2\Z$ summands of $\MSpinh$. Then
    \begin{align*}
        R(t)&=(1-t)\prod_{\substack{n\geq 2\\ n\neq 2^r\pm 1}}(1-t^n)^{-1}-\frac{1}{(1+t)(1+t^2)(1+t^3)}\sum_{k\geq 0}t^{8k}\bigg(p(2k)(1+t^2+t^3)\\
        &\quad+p(2k+1)(t^4+t^5+2t^6+t^7+t^8+t^9)\bigg).
    \end{align*}
\end{cor}
\begin{proof}
Theorem~\ref{thm:spinh splitting} implies that
\begin{align*}
    P(H^* \MSpinh) & = \sum_{k \geq 0}\left( \sum_{\Part(2k)} P(H^* \ksp\langle 8k\rangle) + \sum_{\Part(2k+1)} P(H^* \Sigma^{8k + 4} F) \right)+ R\cdot P(\A).
\end{align*}
Solving for $R$, we obtain
\begin{align*}
    R(t)&=\prod_{n\geq 2}(1-t^n)^{-1}\cdot\prod_{r\geq 2}(1-t^{2^r+1})\cdot\prod_{r\geq 1}(1-t^{2^r-1})\\
    &\quad-\prod_{r=1}^2(1-t^{2^r-1})\sum_{k\geq 0}\Bigg(\frac{t^{8k}}{(1-t^4)(1-t^6)}\Bigg(\sum_{\Part(2k)}(1+t^2+t^3)\\
    &\quad+\sum_{\Part(2k+1)}t^3((1+t)(1+t^2)(1+t^3)-1)\Bigg)\Bigg).
\end{align*}
The result follows from simplifying this expression.
\end{proof}

To give the generating function for the torsion part of $\pi_*\MSpinh$, it remains to add the torsion contributions from the $\ksp\langle 4(2k)\rangle$ and $\Sigma^{4(2k+1)}F$ summands. Bott periodicity for $\KSp$ and Lemma~\ref{lem:homotopy of F} give us the torsion, which we restate here for convenience.

\begin{lem}\label{lem:periodic torsion}
    Let $k\geq 0$. Then
    \[(\pi_*\ksp\langle 8k\rangle)_\mr{tors}\cong(\pi_*\Sigma^{8k+4}F)_\mr{tors}\cong\begin{cases}
        \Z/2 & * = 8n+5\text{ with }n\geq k,\\
        \Z/2 & * = 8n+6\text{ with }n\geq k,\\
        0 & \text{otherwise}.
    \end{cases}\]
\end{lem}

\begin{cor}\label{cor:torsion generating function}
    Let $R(t)$ be the generating series given in Corollary~\ref{cor:em generating series}. Let $S(t):=\sum_{k\geq 0}s_kt^k$, where $(\pi_k\MSpinh)_\mr{tors}\cong(\Z/2)^{s_k}$. Then
    \[S(t)=R(t)+(t^5+t^6)\sum_{k\geq 0}t^{8k}(p(2k)+p(2k+1)).\]
\end{cor}
\begin{proof}
    It suffices to show that the generating series for the torsion groups coming from $\ksp\langle 8k\rangle$ and $\Sigma^{8k+4}F$ is $(t^5+t^6)\sum_{k\geq 0}t^{8k}(p(2k)+p(2k+1))$. By Theorem~\ref{thm:spinh splitting}, the $\ksp\langle 8k\rangle$ terms are indexed over $\Part(2k)$, while the $\Sigma^{8k+4}F$ terms are indexed over $\Part(2k+1)$. By Lemma~\ref{lem:periodic torsion}, the coefficients $p(2k)$ and $p(2k+1)$ are each weighted by $t^{8k}(t^5+t^6)$.
\end{proof}

\begin{table}[p]
\centering
\caption{$\pi_n\MSpin\cong\mb{Z}^r\times(\mb{Z}/2\mb{Z})^t$.}\label{table:spin bordism groups}
\begin{tabular}{|r|rr|}
\hline
$n$ & $r$ & $t$\\
\hline
0 & 1 & 0\\
1 & 0 & 1\\
2 & 0 & 1\\
3 & 0 & 0\\
4 & 1 & 0\\
5 & 0 & 0\\
6 & 0 & 0\\
7 & 0 & 0\\
8 & 2 & 0\\
9 & 0 & 2\\
10 & 0 & 3\\
11 & 0 & 0\\
12 & 3 & 0\\
13 & 0 & 0\\
14 & 0 & 0\\
15 & 0 & 0\\
16 & 5 & 0\\
17 & 0 & 5\\
18 & 0 & 7\\
19 & 0 & 0\\
\hline
\end{tabular}
\hfill
\begin{tabular}{|r|rr|}
\hline
$n$ & $r$ & $t$\\
\hline
20 & 7 & 1\\
21 & 0 & 0\\
22 & 0 & 1\\
23 & 0 & 0\\
24 & 11 & 0\\
25 & 0 & 11\\
26 & 0 & 15\\
27 & 0 & 0\\
28 & 15 & 2\\
29 & 0 & 1\\
30 & 0 & 3\\
31 & 0 & 0\\
32 & 22 & 1\\
33 & 0 & 23\\
34 & 0 & 31\\
35 & 0 & 0\\
36 & 30 & 6\\
37 & 0 & 2\\
38 & 0 & 7\\
39 & 0 & 1\\
\hline
\end{tabular}
\hfill
\begin{tabular}{|r|rr|}
\hline
$n$ & $r$ & $t$\\
\hline
40 & 42 & 4\\
41 & 0 & 45\\
42 & 0 & 60\\
43 & 0 & 2\\
44 & 56 & 14\\
45 & 0 & 6\\
46 & 0 & 17\\
47 & 0 & 4\\
48 & 77 & 11\\
49 & 0 & 86\\
50 & 0 & 114\\
51 & 0 & 7\\
52 & 101 & 31\\
53 & 0 & 15\\
54 & 0 & 38\\
55 & 0 & 13\\
56 & 135 & 29\\
57 & 0 & 159\\
58 & 0 & 210\\
59 & 0 & 22\\
\hline
\end{tabular}
\hfill
\begin{tabular}{|r|rr|}
\hline
$n$ & $r$ & $t$\\
\hline
60 & 176 & 67\\
61 & 0 & 38\\
62 & 0 & 80\\
63 & 0 & 36\\
64 & 231 & 70\\
65 & 0 & 290\\
66 & 0 & 379\\
67 & 0 & 58\\
68 & 297 & 142\\
69 & 0 & 90\\
70 & 0 & 169\\
71 & 0 & 92\\
72 & 385 & 158\\
73 & 0 & 521\\
74 & 0 & 676\\
75 & 0 & 143\\
76 & 490 & 291\\
77 & 0 & 205\\
78 & 0 & 347\\
79 & 0 & 219\\
\hline
\end{tabular}
\hfill
\begin{tabular}{|r|rr|}
\hline
$n$ & $r$ & $t$\\
\hline
80 & 627 & 343\\
81 & 0 & 931\\
82 & 0 & 1196\\
83 & 0 & 330\\
84 & 792 & 589\\
85 & 0 & 448\\
86 & 0 & 698\\
87 & 0 & 494\\
88 & 1002 & 721\\
89 & 0 & 1658\\
90 & 0 & 2103\\
91 & 0 & 729\\
92 & 1255 & 1171\\
93 & 0 & 952\\
94 & 0 & 1385\\
95 & 0 & 1068\\
96 & 1575 & 1472\\
97 & 0 & 2948\\
98 & 0 & 3689\\
99 & 0 & 1550\\
\hline
\end{tabular}
\end{table}

\begin{table}[p]
\centering
\caption{$\pi_n\MSpinc\cong\mb{Z}^r\times(\mb{Z}/2\mb{Z})^t$.}\label{table:spinc bordism groups}
\begin{tabular}{|r|rr|}
\hline
$n$ & $r$ & $t$\\
\hline
0 & 1 & 0\\
1 & 0 & 0\\
2 & 1 & 0\\
3 & 0 & 0\\
4 & 2 & 0\\
5 & 0 & 0\\
6 & 2 & 0\\
7 & 0 & 0\\
8 & 4 & 0\\
9 & 0 & 0\\
10 & 4 & 1\\
11 & 0 & 0\\
12 & 7 & 0\\
13 & 0 & 0\\
14 & 7 & 1\\
15 & 0 & 0\\
16 & 12 & 0\\
17 & 0 & 0\\
18 & 12 & 3\\
19 & 0 & 0\\
\hline
\end{tabular}
\hfill
\begin{tabular}{|r|rr|}
\hline
$n$ & $r$ & $t$\\
\hline
20 & 19 & 1\\
21 & 0 & 0\\
22 & 19 & 5\\
23 & 0 & 0\\
24 & 30 & 2\\
25 & 0 & 0\\
26 & 30 & 9\\
27 & 0 & 0\\
28 & 45 & 4\\
29 & 0 & 1\\
30 & 45 & 14\\
31 & 0 & 1\\
32 & 67 & 8\\
33 & 0 & 2\\
34 & 67 & 24\\
35 & 0 & 2\\
36 & 97 & 15\\
37 & 0 & 4\\
38 & 97 & 37\\
39 & 0 & 5\\
\hline
\end{tabular}
\hfill
\begin{tabular}{|r|rr|}
\hline
$n$ & $r$ & $t$\\
\hline
40 & 139 & 26\\
41 & 0 & 8\\
42 & 139 & 59\\
43 & 0 & 10\\
44 & 195 & 44\\
45 & 0 & 16\\
46 & 195 & 90\\
47 & 0 & 20\\
48 & 272 & 72\\
49 & 0 & 29\\
50 & 272 & 138\\
51 & 0 & 36\\
52 & 373 & 116\\
53 & 0 & 51\\
54 & 373 & 207\\
55 & 0 & 64\\
56 & 508 & 183\\
57 & 0 & 88\\
58 & 508 & 311\\
59 & 0 & 110\\
\hline
\end{tabular}
\hfill
\begin{tabular}{|r|rr|}
\hline
$n$ & $r$ & $t$\\
\hline
60 & 684 & 284\\
61 & 0 & 148\\
62 & 684 & 458\\
63 & 0 & 184\\
64 & 915 & 434\\
65 & 0 & 243\\
66 & 915 & 676\\
67 & 0 & 301\\
68 & 1212 & 658\\
69 & 0 & 391\\
70 & 1212 & 987\\
71 & 0 & 483\\
72 & 1597 & 985\\
73 & 0 & 619\\
74 & 1597 & 1436\\
75 & 0 & 762\\
76 & 2087 & 1462\\
77 & 0 & 967\\
78 & 2087 & 2074\\
79 & 0 & 1186\\
\hline
\end{tabular}
\hfill
\begin{tabular}{|r|rr|}
\hline
$n$ & $r$ & $t$\\
\hline
80 & 2714 & 2152\\
81 & 0 & 1490\\
82 & 2714 & 2986\\
83 & 0 & 1820\\
84 & 3506 & 3145\\
85 & 0 & 2268\\
86 & 3506 & 4273\\
87 & 0 & 2762\\
88 & 4508 & 4564\\
89 & 0 & 3418\\
90 & 4508 & 6095\\
91 & 0 & 4147\\
92 & 5763 & 6583\\
93 & 0 & 5099\\
94 & 5763 & 8651\\
95 & 0 & 6167\\
96 & 7338 & 9440\\
97 & 0 & 7540\\
98 & 7338 & 12237\\
99 & 0 & 9090\\
\hline
\end{tabular}
\end{table}

\begin{table}[p]
\centering
\caption{$\pi_n\MSpinh\cong\mb{Z}^r\times(\mb{Z}/2\mb{Z})^t$.}\label{table:spinh bordism groups}
\begin{tabular}{|r|rr|}
\hline
$n$ & $r$ & $t$\\
\hline
0 & 1 & 0\\
1 & 0 & 0\\
2 & 0 & 0\\
3 & 0 & 0\\
4 & 2 & 0\\
5 & 0 & 2\\
6 & 0 & 2\\
7 & 0 & 0\\
8 & 4 & 0\\
9 & 0 & 1\\
10 & 0 & 1\\
11 & 0 & 0\\
12 & 7 & 0\\
13 & 0 & 7\\
14 & 0 & 8\\
15 & 0 & 2\\
16 & 12 & 1\\
17 & 0 & 3\\
18 & 0 & 3\\
19 & 0 & 1\\
\hline
\end{tabular}
\hfill
\begin{tabular}{|r|rr|}
\hline
$n$ & $r$ & $t$\\
\hline
20 & 19 & 2\\
21 & 0 & 21\\
22 & 0 & 25\\
23 & 0 & 7\\
24 & 30 & 5\\
25 & 0 & 10\\
26 & 0 & 11\\
27 & 0 & 7\\
28 & 45 & 10\\
29 & 0 & 55\\
30 & 0 & 64\\
31 & 0 & 22\\
32 & 67 & 20\\
33 & 0 & 31\\
34 & 0 & 35\\
35 & 0 & 27\\
36 & 97 & 36\\
37 & 0 & 132\\
38 & 0 & 156\\
39 & 0 & 66\\
\hline
\end{tabular}
\hfill
\begin{tabular}{|r|rr|}
\hline
$n$ & $r$ & $t$\\
\hline
40 & 139 & 65\\
41 & 0 & 87\\
42 & 0 & 100\\
43 & 0 & 86\\
44 & 195 & 111\\
45 & 0 & 307\\
46 & 0 & 360\\
47 & 0 & 180\\
48 & 272 & 188\\
49 & 0 & 232\\
50 & 0 & 269\\
51 & 0 & 249\\
52 & 373 & 310\\
53 & 0 & 689\\
54 & 0 & 804\\
55 & 0 & 465\\
56 & 508 & 503\\
57 & 0 & 592\\
58 & 0 & 685\\
59 & 0 & 662\\
\hline
\end{tabular}
\hfill
\begin{tabular}{|r|rr|}
\hline
$n$ & $r$ & $t$\\
\hline
60 & 684 & 803\\
61 & 0 & 1514\\
62 & 0 & 1755\\
63 & 0 & 1154\\
64 & 915 & 1267\\
65 & 0 & 1445\\
66 & 0 & 1663\\
67 & 0 & 1659\\
68 & 1212 & 1972\\
69 & 0 & 3273\\
70 & 0 & 3767\\
71 & 0 & 2746\\
72 & 1597 & 3039\\
73 & 0 & 3402\\
74 & 0 & 3891\\
75 & 0 & 3968\\
76 & 2087 & 4636\\
77 & 0 & 6971\\
78 & 0 & 7962\\
79 & 0 & 6315\\
\hline
\end{tabular}
\hfill
\begin{tabular}{|r|rr|}
\hline
$n$ & $r$ & $t$\\
\hline
80 & 2714 & 7010\\
81 & 0 & 7757\\
82 & 0 & 8808\\
83 & 0 & 9121\\
84 & 3506 & 10510\\
85 & 0 & 14645\\
86 & 0 & 16609\\
87 & 0 & 14094\\
88 & 4508 & 15640\\
89 & 0 & 17174\\
90 & 0 & 19367\\
91 & 0 & 20280\\
92 & 5763 & 23104\\
93 & 0 & 30368\\
94 & 0 & 34201\\
95 & 0 & 30607\\
96 & 7338 & 33906\\
97 & 0 & 37043\\
98 & 0 & 41508\\
99 & 0 & 43818\\
\hline
\end{tabular}
\end{table}

\subsection{Growth rates}
Since $\Spin$, $\Spinc$, and $\Spinh$ bordism groups are combinatorially determined, we can analyze their growth combinatorially as well. The asymptotic growth of partitions is due to Hardy and Ramanujan~\cite{HR1,HR2}, which allows us to prove the asymptotic growth of the ranks of these bordism groups. 

\begin{prop}\label{prop:rank growth}
Let $r_nM:=\rank\pi_nM$, where $M$ is any spectrum. Then 
\begin{align*}
\frac{\exp(\pi\sqrt{2n/3})}{4n\sqrt{3}}&\sim r_{4n}\MSpin,\\
\frac{\exp(\pi\sqrt{2n/3})}{2\pi\sqrt{2n}}&\sim r_{4n}\MSpinc=r_{4n+2}\MSpinc=r_{4n}\MSpinh.
\end{align*}
\end{prop}
\begin{proof}
Hardy--Ramanujan \cite{HR1,HR2} proved that 
\[p(i)\sim\frac{\exp(\pi\sqrt{2i/3})}{4i\sqrt{3}}.\]
Thus $r_{4n}\MSpin\sim\frac{\exp(\pi\sqrt{2n/3})}{4n\sqrt{3}}$ by Lemma~\ref{lem:rank spin}. We proved the equality $\rank\pi_{4n}\MSpinc=\rank\pi_{4n}\MSpinh=\sum_{i=0}^n p(i)$ in Corollary~\ref{cor:rank spinc and spinh}. It thus suffices to calculate
\[\sum_{i=0}^n p(i)\sim\frac{\exp(\pi\sqrt{2n/3})}{2\pi\sqrt{2n}},\]
which follows from Hardy--Ramanujan's asymptotic formula for $p(n)$, Gupta's formula \cite{Gup46}
\[\sum_{i=0}^{n-1}p(i)\sim\frac{p(n)\sqrt{6n}}{\pi},\]
and the calculation $\displaystyle\lim_{n\to\infty}\frac{p(n+1)\sqrt{n+1}}{p(n)\sqrt{n}}=1$.
\end{proof}

\begin{rem}
The asymptotic growth of partitions (and hence the ranks of $\pi_*\MSpin$, $\pi_*\MSpinc$, and $\pi_*\MSpinh$) are calculated using the circle method. This same method could be used to calculate the growth of the torsion parts as well. For example, the growth of the torsion part of $\pi_*\MSpinh$ is given by the growth of the coefficients of $S(t)$, whose poles all lie on the unit circle. We will not investigate the asymptotics of the torsion parts of $\pi_*\MSpinh$ here.
\end{rem}

\section{Characterizing $\Spinh$ cobordism classes}\label{sec:characteristic classes}
In Section \ref{sec:computing groups}, we saw that we can explicitly compute $\Spinh$ bordism groups up to any desired degree (contingent upon having enough computational power). However, these computations only describe the $\Spinh$ bordism groups abstractly. What we really want from $\pi_*\MSpinh$ is an understanding of the geometry of Spin$^h$ manifolds up to cobordism. 

Theorem~\ref{thm:main} implies that the $\KSp$-characteristic classes given in Definition~\ref{def:ksp classes}, together with $H\Z/2\Z$-characteristic classes, can be used to distinguish $\Spinh$ cobordism classes (by evaluating on an appropriate homology class). In this section, we will show that instead of using the elephant classes for odd partitions, it suffices to use $\KSp$-Pontryagin classes for all partitions.

\begin{setup}
Recall that a $\Spinh$ manifold is a smooth compact manifold without boundary, equipped with a $\Spinh$ structure on its stable normal bundle $\nu$. If $M$ is a smooth compact $n$-manifold without boundary, the Pontryagin--Thom construction gives a map of spectra $\theta : \Sigma^n \Sph \to \Th(\nu)$, where $\Th(\nu)$ is the Thom spectrum of the stable normal bundle of $M$.

The unit map $\Sph \to \KO \cong \Sph \wedge \KO$ induces a $\KO$-homology class $1 \in \KO_0 \Sph$. Shifting and then transferring along $\theta$ gives us a class $\theta_* 1 \in \KO_n \Th(\nu)$. We also have a class $a \in \KSp^0 \Th(\nu)$ given by the composition of $\hw : \Th(\nu) \to \MSpinh$ (coming from the $\Spinh$ structure on $M$) and the Atiyah--Bott--Shapiro map $\vphi^h : \MSpinh \to \KSp$ (Proposition~\ref{prop:abs map}). The $\KSp$-homology class $\theta_*1\frown a$ can be thought of as a sort of $\KSp$-fundamental class of $M$ as a $\Spinh$ manifold.
\end{setup}

\begin{defn}
    Let $I$ be a partition and $M$ a $\Spinh$-manifold. The \textit{$I^\textsuperscript{th}$ $\KSp$- characteristic number} of $M$ is $\langle \pi^I_h(\nu), \theta_* 1 \frown a \rangle \in \KSp_n$, where $\pi^I_h\in\KO^0(\BSpinh)$ is the $\KO$-Pontryagin class.
\end{defn}

Diagramatically, the $\KSp$-homology class $\theta_* 1 \frown a$ on $M$ is given by
\[
    \begin{tikzcd}
        \arrow["{\theta_* 1}", from=2-1, to=1-1]
        \arrow["{\delta \wedge \id}", from=1-1, to=1-2]
        \arrow["{\id \wedge a \wedge \id}", from=1-2, to=1-3]
        \arrow["{\id \wedge \mu}", from=1-3, to=2-3]
        \Th(\nu) \wedge \KO & M_+ \wedge \Th(\nu) \wedge \KO & M_+ \wedge \KSp \wedge \KO\\
        \Sigma^n \Sph &&  M_+ \wedge \KSp,
    \end{tikzcd}
\]
where $\delta$ is the Thom diagonal and $\mu:\KSp\wedge\KO\to\KSp$ is the $\KO$-module structure. The $I^\textsuperscript{th}$ $\KSp$-characteristic number of $M$ is then given by
\[
    \Sigma^n \Sph \xrightarrow{\theta_*1\frown a} M_+ \wedge \KSp \xrightarrow{\pi^I_h(\nu)\wedge\id} \KO \wedge \KSp \xrightarrow{\ \mu\ } \KSp.
\]

The main lemma of this section is that $\KSp$-characteristic numbers are indeed related to our $\KSp$-characteristic classes.

\begin{lem}\label{lem:char numbers}
    If $M$ is a $\Spinh$-manifold, then the $I^\textsuperscript{th}$ $\KSp$-characteristic number can be computed as the composite
    \[
        \begin{tikzcd}
            \arrow["{\theta}", from=1-1, to=1-2]
            \arrow["\hw", from=1-2, to=1-3]
            \arrow["\kappa^I",from=1-3, to=1-4]
            \Sigma^n \Sph & \Th(\nu) & \MSpinh & \KSp,
        \end{tikzcd}
    \]
    where $\kappa^I$ is the $I^\textsuperscript{th}$ $\KSp$-Pontryagin class (see Remark~\ref{rem:odd ksp-pontryagin}).
\end{lem}
\begin{proof}
    As before, let $\delta:\Th(\nu)\to M_+\wedge\Th(\nu)$ denote the Thom diagonal. Let $e:\Sph\to\KO$ denote the unit map and $\mu:\KSp\wedge\KO\to\KSp$ the $\KO$-module multiplication of $\KSp$. The diagram
    \begin{equation}\label{eq:3x3}
        \begin{tikzcd}
            \arrow["{\theta \wedge \id}", from=1-1, to=1-2]
            \arrow["{\delta \wedge \id}", from=1-2, to=1-3]
            \arrow["{\id \wedge e}", from=2-1, to=1-1]
            \arrow["{\id \wedge e}", from=2-2, to=1-2]
            \arrow["{\id \wedge \id \wedge e}", from=2-3, to=1-3]
            \arrow["{\theta \wedge \id}", from=2-1, to=2-2]
            \arrow["{\delta \wedge \id}", from=2-2, to=2-3]
            \arrow["{\rho}", from=3-1, to=2-1]
            \arrow["{\rho}", from=3-2, to=2-2]
            \arrow["{\rho}", from=3-3, to=2-3]
            \arrow["{\theta}", from=3-1, to=3-2]
            \arrow["{\delta}", from=3-2, to=3-3]
            \Sigma^n \Sph \wedge \KO & \Th(\nu) \wedge \KO & M_+ \wedge \Th(\nu) \wedge \KO \\
            \Sigma^n \Sph \wedge \Sph & \Th(\nu) \wedge \Sph & M_+ \wedge \Th(\nu) \wedge \Sph \\
            \Sigma^n \Sph & \Th(\nu) & M_+ \wedge \Th(\nu)
        \end{tikzcd}
    \end{equation}
    commutes, because the unit isomorphisms in a symmetric monoidal category are natural. Next, the diagram
    \begin{equation}\label{eq:4x2}
        \begin{tikzcd}
            \arrow["{\delta}", from=1-1, to=1-2]
            \arrow["{\pi^I_h(\nu) \wedge a}", from=1-2, to=1-3]
            \arrow["{\mu}", from=1-3, to=1-4]
            \arrow["{\hw}", from=1-1, to=2-1]
            \arrow["{\nu\wedge\hw}", from=1-2, to=2-2]
            \arrow["{\id}", from=1-3, to=2-3]
            \arrow["{\id}", from=1-4, to=2-4]
            \arrow["{\delta}", from=2-1, to=2-2]
            \arrow["{\pi^I_h \wedge \varphi^h}", from=2-2, to=2-3]
            \arrow["{\mu}", from=2-3, to=2-4]
            \Th(\nu) & M_+ \wedge \Th(\nu) & \KO \wedge \KSp & \KSp \\
            \MSpinh & \BSpinh \wedge \MSpinh & \KO \wedge \KSp & \KSp
        \end{tikzcd}
    \end{equation}
    commutes by naturality of the Thom diagonal. Finally, consider the diagram
    \begin{equation}\label{eq:3x2}
        \begin{tikzcd}
            \arrow["{\id \wedge a \wedge \id}", from=1-1, to=1-2]
            \arrow["{\id \wedge \id \wedge e}", from=2-1, to=1-1]
            \arrow["{\id \wedge \mu}", from=1-2, to=2-2]
            \arrow["{\rho}", from=3-1, to=2-1]
            \arrow["{\pi^I_h (\nu) \wedge \id}", from=2-2, to=3-2]
            \arrow["{\pi^I_h (\nu) \wedge a}", from=3-1, to=3-2]
            M_+ \wedge \Th(\nu) \wedge \KO & M_+ \wedge \KSp \wedge \KO \\
            M_+ \wedge \Th(\nu) \wedge \Sph & M_+ \wedge \KSp \\
            M_+ \wedge \Th(\nu) & \KO \wedge \KSp.
        \end{tikzcd}
    \end{equation}
    To see that Diagram~\ref{eq:3x2} commutes, we use the identity axiom for $\KSp$ as a $\KO$-module, which states that the diagram
    \[
        \begin{tikzcd}
            \arrow["{\id \wedge e}", from=1-1, to=1-2]
            \arrow["{\rho}", from=2-1, to=1-1]
            \arrow["{\mu}", from=1-2, to=2-2]
            \arrow["{\id}", from=2-1, to=2-2]
            \KSp \wedge \Sph & \KSp \wedge \KO \\
            \KSp & \KSp
        \end{tikzcd}
    \]
    commutes. Therefore the diagram
    \[
        \begin{tikzcd}
            \arrow["{a \wedge \id}", from=1-1, to=1-2]
            \arrow["{\id \wedge e}", from=1-2, to=1-3]
            \arrow["{\rho}", from=2-1, to=1-1]
            \arrow["{\rho}", from=2-2, to=1-2]
            \arrow["{\mu}", from=1-3, to=2-3]
            \arrow["{a}", from=2-1, to=2-2]
            \arrow["{\id}", from=2-2, to=2-3]
            \Th(\nu) \wedge \Sph & \KSp \wedge \Sph & \KSp \wedge \KO \\
            \Th(\nu) & \KSp & \KSp
        \end{tikzcd}
    \]
    commutes. Since $(\id \wedge e) \circ (a \wedge \id) = (a \wedge \id) \circ (\id \wedge e)$, the diagram
    \begin{equation}\label{eq:swapped}
        \begin{tikzcd}
            \arrow["{\id \wedge e}", from=1-1, to=1-2]
            \arrow["{a \wedge \id}", from=1-2, to=1-3]
            \arrow["{\rho}", from=2-1, to=1-1]
            \arrow["{\mu}", from=1-3, to=2-3]
            \arrow["{a}", from=2-1, to=2-2]
            \arrow["{\id}", from=2-2, to=2-3]
            \Th(\nu) \wedge \Sph & \Th(\nu) \wedge \KO & \KSp \wedge \KO \\
            \Th(\nu) & \KSp & \KSp
        \end{tikzcd}
    \end{equation}
    commutes. Smashing Diagram~\ref{eq:swapped} on the left with $M_+$, we find that
    \begin{equation}\label{eq:pentagon}
        \begin{tikzcd}
            \arrow["{\id \wedge a \wedge \id}", from=1-1, to=1-2]
            \arrow["{\id \wedge \id \wedge e}", from=2-1, to=1-1]
            \arrow["{\id \wedge \mu}", from=1-2, to=2-2]
            \arrow["{\rho}", from=3-1, to=2-1]
            \arrow["{\id \wedge a}", from=3-1, to=2-2]
            M_+ \wedge \Th(\nu) \wedge \KO & M_+ \wedge \KSp \wedge \KO \\
            M_+ \wedge \Th(\nu) \wedge \Sph & M_+ \wedge \KSp \\
            M_+ \wedge \Th(\nu)
        \end{tikzcd}
    \end{equation}
    commutes. To complete the commutativity of Diagram~\ref{eq:3x2}, we observe that
    \[\begin{tikzcd}
            \arrow["{\id \wedge a}", from=2-1, to=1-2]
            \arrow["\pi^I_h(\nu)\wedge\id", from=1-2, to=2-2]
            \arrow["\pi^I_h(\nu)\wedge a", from=2-1, to=2-2]
            & M_+ \wedge \KSp \\
            M_+ \wedge \Th(\nu) & \KO\wedge\KSp
        \end{tikzcd}\]
    commutes. To conclude the lemma, we stitch together Diagrams~\ref{eq:3x3}, \ref{eq:4x2}, and \ref{eq:3x2} and take two different routes $\Sigma^n\Sph\to\KSp$. For the reader's convenience, we depict these routes in Figure~\ref{fig:char number}.
\end{proof}

\begin{figure}
    \begin{tikzpicture}[scale=1.75]
        \foreach \i in {(0,1),(0,2),(0,3),(1,0),(1,1),(1,2),(1,3),(2,0),(2,1),(2,2),(2,3),(3,0),(3,1),(3,2),(3,3),(4,0),(4,1)}{
        \draw[black!30,fill=black!30] \i circle (1pt);}
        \draw[black!30] (0,1) -- (0,2) -- (0,3) -- (1,3) -- (1,2) -- (1,1) -- (1,0) -- (2,0) -- (2,1) -- (2,2) -- (2,3) -- (3,3) -- (3,2) -- (3,1) -- (3,0) -- (4,0) -- (4,1) -- (3,1) -- (2,1) -- (1,1) -- (0,1);
        \draw[black!30] (1,3) -- (2,3);
        \draw[black!30] (0,2) -- (2,2);
        \draw[black!30] (2,0) -- (3,0);
        \node[fill=white] (A) at (0,1) {$\Sigma^n\Sph$};
        \node[fill=white] (B) at (1,1) {$\Th(\nu)$};
        \node[fill=white] (C) at (1,0) {$\MSpinh$};
        \node[fill=white] (D) at (4,0) {$\KSp$};
        \node[fill=white] (E) at (1,3) {$\Th(\nu)\wedge\KSp$};
        \node[fill=white] (F) at (3,2) {$M_+\wedge\KSp$};
        \node[fill=white] (G) at (3,1) {$\KO\wedge\KSp$};
        \draw[thick,->] (A) -- (B) node[midway,above] {\tiny$\theta$};
        \draw[thick,->] (B) -- (C) node[midway,left] {\tiny$\hw$};
        \draw[thick,->] (C) -- (D) node[midway,above] {\tiny$\kappa^I$};
        \draw[thick,->] (A) -- (0,2) --node[midway,above] {\tiny$\theta_*1$} (1,2) -- (E);
        \draw[thick,->] (E) -- (2,3) -- node[midway,above] {\tiny$(-)\frown a$}(3,3) -- (F);
        \draw[thick,->] (F) -- (G) node[midway,right] {\tiny$\pi^I_h(\nu)\wedge\id$};
        \draw[thick,->] (G) -- (4,1) -- (D) node[midway,right] {\tiny$\mu$};
    \end{tikzpicture}
    \caption{Two ways of computing $\KSp$-characteristic numbers}\label{fig:char number}
\end{figure}

We are now ready to prove Theorem~\ref{thm:cnumber condition}, which is an analog of \cite[Corollary 1]{Wal60} and \cite[Corollary 2.3]{ABP67}.

\begin{thm}\label{thm:cnumber condition with proof}
    Two $\Spinh$-manifolds are $\Spinh$-cobordant if and only if their $\KSp$- and $\Z / 2\Z$-characteristic numbers agree.
\end{thm}
\begin{proof}
    First, we form the composition
    \[
        \begin{tikzcd}
            \arrow[from=2-1, to=1-2,"\phi",bend left,end anchor=west]
            \arrow[from=1-2, to=2-2,"\psi"]
            \arrow[from=2-1, to=2-2,"\psi\circ\phi"]
            & \bigvee_{I \in \Part_{\mathrm{even}}} \ksp \langle 4| I| \rangle \vee \bigvee_{I \in \Part_{\mathrm{odd}}} \Sigma^{4| I|} F \vee \bigvee_{z \in Z} \Sigma^{\deg z} H\Z / 2\Z \\
            \MSpinh & \bigvee_{I \in \Part} \ksp \langle 4| I| \rangle \vee \bigvee_{z \in Z} \Sigma^{\deg z} H\Z / 2\Z,
        \end{tikzcd}
    \]
    where $\psi$ is given by the identity maps on the $\ksp\langle 4|I|\rangle$ (when $I\in\Part_\mr{even}$) and $\Sigma^{\deg z}H\Z / 2\Z$ summands, and by the map $\Sigma^{4|I|} F \to \ksp\langle 4|I|\rangle$ for $I\in\Part_\mr{odd}$. The map $\phi$ is the splitting of Theorem~\ref{thm:main}, so $\phi$ is a $2$-local equivalence. Taking homotopy groups, we find that 
    \[\ker\left((\psi\circ\phi)_*:\pi_*\MSpinh\to\bigoplus_{I\in\Part}\pi_*\ksp\langle 4|I|\rangle\oplus\bigoplus_{z\in Z}\pi_*\Sigma^{\deg{z}}H\Z/2\Z\right)\]
    is trivial. To see this, note that $\phi$ induces an isomorphism (in particular, an injection) on homotopy groups. Similarly, $\psi$ induces an injection on homotopy groups, since $\psi_*$ is a direct sum of identity maps and copies of the inclusion $2\Z \to \Z$. Thus $(\psi\circ\phi)_*=\psi_*\circ\phi_*$ is an injection.

    The previous paragraph suggests that $\psi\circ\phi$ can separate $\Spinh$-cobordism classes. Indeed, two $\Spinh$ manifolds $M_1$ and $M_2$ are $\Spinh$-cobordant if and only if the class of $M = M_1 - M_2$ corresponds to $0\in\pi_*\MSpinh$ (under Pontryagin--Thom). Since $(\psi\circ\phi)_*$ is injective, $[M]$ corresponds to $0\in\pi_*\MSpinh$ if and only if $[M]$ maps to zero in each $\pi_*\ksp \langle 4|I| \rangle$ and each $\pi_*\Sigma^{\deg z} H\Z / 2\Z$. 
    
    It remains to show that $(\psi\circ\phi)_*$ is the direct sum of the $\KSp$- and $\Z/2\Z$-characteristic numbers. If $I$ is a partition, then the element in $\pi_*\ksp \langle 4| I| \rangle$ determined by $[M]$ is
    \[
        \Sigma^n \Sph \longrightarrow \Th(\nu) \longrightarrow \MSpinh \longrightarrow \ksp \langle 4|I| \rangle,
    \]
    which is precisely the $I^\textsuperscript{th}$ $\KSp$-characteristic number of $M$ by Lemma~\ref{lem:char numbers}. Similarly, for $z \in Z$, the element of $\pi_*\Sigma^{\deg z} H\Z / 2\Z$ corresponding to $[M]$ is the sum of various ordinary $\Z / 2\Z$-characteristic numbers of $M$ arising from the expression of $z$ in the polynomial basis of the Stiefel--Whitney classes. So if all the $\KSp$- and $\Z / 2\Z$-characteristic numbers of $M_1$ and $M_2$ agree, then they vanish for $M$, and therefore the element of $\pi_*\MSpinh$ determined by $M$ is zero. 
    
    Conversely, if $M_1$ and $M_2$ are $\Spinh$-cobordant, then the $\KSp$-characteristic numbers of $M$ all vanish. Moreover, $M_1$ and $M_2$ being $\Spinh$-cobordant implies that their underlying unoriented manifolds are cobordant, and two unoriented manifolds are cobordant if and only if their Stiefel--Whitney numbers agree \cite{Tho54}. It follows that the Stiefel--Whitney numbers of $M$ are all zero as well.
\end{proof}

\begin{rem}
    Theorem~\ref{thm:cnumber condition with proof} can be summarized by saying that two $\Spinh$ manifolds are $\Spinh$-cobordant if and only if their underlying unoriented manifolds are cobordant and all of their $\KSp$-characteristic numbers agree.
\end{rem}

\section{Potential applications}\label{sec:applications}
In this section, we list a few more problems of interest in Spin$^h$ cobordism theory.

\subsection{Explicit representatives of generators}
As seen in Section~\ref{sec:computing groups}, we can now calculate the bordism groups $\Omega^{\Spinh}_*$ in any degree (within the bounds of time and computational power). It would be desirable to have explicit Spin$^h$ manifolds whose classes are generators in $\Omega^{\Spinh}_*$.

\begin{prob}\label{prob:representatives}
    Write $\pi_n\MSpinh\cong\mb{Z}^{r_n}\times(\mb{Z}/2\mb{Z})^{t_n}$. Given a dimension $n$, find $n$-dimensional Spin$^h$ manifolds $M_1,\ldots,M_{r_n},N_1,\ldots,N_{t_n}$ such that $[M_1],\ldots,[N_{t_n}]$ generate $\Omega^{\Spinh}_n$.
\end{prob}

\begin{ex}
    Since $\pi_n\MSpinh$ is trivial for $n\in\{1,2,3,7,11\}$, Problem~\ref{prob:representatives} is trivial in these dimensions. We can also make a few remarks in some small non-trivial dimensions.
    \begin{enumerate}[(i)]
    \item In dimension 0, $\pi_0 \MSpinh \cong \Z$ is generated by a point with a choice of one of two $\Spinh$ structures.

    \item In dimension 4, one can use the Adams spectral sequence for the cofiber of the map $\MSpinc \to \MSpinh$ to show that the map $\pi_4 \MSpinc \to \pi_4 \MSpinh$ is injective. However, we do not know how to characterize this injection in terms of $\Spinc$ and $\Spinh$ manifolds.

    \item In dimension 5, $\pi_5\MSpinh\cong(\Z/2\Z)^2$ is generated by the Wu manifold $W = \SU(3) / \SO(3)$ and $S^1\times S^4$ with a non-bounding $\Spinh$ structure \cite[p.~37]{Hu22}. 
    
    Recall that $W$ admits a $\Spinh$-structure \cite[Theorem 1.4]{AM21}, while $W$ does not admit a $\Spinc$-structure \cite[p.~393]{MR1989}. Moreover, $H^5(W;\Z/2\Z)$ is generated by $w_2 w_3$ of the stable normal bundle \cite[p.~393]{MR1989}, so we are able to detect one of its $\KSp$-characteristic numbers using ordinary cohomology. 
    
    The class $w_2 w_3 U_h \in H^* \MSpinh$ comes from the lowest elephant class $\MSpinh \to \Sigma^4 F$. Indeed, there are no $\Sigma^4 H\Z / 2\Z$ summands in the splitting (see Table~\ref{table:spinh bordism groups}), and the only nonvanishing degree four cohomology class of $\ksp$ is $\Sq^4 y_0$, which maps to $w_4 U_h$. Because the Pontryagin--Thom map $\Sigma^5 \Sph \to \Th(\nu)$ maps the generator of $H^5 \Sigma^5 \Sph$ to $[W]\smile U_h$, the map $\Sigma^5 \Sph \to \Th(\nu)$ must send $w_2 w_3 U$ to the generator of $H^5 \Sigma^5 \Sph$ in cohomology. We can thus conclude that the map $\Sigma^5 \Sph \to \KSp$ is nontrivial, so the $\KSp$-characteristic number of $W$ determined by the partition $(1)$ is $1$.

    This determines one of the components of $[W] \in \Omega_5^{\Spinh} \cong (\Z / 2\Z)^2$. Determining the other component would likely require us to understand the K-theory of $W$.

    \item In dimension 6, the Adams spectral sequence for the cofiber of $\MSpinc \to \MSpinh$ can be used to show that $\pi_6 \MSpinc \to \pi_6 \MSpinh$ is surjective. As in dimension 4, we do not know how to characterize this surjection in terms of $\Spinc$ and $\Spinh$ manifolds.

    In private communication to the authors, Hu suggested $\U(3)/\SO(3)$ and $S^1\times S^1\times S^4$ as natural candidates for generators of $\pi_6\MSpinh\cong(\Z/2\Z)^2$. By Theorem~\ref{thm:cnumber condition}, one could verify or refute this suggestion by computing the $\KSp$-characteristic numbers of these two manifolds.
    \end{enumerate}
\end{ex}

One question related to Problem~\ref{prob:representatives} is about the relationship between free $\Spinc$ and $\Spinh$ cobordism classes.

\begin{ques}\label{ques:spinc and spinh}
    We saw in Corollary~\ref{cor:rank spinc and spinh} that $\rank\pi_{4n}\MSpinh=\rank\pi_{4n}\MSpinc$. Is there a geometric explanation of this fact? In other words, is there a procedure for producing generators of the free part of $\Omega^{\MSpinh}_{4n}$ from generators of the free part of $\Omega^{\MSpinc}_{4n}$, and vice versa?
\end{ques}

\begin{rem}
    Question~\ref{ques:spinc and spinh} is related to the injection $\pi_4\MSpinc\to\pi_4\MSpinh$ and surjection $\pi_6\MSpinc\to\pi_6\MSpinh$ coming from the Adams spectral sequence. Neither of these maps are isomorphisms, but they both have $(\Z/2\Z)^2$ as their (co)kernel.
\end{rem}

\begin{rem}
Debray and Krulewski have shown that the inclusion $\Spinc_n\hookrightarrow\Spinh_n$ induces a map $\Omega^{\Spinc}_{4k}\to\Omega^{\Spinh}_{4k}$ that is an isomorphism after tensoring with $\mb{Z}[1/2]$ \cite{DK25}. This gives a geometric explanation for Corollary~\ref{cor:rank spinc and spinh}, thereby answering the first part of Question~\ref{ques:spinc and spinh}. This also suggests that constructing generators of the free part of $\Omega^{\Spinh}_{4k}$ from generators of the part of $\Omega^{\Spinc}_{4k}$ would be quite difficult.
\end{rem}

\subsection{$\MSpin$-module structure of $\MSpinh$}
Since $\MSpinh$ is an $\MSpin$-module in the category of spectra, $\pi_*\MSpinh$ is a $\pi_*\MSpin$-module in the category of rings. One can ask to characterize this module structure explicitly.

\begin{prob}\label{prob:module structure}
Calculate the module structure of $\pi_*\MSpinh$ over the ring $\pi_*\MSpin$.
\end{prob}

Problem~\ref{prob:module structure} should be quite difficult, as even the ring structure of $\pi_*\MSpin$ is not completely understood \cite{Lau03}. However, the ring structure of $\pi_*\MSpin$ is known modulo torsion \cite{Sto66} (see also \cite[Theorem 2.8]{ABP67}). This suggests a suitable weakening of Problem~\ref{prob:module structure}.

\begin{prob}
    Determine the structure of $\pi_*\MSpinh/\text{torsion}$ as a module over the ring $\pi_*\MSpin/\text{torsion}$.
\end{prob}

\subsection{Calculating Pin$^h$ bordism groups}
Shortly after proving the 2-local splitting of $\MSpin$, Anderson--Brown--Peterson computed the additive structure of $\Omega^{\Pin^-}_*$ using the isomorphism $\Omega^{\Pin^-}_n\cong\tilde{\Omega}^\Spin_{n+1}(\mb{RP}^\infty)$ \cite{ABP69}. The additive structure of $\Omega^{\Pin^+}_*$ was computed by Kirby and Taylor \cite{KT90}.

The quaternionic pin groups $\Pin^{h\pm}:=\Pin^{\pm}\times_{\{\pm 1\}}\Sp(1)$ were introduced by Freed and Hopkins under the notation $G^{\pm}$ \cite[Proposition 9.16]{FH21}. Using Theorem~\ref{thm:spinh splitting} as a starting point, computing the additive structure of $\Omega^{\Pin^{h\pm}}_*$ might be an accessible problem.

\begin{prob}
    Compute the additive structure of $\Omega^{\Pin^{h\pm}}_*$.
\end{prob}

For the $\Pin^{h-}$ case, one can try to construct a Smith isomorphism connecting $\Spinh$ and $\Pin^{h-}$ cobordism.

\begin{ques}\label{ques:spinh smith}
    Is there is an isomorphism $\Omega^{\Pin^{h-}}_n\cong\tilde{\Omega}^{\Spinh}_{n+1}(\mb{HP}^\infty)$ for each $n$?
\end{ques}

\begin{rem}
    A natural candidate for the morphism $\sigma:\tilde{\Omega}^{\Spinh}_{n+1}(\mb{HP}^\infty)\to\Omega^{\Pin^{h-}}_n$ is as follows. Let $M$ be a manifold representing a class in $\tilde{\Omega}^{\Spinh}_{n+1}(\mb{HP}^\infty)$. Then there exists some $k\gg 0$ and a classifying map $f:M\to\mb{HP}^k$. Moreover, we can take $f$ to be transverse to $\mb{HP}^{k-1}\subset\mb{HP}^k$. Set $\sigma(M):=f^{-1}(\mb{HP}^{k-1})$.

    The candidate manifold $\sigma(M)$ is constructed in the same manner as Bahri--Gilkey's Smith isomorphism for $\Spinc$ and $\Pin^{c-}$ cobordism \cite[Lemma 3.1(a)]{BG87}. In the $\Spinc$ setting, checking that $\sigma(M)$ is a $\Pin^{c-}$ manifold is a single characteristic class computation. We do not have an analogous result for determining the existence of $\Pin^{h-}$ structure, so new ideas are needed to continue this approach.
\end{rem}

\begin{rem}
    Question~\ref{ques:spinh smith} has been answered (in a corrected form) by Debray and Krulewski \cite{DK25}.
\end{rem}

\subsection{Conner--Floyd surjection}
One important application of the Anderson--Brown--Peterson splitting of $\MSpin$ and $\MSpinc$ is in the work of Hopkins and Hovey \cite{HH92}, who proved that $\MSpin_*(-)$ and $\MSpinc_*(-)$ satisfy Conner--Floyd isomorphisms with respect to $\KO_*(-)$ and $\KU_*(-)$.

\begin{thm}[Hopkins--Hovey]
    The Atiyah--Bott--Shapiro orientations $\vphi^r:\MSpin\to\KO$ and $\vphi^c:\MSpinc\to\KU$ induce maps
    \begin{align*}
        \MSpin_*(X)\otimes_{\MSpin_*}\KO_*&\to\KO_*(X),\\
        \MSpinc_*(X)\otimes_{\MSpinc_*}\KU_*&\to\KU_*(X)
    \end{align*}
    that are natural isomorphisms of $\KO_*$- and $\KU_*$-modules, respectively, for all spectra $X$.
\end{thm}

It is natural to wonder whether an analog holds for $\MSpinh_*(-)$ with respect to $\KSp_*(-)$. One obvious wrinkle is that $\MSpinh_*$ is not itself a ring, but rather a module over $\MSpin_*$. It turns out that we get a Conner--Floyd surjection, but not an isomorphism \cite[Theorem 6.1.1]{Hu23}:

\begin{thm}[Hu]
The Atiyah--Bott--Shapiro map $\vphi^h:\MSpinh\to\KSp$ induces a surjection
\[\MSpinh_*(X)\otimes_{\MSpin_*}\KO_*\to\KSp_*(X)\]
for all spectra $X$. Moreover, this surjection admits a canonical splitting that is natural in $X$.
\end{thm}

Since the splitting of $\MSpin_*(X)\otimes_{\MSpin_*}\KO_*\to\KSp_*(X)$ is natural in $X$, one might hope to characterize the kernel in terms of $X$.

\begin{prob}
    Characterize the kernel of $\MSpinh_*(X)\otimes_{\MSpin_*}\KO_*\to\KSp_*(X)$.
\end{prob}

\bibliography{spinh}{}
\bibliographystyle{alpha}
\end{document}